\pdfoutput=1

\documentclass[paper=a4, USenglish, numbers=noenddot, DIV=12]{scrartcl} 
\usepackage{cmap}		
\usepackage[utf8]{inputenc}	
\usepackage[T1]{fontenc}	
\usepackage[USenglish]{babel}
\usepackage[]{microtype}

\usepackage{subcaption}
\usepackage{amsmath}
\usepackage{amssymb}
\usepackage{enumerate}
\usepackage{mathtools}		

\usepackage{aliascnt}		
\usepackage{amsthm}

\usepackage{tikz}			
\usetikzlibrary{matrix,arrows,patterns,intersections,calc,decorations.pathmorphing,decorations.markings}

\usepackage[pdftitle={Rigidity of the saddle connection complex},pdfauthor={Valentina Disarlo, Anja Randecker, Robert Tang},pdfborder={0 0 0}]{hyperref}   
\usepackage[figure]{hypcap}		

\usepackage{array}
\usepackage{adjustbox}
\usepackage{longtable}

\tikzset{middlearrow/.style={
        decoration={markings,
            mark= at position 0.5 with {\arrow{#1}} ,
        },
        postaction={decorate}
    }
}

\newtheoremstyle{break}
 {} 
 {} 
 {\itshape} 
 {} 
 {\bfseries} 
 {} 
 {\newline} 
 {\thmname{#1}\thmnumber{ #2}\thmnote{ (#3)}} 
 
\newtheoremstyle{breakdef}
 {} 
 {} 
 {} 
 {} 
 {\bfseries} 
 {} 
 {\newline} 
 {\thmname{#1}\thmnumber{ #2}\thmnote{ (#3)}} 
 
\newtheoremstyle{remark}
 {} 
 {} 
 {} 
 {} 
 {\itshape} 
 {.} 
 {0.5em} 
 {\thmname{#1}\thmnumber{ #2}\thmnote{ {\normalfont (}#3{\normalfont )}}} 

\theoremstyle{breakdef}
\newtheorem{definition}{Definition}[section]

\theoremstyle{remark}

\newaliascnt{rem}{definition}  
\newtheorem{rem}[rem]{Remark}
\aliascntresetthe{rem}

\newaliascnt{exa}{definition}  
\newtheorem{exa}[exa]{Example}
\aliascntresetthe{exa} 

\newaliascnt{lem}{definition}  
\newtheorem{lem}[lem]{Lemma}
\aliascntresetthe{lem}

\theoremstyle{break}

\newtheorem{thm}{Theorem}
\newaliascnt{prop}{definition}  
\newtheorem{prop}[prop]{Proposition}
\aliascntresetthe{prop}

\newaliascnt{cor}{definition}  
\newtheorem{cor}[cor]{Corollary}
\aliascntresetthe{cor}

\newcommand{\fakeenv}{} 

\newcounter{dummycounter}
\renewcommand{\thedummycounter}{\kern-0em}

\newenvironment{restate}[1]  
{
	\renewcommand{\fakeenv}{#1} 
	\newtheorem{\fakeenv}[dummycounter]{\cref{#1}} 
	\begin{\fakeenv}
}
{
	\end{\fakeenv}
}

\numberwithin{figure}{section}			

\renewcommand{\Re}{\mathrm{Re}}
\renewcommand{\Im}{\mathrm{Im}}

\newcommand{\G}{\mathcal{G}}
\newcommand{\K}{\mathcal{K}}
\newcommand{\A}{\mathcal{A}}
\newcommand{\C}{\mathcal{C}}

\newcommand{\F}{\mathcal{F}}

\newcommand{\J}{\mathcal{J}}
\newcommand{\vis}{\hat\J}

\newcommand{\TV}{\overrightarrow{T}}
\newcommand{\QD}{\mathcal{QD}}

\DeclareMathOperator{\SL}{SL}
\DeclareMathOperator{\GL}{GL}
\DeclareMathOperator{\Aff}{Aff}
\DeclareMathOperator{\cyl}{cyl}
\DeclareMathOperator{\Mod}{Mod}

\DeclareMathOperator{\lk}{lk}
\DeclareMathOperator{\IL}{IL}
\DeclareMathOperator{\MIL}{MIL}
\newcommand{\T}{\mathcal{T}}
\DeclareMathOperator{\area}{area}
\DeclareMathOperator{\height}{height}
\DeclareMathOperator{\width}{width}
\DeclareMathOperator{\length}{length}
\DeclareMathOperator{\FP}{FP}
\DeclareMathOperator{\B}{B}
\DeclareMathOperator{\FB}{FB}
\DeclareMathOperator{\Ft}{F_\tau}
\DeclareMathOperator{\KB}{KB}
\DeclareMathOperator{\KFB}{KFB}
\renewcommand{\t}{\mathbf{t}}
\renewcommand{\i}{\mathbf{i}}
\DeclareMathOperator{\Aut}{Aut}

\DeclareMathOperator{\codim}{codim}
\DeclareMathOperator{\id}{id}
\DeclareMathOperator{\Homeo}{Homeo}

\newcommand{\PG}{\mathbf{P}}
\newcommand{\CG}{\mathbf{C}}
\renewcommand{\NG}{\mathbf{N}}

\newcommand{\R}{\mathbb{R}}

\newcommand{\CC}{\mathbb{C}}
\newcommand{\RP}{\mathbb{RP}}
\renewcommand{\P}{\mathcal{Z}}

\newcommand{\imagi}{\mathrm{i}}

\usepackage[nameinlink]{cleveref}
\crefname{rem}{Remark}{Remarks}
\crefname{definition}{Definition}{Definitions}
\crefname{exa}{Example}{Examples}
\crefname{lem}{Lemma}{Lemmas}
\crefname{thm}{Theorem}{Theorems}
\crefname{prop}{Proposition}{Propositions}
\crefname{cor}{Corollary}{Corollaries}
\crefname{section}{Section}{Sections}
\crefname{subsection}{Subsection}{Subsections}
\crefname{figure}{Figure}{Figures}

\author{
Valentina Disarlo
\thanks{Mathematisches Institut, Heidelberg Universit\"at, Im Neuenheimer Feld 205, 69120 Heidelberg, Germany, 
{\href{mailto:vdisarlo@mathi.uni-heidelberg.de}{\texttt{vdisarlo@mathi.uni-heidelberg.de}}}}
\and Anja Randecker
\thanks{Department of Mathematics, University of Toronto, 40 St.~George St., Toronto, ON, M5S 2E4, Canada,
{\href{mailto:randecker@mathi.uni-heidelberg.de}{\texttt{randecker@mathi.uni-heidelberg.de}}}}
\and Robert Tang
\thanks{Department of Pure Mathematics, Xi'an Jiaotong--Liverpool University, 111 Ren'ai Road, Suzhou Industrial Park, Suzhou, Jiangsu, 215123, China,
{\href{mailto:robert.tang@xjtlu.edu.cn}{\texttt{robert.tang@xjtlu.edu.cn}}}}
}
\title{Rigidity of the saddle connection complex}
\date{March 18, 2022}

\begin{document}
 
\maketitle


\begin{abstract}
 For a half-translation surface $(S,q)$, the associated \emph{saddle connection complex}
 $\A(S,q)$ is the simplicial complex where vertices are the saddle connections
 on~$(S,q)$, with simplices spanned by sets of pairwise disjoint saddle connections.
 This complex can be naturally regarded as an induced subcomplex of the arc complex.
 We prove that any simplicial isomorphism $\phi \colon \A(S,q) \to \A(S',q')$
 between saddle connection complexes is induced by an affine diffeomorphism
 $F \colon (S,q) \to (S',q')$. In particular, this shows that the saddle connection complex
 is a complete invariant of affine equivalence classes of half-translation surfaces.
 Throughout our proof, we develop several combinatorial criteria of independent interest for detecting various geometric objects on a half-translation surface.
\end{abstract}

\section{Introduction}

For a closed topological surface $S$ of finite type with a non-empty finite set $\P$
of marked points, the associated \emph{arc complex} $\A(S,\P)$ has as vertices the essential arcs on $(S,\P)$ (considered up to proper homotopy),
with simplices spanned by sets of arcs that can be realised pairwise disjointly.
The \emph{(extended) mapping class group} $\Mod(S,\P)$ acts naturally on $\A(S,\P)$
by simplicial isomorphisms. In \cite{irmak_mccarthy10}, Irmak and McCarthy proved a combinatorial rigidity theorem for arc complexes:
any simplicial isomorphism $\phi : \A(S,\P) \to \A(S', \P')$ between
arc complexes is induced by a homeomorphism ${F: (S,\P) \to (S',\P')}$
between the associated surfaces.
In particular, this implies that, except for finitely many cases, the automorphism group of $\A(S,\P)$ is isomorphic to $\Mod(S,\P)$. 

In this paper, we explore an analogous combinatorial rigidity phenomenon
for a complex associated to \emph{half-translation surfaces}.
Let $(S,q)$ be a half-translation surface, with $\P$ taken to
be the set of singularities.

The \emph{saddle connection complex} $\A(S,q)$ is the simplicial complex 
where vertices are saddle connections on~$(S,q)$, with simplices spanned by
sets of pairwise disjoint saddle connections. In particular,
$\A(S,q)$ can be regarded as the induced subcomplex of $\A(S,\P)$ spanned by
the arcs that can be realised as saddle connections on $(S,q)$.
Any affine diffeomorphism between half-translation
surfaces naturally induces a simplicial isomorphism between the respective
saddle connection complexes.
In particular, the \emph{affine diffeomorphism group} $\Aff(S,q)$ acts naturally
on $\A(S,q)$ by simplicial automorphisms.

The main result of this paper is the converse, that is, any simplicial isomorphism
between saddle connection complexes arises from an affine diffeomorphism.
This gives a rigidity theorem analogous to that of Irmak--McCarthy for the
arc complex.

\begin{thm}[Rigidity of the saddle connection complex] \label{thm:main}
 Let $(S,q)$ and $(S',q')$ be half-translation surfaces, neither of which are flat tori with exactly one removable singularity. Suppose $\phi : \A(S,q) \to \A(S',q')$ is a simplicial isomorphism.
 Then there exists a unique affine diffeomorphism $F: (S,q) \to (S',q')$ inducing $\phi$.
\end{thm}

In the case where at least one of $(S,q)$ or $(S',q')$ is a flat torus with exactly one removable singularity, and $\phi : \A(S,q) \to \A(S',q')$ is a simplicial isomorphism, there are exactly two affine diffeomorphisms $F_1, F_2: (S,q) \to (S',q')$ inducing $\phi$. In particular, both $(S,q)$ and $(S',q')$ are flat tori, and the maps $F_1$ and $F_2$ differ by a hyperelliptic involution.

Consequently, the natural map $\Aff(S,q) \to \Aut(\A(S,q))$ is an isomorphism (unless $(S,q)$ is a flat torus with one removable singularity, in which case the kernel is the order--$2$ subgroup generated by the hyperelliptic involution).
In particular, if $[F] \in \Mod(S,\P)$ is a mapping class that preserves the subcomplex $\A(S,q) \subset \A(S,\P)$ (setwise), then $[F]$ has an affine representative $F : (S,q) \to (S,q)$. 

\cref{thm:main} shows that the combinatorial structure of the saddle connection complex
completely governs the half-translation structure on the surface up to
affine equivalence. In particular, the saddle connection complex is a complete
invariant of $\SL(2,\R)$--orbits in the moduli space of quadratic differentials,
taken over all surfaces.
Consequently, there are uncountably many isomorphism classes of saddle connection
complexes.
Finding explicit relations between the combinatorial
properties of saddle connection complexes and the dynamical aspects of 
half-translation surfaces and their $\SL(2,\R)$--orbit closures
would provide an interesting direction for future research.

\paragraph{Overview of the paper}
In \cref{sec:arcs}, we review some results
about arc complexes and flip graphs.
While not directly necessary for proving our main theorem,
we shall give proofs of some standard properties concerning arc complexes to
provide context for the analogous results for
the saddle connection complex. These also serve
to point out the differences between the topological
and the Euclidean settings.
In \cref{sec:flat}, we give some basic definitions
and properties of half-translation surfaces
and saddle connection complexes.
We then prove some elementary results regarding links
in \cref{sec:links}, followed by a detailed examination
of simplices with infinite links in \cref{sec:cylinder}.
In Sections \ref{sec:flow}, \ref{sec:extend}, and \ref{sec:flip_pairs_bad_kites}, 
we develop some technical results in order to establish the 
key combinatorial criteria: the Triangle Test (\cref{sec:triangle})
and the Orientation Test (\cref{sec:orient}).
Finally, we give a proof of \cref{thm:main} in \cref{sec:rigid}
using the following~strategy.

\begin{enumerate}
 \item \textbf{Triangle Test (\cref{tri_test}):}
 There is a combinatorial criterion that can detect 
 the simplices in $\A(S,q)$ that bound triangles on $(S,q)$.
 \item \textbf{Orientation Test (\cref{prop:orient}):} 
 There is a combinatorial criterion that can detect whether two given oriented triangles
 on $(S,q)$ are consistently oriented.
 \item Using the previous two tests, we recover the gluing pattern of a triangulation
 and hence the underlying surface $(S,\P)$. By construction, we show that any
 simplicial isomorphism $\phi : \A(S,q) \to \A(S',q')$
 is induced by a piecewise affine diffeomorphism $F : (S,q) \to (S',q')$.
 \item \textbf{Cylinder Test (\cref{MIL_cyl}):}
 There is a combinatorial criterion that can detect the set of cylinder curves on $(S,q)$.
 \item Finally, we apply the Cylinder Rigidity Theorem (\cref{thm:cyl_rigid})
 of Duchin, Leininger, and Rafi \cite{duchin_leininger_rafi_10} to deduce that
 $F$ is an affine diffeomorphism.
\end{enumerate}

Throughout this paper, we develop several other combinatorial criteria to detect various features of a half-translation surface via its saddle connection complex.
We expect these tools to be of independent interest, with potential applications to future work.

As readers might have background in combinatorial complexes or in translation surfaces but not both, we give some basics for both topics. These parts can be skipped by experts in the respective topics.
A list of notation is given in \Cref{app:notation}.

\paragraph{Combinatorial rigidity results in the topological setting}
The combinatorics of simplicial complexes associated to a topological surface play an
important role in several groundbreaking papers on the algebraic properties of
the mapping class group in the~1980's. For instance, the arc complex was defined
by Harer in his study of the cohomology of the mapping class group
(see \cite{harer_annals} and \cite{harer_inventiones}); the Hatcher--Thurston complex
was used to find the first explicit presentation of the mapping class group
\cite{hatcher_thurston1980}; the curve complex was introduced by Harvey
in order to study the action of the mapping class group
on the boundary of Teichm\"uller space~\cite{harvey_1978}.

Rigidity phenomena for combinatorial complexes associated to surfaces were first
studied by Ivanov in the case of the curve complex \cite{ivanov_97, ivanov_02}.
Ivanov showed that, when the surface has genus at least $2$, any automorphism
of the curve complex is induced by a homeomorphism of the surface.
He applied this result to study the isometry group of Teichm\"uller space,
giving a new proof of a celebrated result of Royden \cite{royden_71}.
Korkmaz then extended Ivanov's result to the remaining low-genus cases \cite{korkmaz}. 
In \cite{Luo}, Luo gave a new proof of the curve complex rigidity theorem,
showing also that the isomorphism type of the curve complex is a complete
invariant of the topological type of the surface. 

Subsequently, \emph{Ivanov-style rigidity} has been established for other
complexes such as the arc complex (Irmak--McCarthy \cite{irmak_mccarthy10},
Disarlo \cite{disarlo_15}), 
the pants complex (Margalit \cite{margalit}), the Hatcher--Thurston complex
(Irmak--Korkmaz \cite{irmak_korkmaz}), the flip graph
(Korkmaz--Papadopoulos \cite{korkmaz_papadopoulos}), the polygonalisation complex (Bell--Disarlo--Tang \cite{bell_disarlo_tang_18}),
and many other complexes built from topological objects on surfaces
(for a survey, see McCarthy--Papadopoulos \cite{pap_mccarthy}). 
In response to the multitude of rigidity results, Ivanov \cite{ivanov06}
formulated a \emph{metaconjecture}, stating that every object naturally associated 
to a surface and having a sufficiently rich structure has the mapping class group
as its group of automorphisms. At present, Ivanov's metaconjecture
remains unresolved, however, Brendle and Margalit in \cite{brendle_margalit} have
recently verified it for a large class of complexes
by proving general results related to the complex of domains, introduced by McCarthy and Papadopoulos \cite{pap_mccarthy}.

\paragraph{Relations to other work}
In light of Ivanov's metaconjecture,
\cref{thm:main} demonstrates that the saddle connection complex
also enjoys combinatorial rigidity properties. In particular,
its automorphism group is the affine diffeomorphism group of 
the half-translation structure on the surface. To the best of our knowledge, this is the first combinatorial complex known to exhibit Ivanov-style rigidity for a \emph{geometric structure} on surfaces. 
(Here, we are using only the isomorphism type of $\A(S,q)$, with no a priori
information regarding the underlying surface~$(S,\P)$ or arc complex $\A(S,\P)$.)
It would be interesting to see if other combinatorial complexes associated
to half-translation structures also exhibit analogous rigidity properties.

An independent proof of \cref{thm:main} in the special case of translation surfaces was given simultaneously by Huiping Pan \cite{pan_22}.
The overall strategy of Pan's proof is similar to ours.
He gives a direct proof that any simplicial isomorphism $\phi : \A(S,q) \to \A(S',q')$ must send triangles on $(S,q)$ to triangles on $(S',q')$,
and that consistently oriented triangles remain consistently oriented.
This shows that $\phi$ is induced by a piecewise affine homeomorphism.
He then applies a standard argument of Duchin, Leininger, and Rafi
\cite{duchin_leininger_rafi_10} (see also \cite{nguyen_15}) to prove that
the resulting homeomorphism is affine.

There have been several other recent works on combinatorial complexes
associated to (half-) translation surfaces.
Minsky and Taylor define the saddle connection complex in \cite{minsky_taylor_17},
as well as a subcomplex of the arc complex defined using Veering triangulations,
and relate their geometry to that of an associated fibred hyperbolic 3--manifold.
Nguyen defines a graph of (degenerate) cylinders for genus $2$ translation surfaces
as a subgraph of the curve complex, and proves that any mapping class that
stabilises it must have an affine representative \cite{nguyen_15}.
He also defines a graph of periodic directions in \cite{nguyen_18},
and uses this to give an algorithm to compute a coarse fundamental domain
for the associated Veech group.
In \cite{tang_webb_18}, Tang and Webb consider the multi-arc graph and
filling multi-arc graph associated to a half-translation surface, and
use these to obtain a bounded geodesic image theorem for
a ``straightening operation'' (see \cref{sec:sccomplex}) defined on the curve complex.

Another rigidity result for singular Euclidean surfaces was recently established by
Duchin, Erlandsson, Leininger, and Sadanand in \cite{duchin_erlandsson_leininger_sadanand_21}.
They show that the ``bounce spectrum'' associated to a general polygonal billiard
table determines the shape of the table up to affine equivalence
(for right-angled tables) or up to similarity.
Their techniques deal with flat surfaces with arbitrary cone angles, and so it would be interesting to see whether these methods can be adapted to obtain combinatorial rigidity results for complexes associated to this broader class of geometric structures.

\paragraph{Acknowledgements}
The first and third author are grateful to Mark Bell for many helpful discussions
during the Redbud Topology Conference at the University of Oklahoma in April~2016,
which led to the formulation of the main question of this paper, as well as 
the classification of links in \cref{sec:classify}.
The first substantial progress on this project was carried out
while the authors were in-residence during the
Fall 2016 semester program on ``Geometric group theory'' at 
the Mathematical Sciences Research Institute, Berkeley. The authors were supported by the National Science Foundation under Grant No. DMS
1440140, administered by the Mathematical Sciences Research Institute, during their research stay.
Subsequent work was carried out at the Shanks Conference on
``Low-dimensional topology and geometry'' at Vanderbilt University in May 2017,
at Heidelberg University in December 2017,
and at the ``Geometry of Teichm\"uller space and mapping class groups'' workshop
at the University of Warwick in April 2018. We thank these venues for their kind hospitality. 
The authors acknowledge support from U.S. National Science Foundation grants
DMS 1107452, 1107263, 1107367 "RNMS: GEometric structures And Representation
varieties" (the GEAR Network). The first author acknowledges support from the Olympia Morata Habilitation Programme of Universit\"at Heidelberg. She is currently funded by the Priority Program 2026 ``Geometry at Infinity'' of the Deutsche Forschungsgemeinschaft (DFG, German Research Foundation) DI 2610/2-
1. She also acknowledges support by the European Research Council under Anna Wienhard's ERC-Consolidator grant 614733 (GEOMETRICSTRUCTURES) and by the Deutsche
Forschungsgemeinschaft (DFG, German Research Foundation) under Germany's Excellence Strategy EXC-2181/1 - 390900948 (the Heidelberg STRUCTURES Cluster of Excellence).
The second author acknowledges support from a DFG Forschungsstipendium
and from the National Science Foundation under Grant No. DMS 1607512.
The third author acknowledges support from a JSPS KAKENHI Grant-in-Aid for Early-Career Scientists (No.~19K14541).
The authors also thank Sam Taylor for helpful comments, and Kasra Rafi, Chris Judge,
and Matt Bainbridge for helpful conversations and encouragement. Finally, the authors thank the anonymous referee for their thorough comments and suggestions for improving the exposition of the paper.

\section{Surfaces and arcs}\label{sec:arcs}

In this section, we review some background regarding topological surfaces
and the arc complex, and also give proofs of some standard results.
While they are not strictly necessary for the proof of \cref{thm:main},
these results will provide context for the analogous results
for the saddle connection complex.

Let $(S, \P)$ denote a compact surface $S$, possibly with boundary,
equipped with a finite, non-empty set $\P \subset S$ of marked points
such that each boundary component contains at least one marked point.
A surface homeomorphic to a closed disc with $n$ marked points on its boundary and having no interior marked points is called an \emph{$n$--gon}, or simply a \emph{polygon}.
In particular, we call an $n$--gon a monogon, bigon, triangle, or quadrilateral
in the cases when $n=1,2,3,4$ respectively.
A surface homeomorphic to a closed annulus with $m$ marked points on one boundary
component, $n$ marked points on the other boundary component, and with no interior marked points
is called an \emph{$(m,n)$--annulus}.

\subsection{Arc complexes}

A simplicial complex is called \emph{flag} if for every complete subgraph in the $1$--skeleton, there exists a simplex in the complex which has the subgraph as its $1$--skeleton.
All simplicial complexes considered in this paper shall be flag,
and so we may instead work with their $1$--skeleta whenever convenient.
Distances between vertices of the complex shall always be measured using the
combinatorial metric in the $1$--skeleton.
Given a simplicial complex $\K$,
we shall write $\sigma \in \K$ to refer to both
a simplex of $\K$, and the vertex set of that simplex.
We write $\# \K$ for the number of vertices of $\K$.
The \emph{link} $\lk_\K(\sigma)$ of a simplex $\sigma \in \K$
is the subcomplex of $\K$ consisting
of all simplices $\sigma' \in \K$ such that
$\sigma \cap \sigma' = \emptyset$ and $\sigma \cup \sigma'$ is a simplex.

\begin{definition}[Essential arc]
 An \emph{arc} on $(S,\P)$ is a map $\alpha \colon [0,1] \to S$ such that
 $\{\alpha(0), \alpha(1)\} \subseteq \P$ and the restriction 
 $\alpha|_{(0,1)}$ is an embedding into $S \setminus \P$.
 An arc whose image lies in $\partial S$ is called a \emph{boundary arc}.
 An arc is called \emph{essential} if it is not homotopic (relative to $\P$) to a point or a boundary arc.
\end{definition}

All arcs considered in this paper will be essential unless stated otherwise.
We usually consider arcs up to reversal of orientation and proper isotopy.
Furthermore, we shall usually consider arcs up to \emph{proper homotopy}, that
is, homotopy relative to $\P$. 
Two arcs $\alpha,\beta$ are said to be \emph{disjoint} if they have disjoint interiors;
otherwise we say they \emph{intersect} or \emph{cross} and write~$\alpha \pitchfork \beta$. 

A \emph{multi-arc} on $(S,\P)$ is a non-empty set of pairwise disjoint and non-homotopic arcs (see \cref{fig:multiarc} for an example),
and a \emph{triangulation} of $(S,\P)$ is a maximal multi-arc. Indeed, cutting~$S$ along a maximal multi-arc as described in the next section decomposes it into a finite union of triangles.

\begin{figure}
 \begin{center}
\begin{tikzpicture}[scale=2]
\tikzset{dot/.style={draw,shape=circle,fill=black,scale=0.3}}

\draw [red, dotted] (0.4,0.37) to [out=180+45,in=180-45] (0.4,-0.37);
\draw [red, dotted] (1.3,-0.1) to [out=-60,in=180] (2,-0.1) to [out=0,in=240] (2.7,-0.1);

\draw [dotted, thick] (4,-0.3) to [out=180,in=180,looseness=0.7]
	(4,0.3);

\draw
	(0,0.3) to [out=0,in=180]
	(1,0.5) to [out=0,in=180]
	(2,0.5-0.1) to [out=0,in=180]
	(2,0.5-0.1) to [out=0,in=180]
	(3,0.5) to [out=0,in=180]
	(4,0.3);
\draw (4,-0.3) to [out=180,in=0]
	(3,-0.5) to [out=180,in=0]
	(2,-0.5+0.1) to [out=180,in=0]
	(1,-0.5) to [out=180,in=0]
	(0,-0.3);

\draw [thick] (0,0.3) to [out=0,in=0,looseness=0.7]
	(0,-0.3) to [out=180,in=180,looseness=0.7]
	(0,0.3);

\draw [thick] (4,0.3) to [out=0,in=0,looseness=0.7]
	(4,-0.3) ;

\draw [red] (0.125,0) to [out=0,in=0] (0.4,0.37);
\draw [red] (0.125,0) to [out=0,in=0] (0.4,-0.37);

\draw [red] (0.125,0) to [out=0,in=180] (1,0.3);

\draw [red] (1,0.3) to [out=0,in=180] (2,0);
\draw [red] (2,0) to [out=180,in=60] (1.3,-0.1);
\draw [red] (2,0) to [out=0,in=120] (2.7,-0.1);

\draw [red] (2,0) to [out=240,in=180] (2,-0.3) to [out=0,in=300] (2,0);

\draw [red] (3,0.3) to [out=0,in=90] (3.6,0) to [out=270,in=0] (3,-0.3);

\draw [red] (4.1,0.25) to [out=180,in=180] (4.1,-0.25);

\draw (0.5,0) to [out=-30,in=180+30] (1.5,0);
\draw (0.7,-0.1) to [out=30,in=180-30] (1.3,-0.1);
\draw (2.5,0) to [out=-30,in=180+30] (3.5,0);
\draw (2.7,-0.1) to [out=30,in=180-30] (3.3,-0.1);

\node [dot] at (0.125,0) {};
\node [dot] at (-0.125,0) {};
\node [dot] at (1,0.3) {};
\node [dot] at (2,0) {};
\node [dot] at (2.3,0.3) {};
\node [dot] at (2,-0.2) {};
\node [dot] at (3,0.3) {};
\node [dot] at (3,-0.3) {};

\node [dot] at (4.07,0.25) {};
\node [dot] at (4.125,0) {};
\node [dot] at (4.07,-0.25) {};

\end{tikzpicture}
  \caption{A multi-arc on a surface.}
  \label{fig:multiarc}
\end{center}
\end{figure}
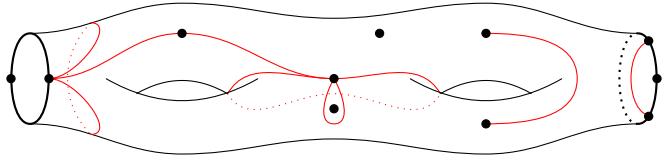

\begin{definition}[Arc complex]
 The \emph{arc complex} $\A(S,\P)$ is the flag simplicial complex whose vertices are the essential arcs on $(S,\P)$ (considered up to homotopy), with simplices corresponding to multi-arcs on $(S,\P)$.
\end{definition}

Every simplex in $\A(S,\P)$ can be extended to a maximal one,
and every maximal simplex in $\A(S,\P)$ corresponds to a triangulation of $(S,\P)$.
Moreover, $\A(S,\P)$ is finite dimensional and every maximal simplex
has the same dimension.
In the case where $S$ is closed, connected, orientable, and has genus $g \geq 2$
then $\A(S,\P)$ is connected, locally infinite, has
dimension $\kappa(S,\P) = 6g + 3|\P| - 7$, and has infinite diameter. Masur and Schleimer \cite{masur_schleimer_13} proved that $\A(S,\P)$ is Gromov hyperbolic, with constants depending on the topology of $(S,\P)$; 
a new proof yielding a hyperbolicity constant independent of the topology was later given in \cite{hensel_webb_15}.

The results of \cite{irmak_mccarthy10}, \cite{aramayona_koberda_parlier_15}, and \cite{disarlo_15} can be combined (as done in \cite[Theorem 2.2]{bell_disarlo_tang_18}) to deduce the following:
\begin{thm}[Rigidity of arc complexes] \label{thm:arc_rigid}
 Assume $(S,\P)$ and $(S',\P')$ are surfaces such that the associated arc complexes are non-empty.
 Suppose ${\phi :\A(S,\P) \to \A(S',\P')}$ is a simplicial isomorphism.
 Then there exists a homeomorphism ${F : (S,\P) \to (S',\P')}$ inducing $\phi$.
 $\hfill\square$
\end{thm}

Here the homeomorphism must act bijectively on the sets of marked points, and
is allowed to reverse orientation.

\begin{definition}[Mapping class group $\Mod(S,\P)$]
 Let $\Homeo(S,\P)$ denote the group of homeomorphisms of $S$ that fix $\P$ setwise. 
 The \emph{extended mapping class group} $\Mod(S,\P)$ is the group of the isotopy classes of elements of $\Homeo(S,\P)$, where isotopies are required to fix $\P$ pointwise. 
 An element of $\Mod(S,\P)$ is called a \emph{mapping class}.
 Mapping classes are allowed to reverse the orientation of $S$.
\end{definition}

Since mapping classes preserve the property of disjointness for arcs, it follows that
$\Mod(S,\P)$ acts naturally on $\A(S,\P)$ by simplicial automorphisms.

In the remainder of this article, we restrict to the case where $S$ is a closed surface.
\begin{thm}[Automorphisms of the arc complex \cite{irmak_mccarthy10}]
 Assume $(S,\P)$ is not a sphere with at most 3 marked points, nor a torus with one marked point. 
 Then the natural map $\Mod(S,\P) \to \Aut(\A(S,\P))$ is an isomorphism. $\hfill\square$
\end{thm}

\subsubsection{Links, cutting, and gluing}

Let $\sigma \in \A(S,\P)$ be a multi-arc. 
Then the \emph{link} of $\sigma$ is the induced subcomplex $\lk_{\A(S,\P)}(\sigma)$ of~$\A(S,\P)$ with vertex set \[\{\alpha \in \A(S,\P) ~|~ \alpha \notin \sigma \textrm{ and } \alpha \textrm{ is disjoint from all arcs in }\sigma \}.\]
Define the surface $(S - \sigma, \P)$ obtained by \emph{cutting $S$ along the multi-arc $\sigma \in \A(S,\P)$} as follows.
Equip $S$ with a Riemannian metric, and realise $\sigma$ as a set of pairwise disjoint smooth arcs. For each complementary component $R'_i$ of $S \backslash \sigma$, let $R_i$ be its metric completion with respect to the induced metric. The inclusion $int(R_i) = R'_i  \hookrightarrow S$ extends to a continuous map $\iota_i : R_i \rightarrow S$. Then $\P_i = \iota_i^{-1}(\P)$ is a non-empty set of marked points on $R_i$. In particular, $(R_i, \P_i)$ is a surface with boundary, where each boundary component has at least one marked point.
We then define $(S - \sigma, \P)$ to be the (possibly disconnected) surface $\bigsqcup_i (R_i, \P_i)$.
We call each $(R_i, \P_i)$ a \emph{region} of~$(S - \sigma, \P)$.
When it is clear from context that $\P$ is the set of marked points on a fixed ambient surface $(S, \P)$, then we may simply write  $S - \sigma$ for $(S - \sigma, \P)$ and $R_i$ for $(R_i, \P_i)$.
For brevity, we shall also refer to each $R_i$ as a region of $\sigma$. 
 
The map $\iota : (S - \sigma, \P) \to (S,\P)$ defined by $\iota|_{R_i} = \iota_i$ is surjective, sends marked points to marked points, and restricts to an embedding on the interior of $S - \sigma$.
If $\alpha$ is a boundary arc of~$S - \sigma$ then $\iota(\alpha)$ is either a boundary arc of $S$ or an arc in $\sigma$.
(In fact, there are exactly two boundary arcs on $S - \sigma$ mapping to each arc in $\sigma$ via $\iota$.) If $\alpha$ is an essential arc on $S - \sigma$ then~$\iota(\alpha)$ is an essential arc on $S$ disjoint from $\sigma$.
Conversely, if $\beta$ is an essential arc on $S$ disjoint from $\sigma$ then~$\iota^{-1}(\beta)$ is the disjoint union of an essential arc on $S - \sigma$ with a (possibly empty) set of marked points.
Therefore, $\iota$ preserves disjointness of arcs, and so we deduce the following.
 
\begin{lem}[Links and cutting] \label{link_iso}
 The map $\iota : (S - \sigma, \P) \to (S,\P)$ as above induces a simplicial isomorphism  $\iota_* : \A(S - \sigma, \P) \to \lk_{\A(S,\P)}(\sigma)$. $\hfill \square$
\end{lem}
 
Given a region $R$ of $\sigma$, we shall regard its boundary $\partial R \subseteq \sigma$ as a multi-arc on $S$, and hence a simplex in $\A(S,\P)$. Each boundary arc of $R$ shall also be called a \emph{side} of $R$. If an arc $\alpha \in \sigma$ appears as two sides of $R$ then we shall refer to these sides as the two \emph{copies} of $\alpha$ on $\partial R$.
We also say a region $R$ \emph{meets} an arc $\alpha$ if $\alpha$ is a side of $R$.
Note that there are either one or two regions of $\sigma$ meeting a given arc in $\sigma$; and each arc appears at most twice as a side of any given region. In particular, a triangular region with an arc appearing twice as a side is called a \emph{folded triangle} (see the triangle $T$ in \cref{fig:links_arc_complex}.)

If $R$ and $R'$ are regions of $\sigma$ meeting along an arc $\alpha$ in $\sigma$, then we may obtain a region of~$\sigma \setminus \{\alpha\}$ by \emph{gluing} them along $\alpha$: this is the unique region of $\sigma \setminus \{\alpha\}$ containing the interiors of $R$ and $R'$.
If $\alpha$ appears twice as a side of $R$, then we may glue $R$ to itself along the two copies of $\alpha$ to produce a region of $\sigma \setminus \{\alpha\}$. Observe that marked points are mapped to marked points under any of these gluing operations.
 
Note that the regions obtained by the cutting operation (considered up to proper isotopy) do not depend on the choice of the Riemannian metric or on the representatives of the smooth arcs~in~$\sigma$.

\subsubsection{Join decompositions of links}

Given a finite set of flag simplicial complexes $A_1, \ldots, A_n$,
we define their \emph{join} ${A_1 * \ldots * A_n}$ as follows:
the vertex set is the disjoint union $\sqcup_{i} A_i$ of the
vertex sets of each $A_i$; the simplices are spanned by sets of vertices
of the form
$\sigma_1 \sqcup \ldots \sqcup \sigma_n$, where each $\sigma_i$ 
is a (possibly empty) simplex of $A_i$.

\begin{lem}[Uniqueness of minimal join decomposition]
 Let $\K$ be a finite-dimensional flag simplicial complex.
 Then there exist subcomplexes
 $A_1, \ldots, A_n$ of $\K$, where $n \geq 1$, such that $\K = A_1 * \ldots * A_n$ and each 
 $A_i$ cannot be decomposed as a non-trivial join.
 Moreover, the $A_i$'s are unique up to permutation.
\end{lem}

\begin{proof}
 Since $\K$ is a flag complex, we may instead work with its $1$--skeleton $G$.
 Observe that the graph $G$ cannot be decomposed as a non-trivial join if
 and only if the complement graph $\overline{G}$ (i.e.\ the graph where vertices are adjacent if and only if they are not adjacent in $G$) is connected.
 Since $\K$ is finite-dimensional, there are finitely many
 connected components $H_1, \ldots, H_n$ of $\overline{G}$.
 Let $V_i$ be the vertex set of $H_i$, and $A_i$ be the induced subgraph of $V_i$ in $G$.
 Then $G = A_1 * \ldots * A_n$, and each $A_i$ cannot be decomposed as a non-trivial join.
 Alternatively, we can take $A_i$ to be the induced subcomplex of $V_i$ in $\K$ which
 yields the desired decomposition $\K = A_1 * \ldots * A_n$.
\end{proof}

Note that $S - \sigma$ cannot have any monogons nor bigons as regions since $\sigma$ is a multi-arc which, by definition, consists of essential and pairwise non-homotopic arcs.
A region $R$ of $S - \sigma$ is a triangle if and only if $\A(R) = \emptyset$; and $R$ is a monogon with one interior marked point if and only if $\A(R)$ is a single vertex.
 
\begin{prop}[Decomposition of links in $\mathcal{A}(S,\P)$] \label{decomp-arc}
 Suppose $\sigma \in \mathcal{A}(S,\P)$ is a simplex, and 
 let $\lk_{\mathcal{A}(S,\P)}(\sigma) = A_1 \ast \cdots \ast A_n$
 be the minimal join decomposition of its link. Then
 \begin{enumerate}
  \item $S - \sigma$ has exactly $n$ non-triangular regions $R_1, \ldots, R_n$,
  \item $A_i = \A(R_i)$ (up to permutation), and
  \item the topological type of $R_i$ is determined by the isomorphism type of $A_i$.
 \end{enumerate}

\begin{proof}
Let $R_1, \ldots R_m$ be the non-triangular regions of $S - \sigma$.
Let $B_i = \A(R_i)$ be the induced subcomplex in $\A(S,\P)$ of the arcs contained in~$R_i$.
Any arc $\gamma \in \lk_{\A(S,\P)}(\sigma)$ is contained in exactly one non-triangular region
of $S - \sigma$, and hence is a vertex of precisely one of the $B_i$.
Moreover, if arcs $\beta, \gamma \in \lk_{\A(S,\P)}(\sigma)$ are contained in 
distinct regions then they are disjoint.
Therefore $\lk_{\A(S,\P)}(\sigma) = B_1 \ast \ldots \ast B_m$.
We shall show that each $B_i$ cannot be decomposed as a non-trivial join.
Together with the above lemma, this will yield Parts (i) and (ii).
Part (iii) follows from the combinatorial rigidity of arc complexes (\cref{thm:arc_rigid}).

To simplify the exposition, we shall refer to each $A_i$ as a ``colour''.
In particular, any pair of intersecting arcs in $\lk_{\A(S,\P)}(\sigma)$ have the same colour.
Our goal is to show that all vertices in~$B_i$ have the same colour.
If $B_i$ is a single vertex then we are done, so suppose it has at least two vertices.
Let $\T$ be a triangulation of $(S,\P)$ containing $\sigma$. Note that $R_i$ contains
at least two triangles of $\T$, for otherwise $R_i$ is a once-marked monogon (as $R_i$
is assumed to be non-triangular) and hence contains only one arc. Let $\T_i = \T \cap B_i$
be the arcs of $\T$ contained in~$R_i$.
If we can show that every arc in $\T_i$ has the same colour, then we are done since each~$\gamma \in B_i$ either intersects $\T_i$ or is an arc in $\T_i$. 

\textbf{Claim:} If $B_i$ has at least two vertices, then for every arc $\gamma\in B_i$ there exists an arc~$\gamma' \in B_i$ intersecting $\gamma$.
If $\gamma$ meets two triangles of $\T$, then we let $\gamma'$ be the other diagonal of the quadrilateral formed by the two triangles meeting $\gamma$. It $\gamma$ meets only one triangle $T$ of $\T$, it has to form two sides of the folded triangle $T$ (see \cref{fig:links_arc_complex}). Let $\beta$ be the other side of $T$. Since~$R_i$ contains at least two triangles of~$\T$, there is a triangle $T' \neq T$ of $\T$ also meeting $\beta$. 
Gluing~$T$ to itself along $\gamma$ and to $T'$ along $\beta$ forms a once-marked bigon, with $\gamma$ connecting the interior marked point to a boundary marked point of the bigon. We can take $\gamma'$ to be the essential arc with both endpoints on the other boundary marked point. Therefore, there always exists an arc $\gamma' \in B_i$ that intersects $\gamma$.

\begin{figure}
 \begin{center}
  \begin{tikzpicture}[]  
   \draw (0,0) -- node[below=-0.0]{$\gamma$} (3,0);
   \draw[fill] (0,0) circle (2pt);
   \draw[fill] (3,0) circle (2pt);
   \draw[fill] (6,0) circle (2pt);
   
   \draw (0,0) to[out=30, in=90, looseness=0.8] node[above=-0.15]{$\beta$} (4,0) to[out=-90, in=-30, looseness=0.8] (0,0);
   \draw[red] (6,0) to[out=165, in=90, looseness=0.7] node[above right=-0.1]{$\gamma'$} (2.5,0) to[out=-90, in=-165, looseness=0.7] (6,0);
   
   \draw (0,0) to[out=45, in=135] (6,0);
   \draw (0,0) to[out=-45, in=-135] (6,0);
   
   \draw (4.8,-0.6) node{$T'$};
   \draw (2.2,-0.4) node{$T$};
  \end{tikzpicture}
  \caption{An arc $\gamma$ forming two sides of a folded triangle $T$.}
  \label{fig:links_arc_complex}
 \end{center}
\end{figure}
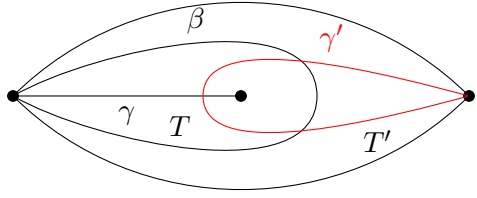

Now suppose arcs $\beta, \gamma \in \T_i$ bound a common triangle $T$ of $\T$.
Applying the above claim, there exist arcs $\beta', \gamma' \in B_i$ intersecting
$\beta$ and $\gamma$ respectively. If $\beta' = \gamma'$, $\beta' \pitchfork \gamma$,
$\beta \pitchfork \gamma'$, or $\beta' \pitchfork \gamma'$ then $\beta$ and $\gamma$
have the same colour. Otherwise, $\beta'$ and $\gamma'$ must both intersect the third
side~$\alpha$ of~$T$, implying that $\alpha$ is an arc in $\T_i$ sharing the same
colour with both $\beta$ and $\gamma$. Since $R_i$ is connected, we deduce that every
arc in $\T_i$ has the same colour as desired.
\end{proof}
\end{prop}
 
\subsubsection{Infinite links}

\begin{definition}[Infinite link simplices]\label{defin:IL}
Given a simplicial complex $\K$, define its set of \emph{infinite link simplices} to be
\[\IL(\K) = \{\sigma \in \K ~|~ \# \lk_\K(\sigma) = \infty\}.\]
Let $\MIL(\K)\subseteq\IL(\K)$ be the set of $\sigma\in\IL(\K)$ for which $\lk(\sigma')$ is finite
for all $\sigma' \supsetneq \sigma$.
In other words, $\MIL(\K)$ comprises the simplices
which are maximal among those with infinite link. 
\end{definition}

A \emph{simple closed curve} (or a \emph{curve}) on $(S,\P)$ is an embedded loop on $S \setminus \P$.
Here, curves are considered up to proper homotopy.
A curve on $(S,\P)$ is \emph{essential} if it does not bound a disc
with at most one marked point. 
Note that an essential curve of $S$ contained in some region $R$ could
be \emph{peripheral} on $R$, that is, parallel to a boundary component of $R$.

We also extend the notion of cutting a surface $S$ along a simple closed curve $\gamma$
to obtain a surface $S - \gamma$. The two boundary components of $S - \gamma$ obtained
by cutting along $\gamma$ will have no marked points.

\begin{rem}[Regions with finitely many arcs]\label{dehn_twist}
 Let $\sigma$ be a multiarc. A region $R$ of $\sigma$ contains an essential simple closed curve if and only if
 $R$ is not a polygon with at most one interior marked point. For regions without simple essential closed curves, the arc complex is
 finite. In fact, these are the only possible regions with finitely many arcs:
 if $R$ contains an essential simple closed curve $\gamma$,
 there must exist some arc $\alpha$
 intersecting~$\gamma$ non-trivially; applying Dehn twists about $\gamma$
 to $\alpha$ yields infinitely many arcs.
\end{rem}

\begin{prop}[Infinite links in the arc complex]\label{IL-arc}
 Suppose $\sigma \in \A(S,\P)$ is a multi-arc. Then
 \begin{enumerate}
 \item $\sigma \in \IL(\A(S,\P))$ if and only if $\sigma$ is disjoint from
 an essential simple closed curve on $(S,\P)$, and
 \item ${\sigma \in \MIL(\A(S,\P))}$ if and only if $S - \sigma$ has a
 $(1,1)$--annulus as its unique non-triangular region.
 \end{enumerate}
\end{prop}

\proof
 Let $R_1, \ldots, R_n$ be the non-triangular regions of $S - \sigma$,
 and $A_i = \A(R_i)$. By \cref{decomp-arc}, we have the minimal join decomposition
 $\lk_{\A(S,\P)}(\sigma) = A_1 * \ldots * A_n$.
 Thus $\# \lk_\K(\sigma) = \infty$ if and only if $\# A_i = \infty$ for some $i$.
 By the above remark, this holds if and only if $A_i$ contains an
 essential simple closed curve, yielding Part (i).
 
 Now, suppose $S - \sigma$ has a $(1,1)$--annulus as its unique non-triangular region.
 The core curve of this annulus is an essential simple closed curve disjoint from
 $\sigma$, and so ${\sigma \in \IL(\A(S,\P))}$. However, any arc
 $\alpha \in \lk_{\A(S,\P)}(\sigma)$ must connect the marked points on opposite boundary
 components of the annulus. Thus $S - (\sigma \cup \{\alpha\})$ has a quadrilateral
 as its unique non-triangular region. Therefore, $\sigma \cup \{\alpha\}$ has finite link,
 and so it follows that $\sigma \in \MIL(\A(S,\P))$.
 
 Conversely, assume $\sigma \in \MIL(\A(S,\P))$. Since $\sigma\in\IL(\A(S,\P))$, there
 exists an essential simple closed curve on $(S,\P)$ disjoint from $\sigma$.
 Let $R$ be the region of $S - \sigma$ containing such a curve.
 By the maximality assumption, every essential simple closed curve on $R$
 must intersect every essential arc on $R$.
 This implies that $R$ is the only non-triangular region of $S - \sigma$.
 We want to show that $R$ is a $(1,1)$--annulus.
 
 \begin{figure}[b]
 \begin{center}
  \begin{tikzpicture}[scale=0.8]
   \draw[color=blue] (0,0) to[out=0, in=-150] node[left]{$\beta$} (2.1,1.62);
   \draw[color=blue, dashed] (2.1,1.62) to[out=0, in=0, looseness=0.3] (2.1,-1.62);
   \draw[color=blue] (2.1,-1.62) to[out=160, in=-20] (0,0);

   \draw[color=blue] (8,0) to[out=180, in=-30] node[right]{$\beta'$} (5.9,1.62);
   \draw[color=blue, dashed] (5.9,1.62) to[out=180, in=180, looseness=0.3] (5.9,-1.62);
   \draw[color=blue] (5.9,-1.62) to[out=20, in=-160] (8,0);
  
   \draw[color=red, fill] (0,0) node[below]{$p$} circle (2pt);
   \draw[color=red, fill] (8,0) node[below]{$p'$} circle (2pt);
   \draw[color=red] (4,-0.5) to[bend right=5] node[below]{$\alpha$} (8,0);
   \draw[color=red] (4,-0.5) to[bend left=5] (0,0);
   
   \draw (4,1.5) to[out=180, in=-30, looseness=0.7] (0,2.2);
   \draw (4,1.5) to[out=0, in=-150, looseness=0.7] (8,2.2);

   \draw (4,-1.5) to[out=180, in=30, looseness=0.7] (0,-2.2);
   \draw (4,-1.5) to[out=0, in=150, looseness=0.7] (8,-2.2);
   
   \draw (4,1.5) to[out=0, in=0, looseness=0.5] node[right]{$\gamma$} (4,-1.5);
   \draw[dashed] (4,1.5) to[out=180, in=180, looseness=0.5] (4,-1.5);
  \end{tikzpicture}
  \caption{A $(1,1)$--annulus bounded by $\beta$ and $\beta'$, with
  $\gamma$ as a core curve.}
  \label{fig:proof_IL-arc}
 \end{center}
\end{figure}
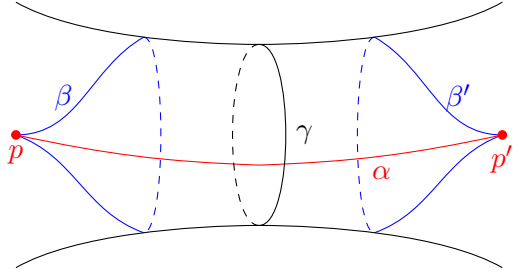

 First, suppose there exists a non-separating curve $\gamma$ on $R$.
 Then there exists an arc $\alpha$ on $R$ that intersects $\gamma$ exactly once
 (in minimal position). Let $p$ and $p'$ be the endpoints of $\alpha$ (these could
 possibly coincide). Let $a, a'$ be the subsegments of $\alpha$
 separated by the intersection point~$x$ of $\alpha$ and $\gamma$,
 with an endpoint at $p,p'$ respectively.
 Let $\beta$ be the arc that begins at $p$, runs along~$a$ to $x$, follows $\gamma$ exactly once,
 and then returns from $x$ to $p$ along $a$.
 Define $\beta'$ similarly using~$a'$ instead (see \cref{fig:proof_IL-arc}).
 Then, after a homotopy, $\beta$ and $\beta'$ bound a $(1,1)$--annulus $R' \subseteq R$,
 containing $\gamma$ as a core curve and $\alpha$ as an arc connecting opposite boundary components.
 By the maximality assumption, $\beta$ and $\beta'$ cannot be essential arcs on $R$, and so
 they must belong to~$\partial R$. Therefore $R = R'$ is a $(1,1)$--annulus.
 
 Next, we consider the case where all curves on $R$ are separating;
 this occurs precisely when~$R$ has genus $0$.
 Let $\gamma$ be any simple closed curve on $R$, and suppose $R'$ and $R''$ are
 the two components obtained by cutting $R$ along $\gamma$. 
 Our goal is to show that $R'$ and $R''$ are $(0,1)$--annuli.
 
 Suppose, for a contradiction, that $R'$ has at least two marked points $p,p'$.
 If~$p$ and~$p'$ are not contained in the same boundary component, then any arc
 connecting them is essential.
 If~$p$ and~$p'$ lie on the same boundary component, then there exists an essential
 arc with both endpoints on $p$ which cuts off a polygon having $p'$ as one of its
 marked points.
 This is an arc that is also essential on $R$.
 This contradicts the assumption on maximality.
 
 Thus, we have shown that $R'$ (and $R''$) has at most one marked point.
 In particular, $R'$~and~$R''$ have at most two boundary components (including
 the one obtained by cutting along~$\gamma$).
 Since~$\gamma$ is essential on $R$,
 neither $R'$ or $R''$ can be a closed disc with at most one interior marked point.
 Therefore $R'$ (and $R''$) has exactly one marked point, and this marked point
 must be a boundary marked point. It follows that both $R', R''$ are $(0,1)$--annuli.
 Gluing $R'$ and $R''$ along~$\gamma$ yields a $(1,1)$--annulus $R$. This completes
 the proof of Part (ii).
\endproof

\subsection{Triangles and flip graphs}

 Recall that a triangulation $\T$ on $(S,\P)$ is a maximal multi-arc.
 We refer to each (triangular) region of $S - \T$ as a triangle of $\T$.
 Note that two arcs $\alpha, \beta \in \A(S,\P)$ bound a folded triangle~$T$, with $\beta$ forming two sides of $T$, if and only if $\lk_{\A(S,\P)}(\alpha) = \{\beta\} * \K$ for some simplicial complex~$\K \subset \A(S,\P)$.
 
 The following was proven in \cite[Proposition 3.2]{irmak_mccarthy10}.
 
 \begin{prop}[Extending triangles to hexagons] 
  Let $T$ be a non-folded triangle on $(S,\P)$. Then there exists a triangulation $\T$ of $(S,\P)$
  containing~$T$ such that $S - (\T \setminus \partial T)$ has a hexagon as its unique non-triangular
  region. In particular, each side of $T$ cuts the hexagon into a pentagon and a triangle. $\hfill\square$
 \end{prop}
 
 This result can be used to give a purely combinatorial characterisation of non-folded triangles on $(S, \P)$.
 
 \begin{prop}[Topological triangle test]\label{top_tri_test}
  Let $\tau \in \A(S,\P)$ be a $2$--simplex. Then~$\tau$ bounds a non-folded triangle on
  $(S,\P)$ if and only if there exists a maximal simplex $\T \in \A(S,\P)$ containing $\tau$
  such that:
  \begin{enumerate}
   \item $\lk_{\A(S,\P)}(\T \setminus \tau)$ is isomorphic to the arc complex of the hexagon, and
   \item $\lk_{\A(S,\P)}(\T \setminus \sigma)$ is isomorphic to the arc complex of the pentagon
   for every $1$--simplex $\sigma \subset \tau$.
  \end{enumerate}
  
 \end{prop}

 \proof
  The necessity of these conditions follows from the previous proposition.

 \begin{figure}
 \begin{center}
  \begin{tikzpicture}[]
   \draw (90:1cm) to[bend right] (30:1cm) to[bend right] (-30:1cm) to[bend right] (-90:1cm) to[bend right] (-150:1cm) to[bend right] (150:1cm) to[bend right] (90:1cm);
   \draw[color=red] (90:1cm) to[bend right] (-30:1cm) to[bend right] (-150:1cm) to[bend right] (90:1cm);
   
   \begin{scope}[xshift=3cm]
    \draw (90:1cm) to[bend right] (30:1cm) to[bend right] (-30:1cm) to[bend right] (-90:1cm) to[bend right] (-150:1cm) to[bend right] (150:1cm) to[bend right] (90:1cm);
    \draw[color=red] (90:1cm) to[bend right] (-30:1cm);
    \draw[color=red] (90:1cm) -- (-90:1cm);
    \draw[color=red] (90:1cm) to[bend left] (-150:1cm);
   \end{scope}
   
   \begin{scope}[xshift=6cm]
    \draw (90:1cm) to[bend right] (30:1cm) to[bend right] (-30:1cm) to[bend right] (-90:1cm) to[bend right] (-150:1cm) to[bend right] (150:1cm) to[bend right] (90:1cm);
    \draw[color=red] (90:1cm) to[bend right] (-30:1cm);
    \draw[color=red] (90:1cm) -- (-90:1cm);
    \draw[color=red] (-90:1cm) to[bend right] (150:1cm);
   \end{scope}
  \end{tikzpicture}
  \caption{Three possible configurations of three disjoint arcs on a hexagon.}
  \label{fig:hyperbolic_hexagon_tau}
 \end{center}
 \end{figure}
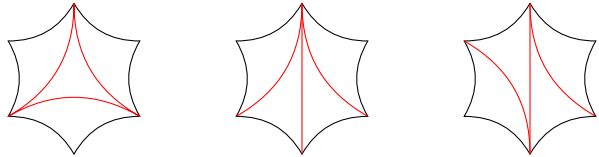
 
  To prove sufficiency, suppose $\T$ is a triangulation containing $\tau$
  which satisfies the given conditions.
  Combining Condition (i) and \cref{decomp-arc}, we deduce that $S - (\T \setminus \tau)$ has one unique non-triangular region $R$, and $R$ is a hexagon. 
  Moreover, $\tau$ forms three disjoint diagonals of $R$.
  This can occur in three possible ways: either $\tau$ bounds a triangle inside $R$,
  or $\tau$ contains a diagonal of $R$ that cuts $R$ into two quadrilaterals,
  leaving two possible choices for the other two diagonals (see \cref{fig:hyperbolic_hexagon_tau}).
  Applying Condition (ii) rules out the latter cases.
\endproof

 Suppose $\alpha$ is an arc of a triangulation $\T$ that meets two
 distinct triangles of $\T$. Then $S - (\T \setminus \{\alpha\})$ has a quadrilateral
 as its unique non-triangular region, with $\alpha$ as a diagonal;
 let $\beta\in\A(S,\P)$ be the other diagonal
 of this quadrilateral. The triangulation $\T'$ obtained by \emph{flipping} $\alpha$ in $\T$
 is given by $\T' = (\T \setminus \{\alpha\}) \cup \{\beta\}$ (see \cref{fig:flip_operation}). Note that
 \[\lk_{\A(S,\P)}(\T \setminus \{\alpha\}) = \{\alpha\} \sqcup \{\beta\} = \lk_{\A(S,\P)}(\T' \setminus \{\beta\}).\]
 We say that $\alpha$ is \emph{flippable} in $\T$ if the above operation can be performed;
 this holds precisely when $\alpha$ meets two distinct triangles of $\T$ or, equivalently,
 when $\alpha$ does not form two sides of a folded triangle of $\T$.
 The following is immediate.
 
\begin{figure}
 \begin{center}
  \begin{tikzpicture}[]
   \draw (0:0) to[bend right] (0,2) to[bend right] (2,2) to[bend right] (2,0) to[bend right] (0,0);
   \draw[color=red] (0,0) -- node[above left]{$\alpha$} (2,2);
   
   \draw[<->] (2.5,1) -- (3.5,1);
   
   \begin{scope}[xshift=4cm]
    \draw (0:0) to[bend right] (0,2) to[bend right] (2,2) to[bend right] (2,0) to[bend right] (0,0);
    \draw[color=red] (2,0) -- node[above right]{$\beta$} (0,2);
   \end{scope}
  \end{tikzpicture}
  \caption{The flip operation.}
  \label{fig:flip_operation}
 \end{center}
\end{figure}
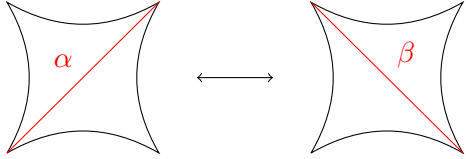
  
 \begin{lem}[Sides in non-folded triangles are flippable]\label{non-fold_flip}
  If $T$ is a non-folded triangle on $(S,\P)$, then every side of $T$ is flippable
  in any triangulation $\T$ containing $T$. $\hfill\square$
 \end{lem}
 
 \begin{definition}[Flip graph]
  The flip graph $\F(S,\P)$ is the graph with the set of all triangulations of $(S,\P)$ as its vertex set,
  where two triangulations $\T,\T'$ are joined by an edge if and only if they differ in
  precisely one~flip.
 \end{definition}

 The flip graph $\F(S,\P)$ is connected; see \cite{disarlo_parlier_14} and \cite{aramayona_koberda_parlier_15} for a discussion about its geometric properties.

\section{Half-translation surfaces and saddle connections}\label{sec:flat}

In this section, we discuss singular Euclidean metrics on $(S,\P)$
known as \emph{half-translation structures}.
This class of Euclidean structures on $S$ naturally arises through the
study of \emph{holomorphic quadratic differentials} on $S$.
However, we will use a more geometric approach instead of working directly with quadratic differentials.
For further details, refer to \cite{strebel_84, zorich_06, wright_15}.

For the remainder of this paper, assume $(S,\P)$ is a closed, connected, orientable surface.
A \emph{half-translation structure} on $(S,\P)$ consists of an atlas of charts from $S \setminus \P$
to~$\CC$, with transition maps of the form $z \mapsto \pm z + c$
for some $c \in \CC$.
If all transition maps are of the form $z \mapsto z + c$ for some $c \in \CC$
then the atlas defines a \emph{translation surface}.
Pulling back the Euclidean metric on $\CC$ gives a locally Euclidean metric on $S \setminus \P$
whose metric completion endows each point of~$\P$ with
the structure of a Euclidean cone point with cone angle $k\pi$ for some positive integer $k$.
Two half-translation structures on $(S,\P)$ are considered equivalent if they are related
by an isometry isotopic to the identity (the isotopies are not required to fix $\P$).

We shall denote a half-translation surface by $(S,q)$.
The set of marked points $\P$ is implicitly
assumed to be the set of singularities of the half-translation structure $q$.
In particular, we allow for \emph{removable singularities} and \emph{simple poles}; these are respectively singularities with cone angle~$2\pi$ and~$\pi$.
Let $\QD(S)$ denote the space of all half-translation structures on $S$
(or, equivalently, the space of holomorphic quadratic differentials on $S$).

Half-translation surfaces enjoy a consistent notion of slope.
Given any slope $\theta \in \RP^1$,
we can pull back the foliation of $\CC$ by straight lines with slope $\theta$
to a singular foliation on $(S,q)$.
Masur showed in \cite[Theorem 2]{masur_86} that for every half-translation surface $(S,q)$, there exists a dense set of slopes
in $\RP^1$ for which the corresponding foliation on $(S,q)$ possesses a closed
(non-singular) leaf.
Each such closed leaf on $(S,q)$ is contained in a (maximal) open Euclidean cylinder foliated by
closed leaves parallel (and isotopic) to the given leaf; each of these closed leaves
is called a \emph{core curve} of the cylinder.
A simple closed curve on~$S$ isotopic to a core curve of a cylinder on $(S,q)$ is
called a \emph{cylinder curve}; we denote the set of all cylinder curves on $(S,q)$
by $\cyl(q)$.

\begin{rem}[Isotopies of curves] \label{rem:curves}
 We can consider simple closed curves on $(S,\P)$ up to two forms of isotopy:
 when isotopies cannot pass through $\P$, and when they can.
 Let $\C(S,\P)$ and~$\C(S)$ denote the respective sets of isotopy classes
 of simple closed curves. We want to consider every isotopy class in $\C(S,\P)$ as an isotopy class in $\C(S)$. Strictly speaking, such a map can only be defined on curves in $\C(S,\P)$ not
 bounding discs with marked points, but curves of this type cannot be cylinder curves. Additionally, each isotopy class in $\C(S)$ arises from some isotopy class in $\C(S,\P)$ upon allowing isotopies to pass through $\P$.
 Therefore, $\cyl(q)$ can be equally regarded as a subset
 of $\C(S,\P)$ or of $\C(S)$. More precisely, the composition
 $\cyl(q) \hookrightarrow \C(S,\P) \to \C(S)$ is injective because
 if $\gamma\in\C(S)$ can be realised by a cylinder curve on~$(S,q)$, then there is
 a unique Euclidean cylinder on $(S,q)$ with $\gamma$ as its core curve.
\end{rem}

\subsection{Saddle connection complex}\label{sec:sccomplex}

We now introduce the main objects of study in this paper.

\begin{definition}[Saddle connection]
 A \emph{saddle connection} on $(S,q)$ is a locally isometric embedding
 $a : [0,l] \to (S,q)$ such that $\{a(0), a(l)\} \subseteq \P$ and 
 the restriction $a|_{(0,l)}$ is an embedding into $S \setminus \P$.
 We shall usually identify~$a$ with its image.
 In particular, we do not care about the orientation.
\end{definition}

Saddle connections are Euclidean straight-line segments, and thus have
a well-defined slope on $(S,q)$. In particular, the boundary of a Euclidean
cylinder on $(S,q)$ is a non-empty, finite set of saddle connections with the same slope.
Therefore, the set of slopes of saddle connections on~$(S,q)$ is also dense
in $\RP^1$.

\begin{definition}[Saddle connection complex]
 The \emph{saddle connection complex} $\A(S,q)$ is the simplicial complex
 whose vertices are saddle connections on $(S,q)$, and whose simplices are
 sets of pairwise disjoint saddle connections.
\end{definition}

Any saddle connection on $(S,q)$ is also (a representative of) an essential
arc on $(S,\P)$; in fact, it is the unique geodesic representative
in its proper homotopy class.
Let $\i_q: \A(S,q) \to \A(S,\P)$ be the natural inclusion map.
Minsky and Taylor introduced the \emph{saddle connection graph}, the $1$--skeleton of $\A(S,q)$,
in \cite{minsky_taylor_17} and established some important results regarding geodesic representatives of arcs, which we shall recall.

Consider the universal cover of $(S,q) \setminus \P$, and take its metric completion.
This metric completion is CAT(0), since $(S,q)$ is locally CAT(0), and so every path has a unique geodesic representative in its proper homotopy class. The action of the deck group extends to the set of completion points $\widetilde \P$; indeed, the set of orbits in $\widetilde \P$ is in natural bijection with $\P$.
Thus, every arc $\alpha \in \A(S,q)$ lifts to an arc $\tilde\alpha$ connecting a pair of distinct points in $\widetilde \P$. The unique geodesic representative of $\tilde\alpha$ descends to a path $\alpha_q$ on $(S,q)$ formed by concatenating a finite sequence of saddle connections. Note that this path does not depend on the choice of the lift $\tilde \alpha$.

We would like to say that $\alpha_q$ is the geodesic representative of $\alpha$ on $(S,q)$, however, in order to do this we need to allow for homotopies to move arcs so that they touch points of $\P$, but without being pushed past them. More precisely, we say that $\alpha_1$ is a representative of $\alpha \in \A(S,q)$ if there is an embedded representative $\alpha_0$ of $\alpha$ on $(S,\P)$, and a path homotopy $\alpha_t : [0,1] \to S$
such that for all $0 < s < 1$ and $0 \leq t < 1$ we have $\alpha_t(s) \in S \setminus \P$;
thus $\alpha_t$ is permitted to pass through points of $\P$ only when $t = 1$.
Note that if $\alpha_1$ passes through some points of $\P$ in its interior, then it does not necessarily determine a unique proper homotopy class of arcs; there is the ambiguity arising from the choice of side for ``pushing'' $\alpha_1$ off a marked point. In the case where consecutive saddle connections of $\alpha_q$ form an angle of less than $\pi$ on one side, then $\alpha_q$ cannot be pushed off on that side while remaining in the homotopy class of $\alpha$.

\begin{prop}[Geodesic representatives of arcs \cite{minsky_taylor_17}]
Every arc $\alpha \in \A(S,\P)$ has a unique geodesic representative on $(S,q)$ in its proper homotopy class.
The geodesic representative~$\alpha_q$ of $\alpha$ is a finite concatenation
of saddle connections, possibly with repetition, where no two of these saddle connections have transverse intersections. Furthermore, if $\alpha, \beta \in \A(S,\P)$
are disjoint arcs then no saddle connection appearing on $\alpha_q$ crosses a saddle connection appearing on $\beta_q$ (some saddle connections may appear on both, however). $\hfill\square$
\end{prop}

Minsky and Taylor use this result in \cite[Section 4.3]{minsky_taylor_17} to define a \emph{straightening operation} $\t_q : \A(S,\P) \to \A(S,q)$
by mapping a multi-arc $\sigma \in \A(S,\P)$ to the simplex $\t_q(\sigma) \in \A(S,q)$
whose vertices correspond to the saddle connections on $(S,q)$ appearing on $\alpha_q$ for some
$\alpha \in \sigma$.
Note that $\t_q \circ \i_q$ is the identity on $\A(S,q)$,
and so $\i_q(\A(S,q))$ is an induced subcomplex of~$\A(S,\P)$.
If $\alpha_0, \ldots, \alpha_k$ is a path in $\A(S,\P)$ with endpoints
$\alpha_0, \alpha_k \in \A(S,q)$, then choosing a vertex $\beta_i \in \t_q(\alpha_i)$
for each $i$ yields a path 
$\beta_0, \ldots, \beta_k$ (possibly with consecutive vertices coinciding)
in $\A(S,q)$ connecting $\alpha_0 = \beta_0$ to
$\alpha_k = \beta_k$. Thus, $\i_q \circ \t_q$ defines a coarse retraction
from $\A(S,\P)$ onto $\i_q(\A(S,q))$ that sends sets of diameter at most 1 to sets of diameter at most~1.
A similar construction was also used in \cite{tang_webb_18}.
Using the connectedness and uniform hyperbolicity of arc complexes,
we deduce the following.

\begin{prop}[Saddle connection complex is connected and hyperbolic] \label{isom_emb}
 The map $\i_q: \A(S,q) \to \A(S,\P)$ is an isometric embedding.
 In particular, $\A(S,q)$ is connected and uniformly Gromov hyperbolic.
 $\hfill \square$
\end{prop}

To prove our main theorem for the saddle connection complex,
we only use its isomorphism class as a simplicial complex.
This means that we can only use the fact that $\A(S,q)$ arises as a saddle connection complex of some half-translation surface without knowing what the underlying surface is. In particular, we do not know the labels of the vertices, that is, which topological arcs or saddle connections they represent.
The topological information must be recovered from the combinatorics of the graph in the process of the proof.

We shall show, at the end of \cref{sec:cylinder}, that every vertex in the saddle connection complex has infinite valence, and so the saddle connection complex is also locally infinite.
By following a standard method of Luo (which is adapted from an argument of Kobayashi \cite{kobayashi_88} and appears in \cite{masur_minsky_99}), Pan shows that $\A(S,q)$ has infinite diameter \cite{pan_22}.
A~consequence of \cref{isom_emb} is that $\i_q$ induces an embedding of the Gromov boundary of~$\A(S,q)$ into that of~$\A(S,\P)$.
The large-scale geometry of $\A(S,q)$, including its Gromov boundary, has been described in the paper \cite{disarlo_pan_randecker_tang_21}, which is a follow-up to this work.

We finish this subsection with an example of a saddle connection complex.

\begin{exa}[Saddle connection complex of the flat torus] \label{exa:scc_flat_torus}
 If $(S,q)$ is a flat torus with exactly one marked point, then it can be formed by gluing two Euclidean triangles along three pairs of sides. It is well-known that its arc complex $\mathcal{A}(S,\P)$ is isomorphic to the Farey tessellation of the hyperbolic plane. Moreover, every arc can be realised as a saddle connection and so the saddle connection complex~$\mathcal{A}(S,q)$ is also isomorphic to the Farey tessellation.
 
 Any two flat tori with exactly one marked point are affine equivalent. Moreover, this is the only affine equivalence class of half-translation surfaces that admit triangulations by three saddle connections, and hence a saddle connection complex of dimension $2$.
 Therefore, $(S,q)$ is a flat torus with one marked point if and only if~$\mathcal{A}(S,q)$ is isomorphic to the Farey tessellation.
\end{exa}

This example shows that if $(S,q)$ or $(S',q')$ is a flat torus with one marked point then for every simplicial isomorphism $\phi \colon \mathcal{A}(S,q) \to \mathcal{A}(S',q')$, there exists an affine diffeomorphism $F \colon (S,q) \to (S', q')$ which induces $\phi$. 
However, it is not true that this affine diffeomorphism is unique.
This is because of the existence of a hyperelliptic involution that acts trivially on~$\A(S,q)$, but exchanges the two triangles in any given triangulation (see Irmak-McCarthy \cite{irmak_mccarthy10}).
Composing~$F$ with this hyperelliptic involution yields the only other affine diffeomorphism inducing~$\phi$.

From now on, we consider only half-translation surfaces whose triangulations contain more than two triangles.

\subsection{Regions on half-translation surfaces}

Given a simplex $\sigma \in \A(S,q)$, we can cut $(S,q)$ along all saddle connections
in $\sigma$ to obtain a surface with boundary in a similar manner as for topological multi-arcs.
The resulting surface, which we shall denote by $(S - \sigma,q)$,
has a locally Euclidean metric, with boundary formed by finitely many
straight-line segments.

We shall define the associated saddle connection complex $\A(S - \sigma, q)$
as in the situation of the arc complex. \cref{link_iso} also
holds for the saddle connection complex: the natural map 
$\iota : (S - \sigma,q) \to (S,q)$ induces a simplicial isomorphism
  $\iota_* : \A(S - \sigma, q) \to \lk_{\A(S,q)}(\sigma)$. 

Suppose $R$ is a polygonal region of $(S - \sigma,q)$.
We shall refer to (essential) topological arcs on polygons as \emph{diagonals}.
A diagonal of $R$ is \emph{straight} if it can be realised by
a saddle connection, otherwise we call it {\emph{broken} (see \cref{fig:good_bad_diagonals}). 

On a polygon $R$, every diagonal is straight if and only if
all corners of $R$ have angle strictly less than $\pi$;
in this case we call $R$ a \emph{strictly convex} polygon.

\begin{figure}
 \begin{center}
  \begin{tikzpicture}[]
   \draw (0,0) -- (3,0) -- (4,3) -- (2,1) -- (1,1) -- (-2,2) -- (0,0);
   \draw[color=blue] (3,0) -- (1,1) -- (0,0) -- (2,1) -- (3,0) -- (-2,2);
   \draw[color=red, densely dashed] (0,0) to[bend right=30] (4,3);
   \draw[color=red, densely dashed] (-2,2) to[bend right=25] (2,1);
   \draw[color=red, densely dashed] (-2,2) to[out=-40, in=-120, looseness=1.5] (4,3);
   \draw[color=red, densely dashed] (1,1) to[out=-50, in=-130] (4,3);
  \end{tikzpicture}
  \caption{A hexagon with five straight diagonals (in blue) and four broken diagonals (in red, dashed).}
  \label{fig:good_bad_diagonals}
 \end{center}
\end{figure}

\begin{lem}[Convex polygons are embedded] \label{poly_embed}
 Let $\tilde R$ be a strictly convex polygon in the universal cover $(\tilde S, \tilde q)$. Then its image~$R$ on $(S,q)$ has embedded interior.
 In particular, if $R$ is a triangle then it is embedded, except for possibly at the corners.
\end{lem}

\proof
Recall that a strictly convex polygon meets the set of (pre-images of) singularities exactly at its corners and hence contains no singularities in its interior.
Suppose the interior of $R$ is not embedded. Then there is some deck transformation~$g$
of~$(\tilde S, \tilde q)$ such that $\tilde R$ and $g(\tilde R)$ have overlapping
interiors. This implies that $\tilde R \cap g(\tilde R)$ is a strictly convex polygon with non-empty interior;
in particular it has at least three sides.
We may choose local co-ordinates such that~$R$ is identified with a Euclidean
polygon in $\CC$, and such that $g$ is given by $w \mapsto \pm w + w_0$ for some~$w_0 \in \CC$. 

First, suppose $g$ is a translation.
Without loss of generality, we may apply a rotation to assume that
$w_0$ is purely imaginary. Consider a tallest vertical line segment $L$
contained in~$\tilde R$; this can be chosen so that at least one endpoint
of $L$ is a corner of $\tilde R$ (see \cref{fig:poly_embed}). Since $\tilde R$ and $g(\tilde R)$ have overlapping
interiors, we deduce that $|w_0| < \length(L)$. But this implies that there
is a singularity in the interior of $R$, or in the interior of
one of its sides, contradicting the assumption that $\tilde R$ is strictly convex.

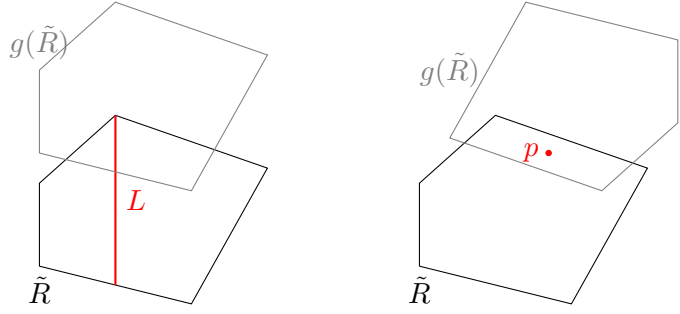
\begin{figure}[b]
 \begin{center}
  \begin{tikzpicture}[]
   \newcommand{\polygon}{(0,0) -- (2,-0.5) -- (3,1.3) -- (1,2) -- (0,1.1) -- (0,0)}
   \draw \polygon;
   \draw (0,0) node[below]{$\tilde R$};
   \draw[red, thick] (1,2) -- node[right]{$L$} (1,-0.25);
   \begin{scope}[yshift=1.5cm]
    \draw[gray] \polygon;
    \draw[gray] (0,1.1) node[above=0.2]{$g(\tilde R)$};
   \end{scope}
   
   \begin{scope}[xshift=5cm]
   \draw \polygon;
   \draw (0,0) node[below]{$\tilde R$};
   \draw[red, fill] (1.7,1.5) node[left]{$p$} circle (1pt);
   \begin{scope}[rotate=180, yshift=-3cm, xshift=-3.4cm]
    \draw[gray] \polygon;
    \draw[gray] (3,0.4) node{$g(\tilde R)$};
   \end{scope}
   \end{scope}
  \end{tikzpicture}
  \caption{The situation where $g$ is a translation (left) and a half-translation (right).}
  \label{fig:poly_embed}
 \end{center}
\end{figure}
Next, suppose $g$ is of the form $w \mapsto -w + w_0$. Then $g$ is a rotation
about a point $p \in \CC$ through an angle of $\pi$. Now $p$ must lie in the interior
of $\tilde R$, for otherwise $\tilde R$ and $g(\tilde R)$ will have disjoint
interiors (see \cref{fig:poly_embed}).
Then $p$ descends to a cone point of angle $\pi$ on $(S,q)$, which again implies the existence of a singularity in the interior of $R$, a contradiction.
\endproof

\subsection{Triangles and saddle flip graphs}

In \cref{sec:orient}, we will consider \emph{oriented triangles};
an oriented triangle is specified by a triple of sides $\TV = [a,b,c]$, considered up to cyclic permutation, where $a,b,c\in\A(S,q)$.
Any non-cyclic permutation of the sides determines the same triangle with the opposite orientation.
We say that two oriented triangles are \emph{consistently oriented} if they give the same orientation on $S$.

\begin{lem}[Maximal simplices in $\A(S,q)$]
 A maximal simplex of $\A(S,q)$ corresponds to a triangulation of $(S,q)$ by saddle connections.
\end{lem}

\begin{proof}
 Let $\sigma$ be a maximal simplex in $\A(S,q)$. Cutting $(S,q)$ along $\sigma$ gives rise to a disjoint union of Euclidean cone surfaces with piecewise geodesic boundary. It is a standard result that any such surface can be triangulated by Euclidean polygons (with corners at singularities). If some component $R$ is triangulated by at least two triangles, then it must contain a saddle connection that is not contained in $\partial R$. But this saddle connection would be disjoint from $\sigma$, contradicting the maximality assumption. Therefore, $\sigma$ cuts $(S,q)$ into a finite collection of Euclidean triangles.
\end{proof}
 
 We call such a triangulation a \emph{saddle triangulation} of $(S,q)$.
 Whenever we speak of a triangulation of a half-translation surface, we shall implicitly
 mean a saddle triangulation.
 Given a simplex $\sigma\in\A(S,q)$, we may also consider triangulations of
 $(S - \sigma,q)$; these are the maximal simplices of $\A(S- \sigma,q)$.
 
 A flip of a saddle triangulation is defined in the same way as for the (topological) flip graph.
 Given a triangulation $\T$ of $(S,q)$,
 a saddle connection $a \in \T$ is flippable in $\T$ if and only if
 the unique non-triangular region of $\T \setminus \{a\}$ is a strictly convex
 quadrilateral; in this case the quadrilateral has two straight diagonals,
 and so the flip is performed by replacing $a$ with the other diagonal.
 Note that $\T$ cannot have any folded triangles, since the three sides of a
 Euclidean triangle have distinct slopes.
 However, folded $n$--gons with $n\geq 4$ such as folded quadrilaterals can appear.
 These will be relevant in \cref{codim2}.
 
 \begin{definition}[Saddle flip graph]
  The \emph{saddle flip graph} $\F(S,q)$ is the graph with the set of all saddle triangulations on $(S,q)$
  as its vertex set, where two triangulations are joined by an edge if and only if they differ in precisely one flip.
 \end{definition}
 
 Given a simplex $\sigma \in \A(S,q)$, let $\F_\sigma(S,q)$
  be the induced subgraph of $\F(S,q)$
  whose vertices are ${\{\T \in \F(S,q) ~|~ \sigma \subseteq \T\}}$.
 
 We shall write $\F(S - \sigma, q)$ for the saddle flip graph of $(S - \sigma,q)$.
 The simplicial isomorphism $\iota_* : \A(S - \sigma, q) \to \lk_{\A(S,q)}(\sigma)$
 induces a natural graph isomorphism ${\F(S - \sigma,q) \to \F_\sigma(S,q)}$, given by $\T \mapsto \iota_*(\T) \cup \sigma$.
 
\begin{thm}[Flip graph is connected \cite{tahar_17}] \label{thm:tahar}
 For any simplex $\sigma\in\A(S,q)$, the graph $\F_\sigma(S,q)$ is connected.
 In particular, $\F(S,q)$ is connected. $\hfill\square$
\end{thm}
 
 Tahar's proof works for a more general class of surfaces.
 He allows for singular Euclidean surfaces with arbitrary cone angles, possibly with
 boundary formed by finitely many straight-line segments.
 The surface $(S - \sigma,q)$ falls into this class, and so
 $\F(S - \sigma,q)$ is connected.
 
 Note that if $\sigma \in \A(S,\P)$ is a topological triangulation, then $\t_q(\sigma)$ is not
 necessarily a saddle triangulation. However, all regions of $S - \t_q(\sigma)$
 will be polygons.

\subsection{Affine diffeomorphisms}

\begin{definition}[Affine diffeomorphism]
Let $(S,q)$ and $(S',q')$ be half-translation surfaces, with respective set of singularities $\P$ and~$\P'$. A homeomorphism $F : (S,q) \to (S', q')$
is called an \emph{affine diffeomorphism} if
it is locally affine on $S \setminus \P$ with respect to the underlying half-translation structures.

We call two half-translation surfaces \emph{affine equivalent} if there exists an affine
diffeomorphism between them. The group of affine self-diffeomorphisms of
$(S,q)$ shall be denoted $\Aff(S,q)$. 
We also allow orientation-reversing
diffeomorphisms.
\end{definition}

Note that if $F : (S,q) \to (S', q')$ is an affine diffeomorphism then $F^{-1}$ is also an affine diffeomorphism.
Furthermore, $F$ maps saddle connections to saddle connections and preserves disjointness, and so it induces a simplicial isomorphism $\A(S,q) \to \A(S',q')$.
The main goal of this paper is to prove that the converse also holds.

There is a natural $\GL(2,\R)$--action on $\QD(S)$ defined as follows.
Given $M \in \GL(2,\R)$ and a half-translation structure $q \in \QD(S)$,
define $M\cdot q \in \QD(S)$ to be the half-translation structure obtained by postcomposing
the charts from $(S,q)$ to $\CC$ (defined away from the set of singularities),
with the $\R$--linear map on $\CC\cong \R^2$ given by $z \mapsto M(z)$,
where $z\in\CC$ is a local co-ordinate.
The metric can be extended to the set of singularities in the usual manner.
Thus, the identity map $\id_S : (S,q) \to (S,M\cdot q)$
on the underlying surface $S$ is (isotopic to) an affine diffeomorphism
with derivative $M$.

The following is a key ingredient for our proof of the main theorem.

\begin{thm}[Cylinder Rigidity Theorem \cite{duchin_leininger_rafi_10}] \label{thm:cyl_rigid}
 Let $q,q' \in \QD(S)$ be half-translation structures on $S$.
 Then $\GL(2,\R) \cdot q = \GL(2,\R) \cdot q'$
 if and only if $\cyl(q) = \cyl(q')$. $\hfill\square$
\end{thm}

The statement in \cite[Lemma 22]{duchin_leininger_rafi_10}
regards $\cyl(q)$ as a subset of $\C(S)$. However, by \cref{rem:curves},
this theorem also holds when $\cyl(q)$ is viewed as a subset of $\C(S,\P)$.
We thank Kasra Rafi for pointing this out.

\section{Links of simplices in \texorpdfstring{$\A(S,q)$}{the saddle connection complex}}\label{sec:links}

In this section, we discuss an important distinction between join
decompositions of links in the saddle connection complex and in the arc complex.
We introduce the notion of a \emph{cordon} of a simplex, which
shall be used throughout this paper. Furthermore, we give a classification
of link types for low-codimensional simplices.

\subsection{Join decompositions}

By \cref{decomp-arc}, there is a natural correspondence between 
the factors of the minimal join decomposition of the
link of a simplex $\sigma$ in the arc complex,
and the non-triangular regions of $S - \sigma$.
In contrast, the analogous correspondence does not necessarily hold
for links in the saddle connection complex.

 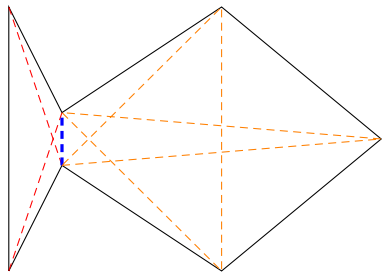
\begin{figure}
  \begin{center}
   \begin{tikzpicture}[scale=0.7]
    \draw (0,0) -- (1,2.0) -- (4.0,0) -- (7,2.5) -- (4,5) -- (1,3.0) -- (0,5) -- (0,0);
    \draw[densely dashed, color=red] (0,0) -- (1,3.0);
    \draw[densely dashed, color=red] (0,5) -- (1,2.0);
    \draw[densely dashed, very thick, color=blue] (1,2.0) -- (1,3.0);
     \draw[densely dashed, color=orange] (1,2.0) -- (7,2.5) -- (1,3.0) -- (4,0) -- (4,5) -- (1,2.0);
   \end{tikzpicture}
   \caption{A ``fish'' with all saddle connections drawn using dashed lines. The blue cordon separates all red saddle connections in the ``tail'' from the orange ones in the ``body''.}
   \label{fish_cordon_blue}
  \end{center}
 \end{figure}

\begin{exa}[Fish with cordon]
\cref{fish_cordon_blue} shows a possible region $R$
of a collection of saddle connections on a half-translation surface $(S,q)$.
Let $A$ be the induced subcomplex of $\A(S,q)$ whose vertices are all
the saddle connections in $R$ (these are indicated with dashed lines).
Let $A_1$, $A_2$, and $A_3$ be the induced subcomplexes of $\A(S,q)$
spanned by the saddle connections in red, blue, and orange respectively. 
The blue saddle connection cuts $R$ into two regions, each containing the
saddle connections of a single colour.
Therefore, no two saddle connections with distinct colours intersect, 
and so $A = A_1 \ast A_2 \ast A_3$ has a non-trivial join decomposition.
\end{exa}

We shall see that the presence of ``separating'' saddle connections,
such as the blue one in the above example, is the only obstruction
to having the desired correspondence between non-triangular
regions and factors in the minimal join decomposition.

For the rest of this paper, $\lk(\sigma)$ shall always denote the link of
$\sigma$ in $\A(S,q)$.
We shall also write $S - \sigma$ for the surface $(S - \sigma, q)$
when the underlying half-translation structure~$q$ is clear.
Thus, each region of $S - \sigma$ will be equipped with a half-translation structure.

\begin{definition}[Cordon]\label{def:cordon}
Let $\sigma\in\A(S,q)$ be a simplex. Call a saddle connection $\gamma \in \lk(\sigma)$ a \emph{cordon} of $\sigma$
if any of the following equivalent conditions hold:
\begin{enumerate}
 \item $\gamma$ is disjoint from every $\gamma' \in \lk(\sigma) \setminus \{\gamma\}$,
 \item $\lk(\sigma) = \{\gamma\} \ast A$ for some simplicial complex $A \subseteq \A(S,q)$, or
 \item $\gamma$ belongs to \emph{every} saddle triangulation containing $\sigma$ (and hence is non-flippable).
\end{enumerate}
\end{definition}

In particular, the cordons of $\sigma$ are precisely the single-vertex factors appearing
in the minimal join decomposition of $\lk(\sigma)$. 
If we were working in
the topological setting, then these arcs are precisely those contained in
once-marked monogon regions of $S - \sigma$.
However, saddle connections on a half-translation surface cannot bound monogons with
one singularity. Instead, cordons play an important role in our setting:
after cutting along all cordons of~$\sigma$, we obtain our desired
analogue of \cref{decomp-arc} for the saddle connection complex.

\begin{prop}[Decomposition of links in $\mathcal{A}(S,q)$] \label{decomp-saddle}
 Suppose $\sigma \in \A(S,q)$ is a simplex. Let $\lk(\sigma) = c_1 \ast \ldots \ast c_k \ast A_1 \ast \ldots \ast A_n$
  be the minimal join decomposition of its link, where the $c_i$'s are precisely the single-vertex factors (i.e.\ the cordons of~$\sigma$).
  Then $S - (\sigma \cup c_1 \cup \ldots \cup c_k)$ has exactly $n$ non-triangular regions $R_1, \ldots, R_n$.
  Moreover, up to permutation, $A_i$ is the induced subcomplex of $\A(S,q)$ whose vertices are the saddle
  connections contained in $R_i$.
\end{prop}

\begin{proof}
 Let $R_1, \ldots, R_m$ be the non-triangular regions of $S - (\sigma \cup c_1 \cup \ldots \cup c_k)$, and $B_i$ be the subcomplex of $\A(S,q)$ spanned by the saddle connections contained in $R_i$.
 We need to show that each $B_i$ cannot be decomposed as a non-trivial join. 
 The proof proceeds in exactly the same manner as in \cref{decomp-arc}, taking $\T$ to be a saddle triangulation containing $\sigma$, and using the following claim instead. (Saddle triangulations do not have any folded triangles, so we do not need to deal with that case.)

 \textbf{Claim:} For every saddle connection $\gamma \in B_i$, there exists a saddle connection $\gamma' \in B_i$ that intersects $\gamma$.
 
 Since we have already cut along all cordons of $\sigma$ to obtain the regions $R_i$,
 it follows that $\gamma \in B_i$ cannot be a cordon of $\sigma$.
 Therefore, there is some saddle connection $\gamma' \in \lk(\sigma)$ intersecting $\gamma$.
 Now, $\gamma'$ cannot be cordon of $\sigma$, and so it is contained in some region of
 $S - (\sigma \cup c_1 \cup \ldots \cup c_k)$. But $\gamma$ and $\gamma'$ intersect, and
 so they must belong to the same region, namely $R_i$. Thus $\gamma' \in B_i$, yielding the claim.
 \end{proof}

\subsection{Classification of low-codimensional simplices}\label{sec:classify}

Recall from \cref{decomp-arc} that
the isomorphism type of the link of a simplex in the arc complex
completely determines the topological type of its non-triangular regions.
For the saddle connection complex, however, this does not hold. 
We shall nevertheless provide a partial classification for simplices
corresponding to triangulations missing at most two saddle~connections.

Recall that every maximal simplex in $\A(S,q)$ corresponds to
a triangulation, and all triangulations possess the same number
of saddle connections.
Thus, we may define the \emph{codimension} of a simplex $\sigma \in \A(S,q)$
to be $\codim(\sigma) = \dim(\A(S,q)) - \dim(\sigma)$; this counts the number
of saddle connections that need to be added to $\sigma$ in order to produce a triangulation.
Observe that $\dim(\lk(\sigma)) = \codim(\sigma) - 1$.
Thus, for simplices $\sigma \in \A(S,q)$ of codimension at most 2,
the link $\lk(\sigma)$ is a simplicial graph.

Let $\CG_k$ denote the cycle graph of length $k$; $\PG_k$ the path graph of length $k$;
and $\NG_k$ the edgeless graph on $k$ vertices. We shall write $\PG_\infty$ for the
bi-infinite path graph.

The classification of links for codimension--$1$ simplices is as follows
(see \Cref{tab:codim1} in \Cref{app:classification}).

\begin{lem}[Codimension--$1$ simplices] \label{codim1}
 Let $\sigma\in\A(S,q)$ be a codimension--$1$ simplex.
 Then $S - \sigma$ has a quadrilateral as its unique non-triangular region
 which is:
 \begin{enumerate}
  \item strictly convex $\iff \lk(\sigma) \cong \NG_2$, 
  \item not strictly convex $\iff \lk(\sigma) \cong \NG_1$.
 \end{enumerate}
 \end{lem}
 
 \begin{proof}
 Complete $\sigma$ to a triangulation $\T$, and let $a$ be the unique saddle 
 connection in~\mbox{$\T \setminus \sigma$}. Then~$a$ meets two distinct triangles $T, T'$ of $\T$.
 Gluing $T$ and $T'$ along $a$ produces a quadrilateral~$Q$ with $a$ as a diagonal. 
 Now, $Q$ is strictly convex if and only if the angles at each of its corners are strictly less
 than $\pi$. This occurs precisely when there exists another diagonal $b$ of~$Q$,
 obtained by flipping $a$ in $\T$, cutting~$Q$ into two triangles (see \cref{fig:codimension_1_simplices}).
 \begin{figure}
 \begin{center}
  \begin{tikzpicture}[]
   \draw (0,0) -- (1,0.8) -- (-0.2,2) -- (-2.1,1.1) -- (0,0);
   \draw[color=red] (0,0) -- (-0.2,2);
   \draw[color=red] (1,0.8) -- (-2.1,1.1);
   
   \begin{scope}[xshift=3cm]
    \draw (0,0) -- (1,0.6) -- (2.6,0.2) -- (1.2,2) -- (0,0);
    \draw[color=red] (1,0.6) -- (1.2,2);
   \end{scope}
  \end{tikzpicture}
  \caption{A strictly convex quadrilateral and a quadrilateral that is not strictly convex, with their straight diagonals.}
  \label{fig:codimension_1_simplices}
 \end{center}
 \end{figure}
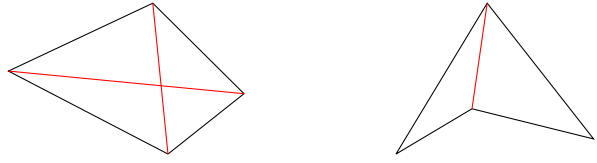
 In this case we have $\lk(\sigma) = \{a\} \sqcup \{b\} \cong \NG_2$.
 On the other hand, if $Q$ is not strictly convex then $a$ is the only diagonal of $Q$, and so $\lk(\sigma) = \{a\} \cong \NG_1$.
\end{proof}
 
\begin{cor}[Flippability condition]\label{cor:flip}
 Let $\T$ be a triangulation. Then a saddle connection $a\in\T$
 is flippable in $\T$ if and only if $\lk(\T \setminus \{a\}) \cong \NG_2$.
 $\hfill\square$
\end{cor}

For codimension--$2$ simplices, there are six possibilities for the links and eight possibilities for the non-triangular regions (see \Cref{tab:codim2} in \Cref{app:classification}).

\begin{lem}[Codimension--$2$ simplices] \label{codim2}
 Let $\sigma\in\A(S,q)$ be a codimension--$2$ simplex.
 Then the non-triangular region(s) of $S - \sigma$ comprises:
 \begin{enumerate}
  \item a $(1,1)$--annulus $\iff \lk(\sigma) \cong \PG_\infty$,
  \item a strictly convex pentagon $\iff \lk(\sigma) \cong \CG_5$,
  \item a pentagon with one broken diagonal $\iff \lk(\sigma) \cong \PG_3$,
  \item two strictly convex quadrilaterals $\iff \lk(\sigma) \cong \CG_4$,
  \item either a pentagon with two broken diagonals, or two quadrilaterals where
  exactly one is strictly convex$\iff \lk(\sigma) \cong \PG_2$,
  \item either a pentagon with three broken diagonals, two quadrilaterals that are both
  not strictly convex, or a bigon containing a simple pole $\iff \lk(\sigma) \cong \PG_1$.
 \end{enumerate}
 Moreover, the above list is exhaustive.
\end{lem}

\proof
Complete $\sigma$ to a triangulation $\T$, and suppose $a_1, a_2$ are
the saddle connections of~$\T \setminus \sigma$. Let $T_i \neq T_i'$ be the triangles
of $\T$ meeting $a_i$, for $i = 1,2$, and $Q_i$ be the quadrilateral
obtained by gluing $T_i$ and $T'_i$ along $a_i$.
We shall go through the cases depending on
how many of the triangles $T_1, T'_1, T_2, T'_2$ coincide.

First, suppose that the four triangles are distinct. Then $Q_1$ and $Q_2$
have disjoint interiors, and form the two non-triangular regions of $S - \sigma$.
Therefore, $\lk(\sigma) = A_1 * A_2$, where $A_i$ is the induced subcomplex
of $\A(S,q)$ with the saddle connections contained in $Q_i$ as vertices. By the
classification of links of codimension--$1$ simplices, we have
\[
  \lk(\sigma) \cong \left.
  \begin{cases}
    \NG_2 * \NG_2 \cong \CG_4 & \textrm{if both } Q_1, Q_2 \textrm{ are strictly convex,} \\
    \NG_2 * \NG_1 \cong \PG_2 & \text{if exactly one of } Q_1, Q_2 \textrm{ is strictly convex, or} \\
    \NG_1 * \NG_1 \cong \PG_1 & \text{if neither of } Q_1, Q_2 \textrm{ is strictly convex.}
  \end{cases} \right.
\]

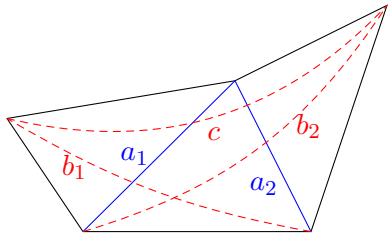
\begin{figure}
 \begin{center}
  \begin{tikzpicture}
   \draw (0,0) -- (3,0) -- (4,3) -- (2,2) -- (-1,1.5) -- (0,0);
   \draw[blue] (0,0) -- node[left]{$a_1$} (2,2);
   \draw[blue] (3,0) -- node[left, pos=0.3]{$a_2$} (2,2);
   \draw[red, densely dashed] (0,0) to[bend right=20] node[right,pos=0.6]{$b_2$} (4,3);
   \draw[red, densely dashed] (3,0) to[bend left=10] node[left,pos=0.7]{$b_1$} (-1,1.5);
   \draw[red, densely dashed] (-1,1.5) to[bend right] node[below]{$c$} (4,3);
  \end{tikzpicture}
  \caption{A pentagon, glued from $T_1'$, $T$, and $T_2'$, with its diagonals.}
  \label{fig:classification_codim2_pentagon}
 \end{center}
\end{figure}
  
Next, suppose $T_1 = T_2$, but $T'_1 \neq T'_2$. Then gluing $T \coloneqq T_1 = T_2$
to $T'_1$ and $T'_2$ along $a_1$ and $a_2$ produces a pentagon $P$, with
$a_1,a_2$ as non-intersecting straight diagonals. There are three more diagonals
in $P$: the topological arcs $b_1, b_2$ obtained by respectively flipping $a_1,a_2$
in $\T$, and the arc $c$ that intersects both $a_1,a_2$ (see \cref{fig:classification_codim2_pentagon}).
If $P$ is strictly convex, then all five diagonals are straight, and so $\lk(\sigma)$
is a copy of $\CG_5$ as shown:
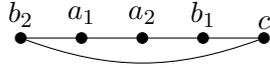
\begin{center}
 \begin{tikzpicture}[scale = 0.8]
  \tikzstyle{every node}=[draw,circle,fill=black,minimum size=4pt, inner sep=0pt]
   \draw (0,0) node (a) [label=above:$b_2$] {}
   -- (1,0) node (b) [label=above:$a_1$] {}
   -- (2,0) node (c) [label=above:$a_2$] {}
   -- (3,0) node (d) [label=above:$b_1$] {}
   -- (4,0) node (e) [label=above:$c$] {};
   \draw -- (e) to [bend left=20] (a) ;
 \end{tikzpicture}
\end{center}
If $P$ is not strictly convex, then $\lk(\sigma)$ is an induced subgraph of the above.
We claim that if~$c$ is a straight diagonal, then at least one of
$b_1$ or $b_2$ is also straight. This will imply that $\lk(\sigma)$ is isomorphic to a connected subgraph
of $\CG_5$. Consequently, if $P$ has $1 \leq k \leq 3$ broken diagonals then
$\lk(\sigma) \cong \PG_{4-k}$.

To prove the claim, suppose $c$ is a straight diagonal. Cutting $P$ along $c$
produces a triangle~$T'$ and a quadrilateral $Q$. By considering the angle sum of quadrilaterals,
at least three corners of~$Q$ have angle strictly less than $\pi$.
Suppose this holds for the corner of $Q$ lying at an endpoint of $a_1$ (otherwise at an endpoint of $a_2$).
Note that this corner is also a corner of $Q_1$. A small neighbourhood of the corner of $Q_1$ at the other endpoint of $a_1$ lies strictly inside $T'$, and so it also has angle strictly less than $\pi$. The remaining two corners of $Q_1$ are themselves corners of~$T$ and~$T'_1$ respectively. Therefore, $Q_1$ is a strictly convex quadrilateral and so $b_1$ is a straight diagonal of~$P$.

 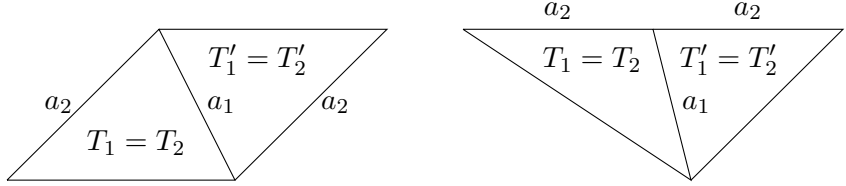
\begin{figure}[b]
 \begin{center}
  \begin{tikzpicture}[]
   \draw (0,0) -- (3,0) -- node[right]{$a_1$} (2,2) -- node[left]{$a_2$} (0,0);
   \draw (1.7,0.5) node{$T_1=T_2$};
   \draw (3,0) -- node[right]{$a_2$} (5,2) -- (2,2);
   \draw (3.3,1.6) node{$T_1'=T_2'$};
   
   \begin{scope}[xshift=6cm]
    \draw (0,2) -- (3,0) -- node[right]{$a_1$} (2.5,2) -- node[above]{$a_2$} (0,2);
    \draw (1.7,1.6) node{$T_1=T_2$};
    \draw (3,0) -- (5,2) -- node[above]{$a_2$} (2.5,2);
    \draw (3.5,1.6) node{$T_1'=T_2'$};
   \end{scope}
  \end{tikzpicture}
  \caption{A $(1,1)$--annulus (left) and a bigon (right) are formed by gluing $T_1=T_2$ and $T_1'=T_2'$ along $a_1$ and $a_2$.}
  \label{fig:codim2_annulus}
 \end{center}
 \end{figure}

Finally, we consider the case where $T_1 = T_2$ and $T'_1 = T'_2$.
Observe that $a_2$ appears as a side of both $T_1$ and $T'_1$.
Therefore $Q_1$ is either a parallelogram with $a_2$ appearing on opposite sides, or is a folded quadrilateral with a corner of angle $\pi$ formed by adjacent copies of $a_2$; see \cref{fig:codim2_annulus}.
In the first case, gluing $Q_1$ along the two copies of $a_2$ produces a $(1,1)$--annulus~$A$ forming the unique non-triangular region of $S - \sigma$. Every topological arc in $A$ is realisable as a saddle connection. Therefore, $\lk(\sigma)$ is isomorphic to the arc complex of the $(1,1)$--annulus, the bi-infinite path graph $\PG_\infty$.
In the second case, gluing $Q_1$ along the two copies of $a_2$ produces a bigon containing a simple pole as the unique triangular region of $S - \sigma$. Moreover, $a_1$ and $a_2$ are the only saddle connections in this bigon, so $\lk(\sigma) \cong \PG_1$.
\endproof

\pagebreak[3]

Observe that a region of a codimension--$2$ simplex that can be
recognised by its link is either a $(1,1)$--annulus, a strictly convex pentagon,
or a pentagon with at most one broken diagonal.
We shall call a pentagon \emph{almond\,\footnote{Almond stands for ``\emph{a}t \emph{m}ost \emph{o}ne \emph{n}on-straight \emph{d}iagonal'' (where at most one \emph{l}etter in the acronym does not stand for a word).}} if it has at most one broken diagonal.
Also note that any $(1,1)$--annulus on a half-translation surface must be a cylinder;
call any triangle contained in such an annulus a \emph{$(1,1)$--annular triangle}.
By examining Cases (i)--(iii) in the above proposition, we deduce the~following.

\begin{cor}[Detectable codimension--2 simplices] \label{detectable}
 Let $a,b$ be saddle connections in a triangulation $\T$. Then
 $\lk(\T \setminus \{a,b\})$ contains $\PG_3$ as an induced subgraph if and only
 if $\T \setminus \{a,b\}$ has either an almond pentagon or a $(1,1)$--annulus as its unique
 non-triangular region.
 In this situation, the region contains a triangle of $\T$ having $a$ and~$b$
 as two of its sides, and at
 least one of $a$ or $b$ is flippable in $\T$. $\hfill\square$
\end{cor}

\section{Cylinders and infinite links}\label{sec:cylinder}

By \cref{IL-arc}, a simplex in the arc complex of $(S,\P)$ has infinite link if and only if
it is disjoint from some simple closed curve on $(S,\P)$.
For the saddle connection complex, this does not hold, as the following
example shows.

\begin{exa}[Finite-link region with simple closed curve]
 Suppose $\sigma\in\A(S,q)$ is a simplex where every region is planar.
 Then each region contains only finitely many saddle connections, and so
 $\lk(\sigma)$ is finite. But if some region of $S - \sigma$ is not a topological disc with
 at most one interior marked point, then $\sigma$ will be disjoint from some
 simple closed curve (see \cref{fig:finite_link_disjoint_curve} for such an example).
  \begin{figure}
 \begin{center}
  \begin{tikzpicture}[scale=0.5]
   \draw (-1,-1) -- (1,-1) -- (1,1) -- (-1,1) -- (-1,-1);
   \draw (-4,-3) -- (4,-3) -- (4,3) -- (-4,3) -- (-4,-3);
   \draw[red,dashed] (0,0) circle (2.5cm);

   \draw[blue] (-1,-1) -- (-4,-3);
   \draw[blue] (-1,-1) -- (4,-3);
   \draw[blue] (-1,-1) -- (-4,3);

   \draw[blue] (1,-1) -- (4,-3);
   \draw[blue] (1,-1) -- (-4,-3);
   \draw[blue] (1,-1) -- (4,3);

   \draw[blue] (1,1) -- (4,3);
   \draw[blue] (1,1) -- (-4,3);
   \draw[blue] (1,1) -- (4,-3);

   \draw[blue] (-1,1) -- (-4,3);
   \draw[blue] (-1,1) -- (4,3);
   \draw[blue] (-1,1) -- (-4,-3);
  \end{tikzpicture}
  \caption{A planar region that contains a simple closed curve (in red) but only finitely many saddle connections (in blue).}
  \label{fig:finite_link_disjoint_curve}
 \end{center}
 \end{figure}
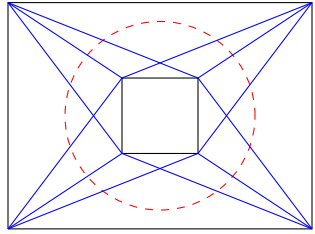
\end{exa}

Suppose $C$ is a cylinder on $(S,q)$. A \emph{transverse arc} of $C$ is an arc
contained in $C$ whose endpoints lie on opposite boundary components of $C$.
An arc contained in $C$ is realisable as a saddle connection if and only if it is a transverse arc.
In particular, $C$ contains infinitely many saddle connections.
Therefore, if a simplex $\sigma\in\A(S,q)$ is disjoint from (the interior of) 
some Euclidean cylinder $C$ on $(S,q)$, then
$\lk(\sigma)$ is infinite. We shall see that the converse is also true.
Call a simplex $\sigma \in \A(S,q)$ a \emph{triangulation away from a cylinder} if~$\sigma$ has a Euclidean cylinder as its only non-triangular region on $(S,q)$.

Recall that for a simplicial complex $\K$, the set of \emph{infinite link simplices} is
\[\IL(\K) = \{\sigma \in \K \colon \# \lk_\K(\sigma) = \infty\},\]
and $\MIL(\K)\subseteq\IL(\K)$ is the set of $\sigma\in\IL(\K)$ for which $\lk(\sigma')$ is finite
for all $\sigma' \supsetneq \sigma$.
We now state analogues of \cref{IL-arc} for the saddle connection complex.

\begin{prop}[Infinite link simplices] \label{IL}
 Let $\sigma \in \A(S,q)$ be a simplex. Then {$\sigma \in \IL(\A(S,q))$}
 if and only if $\sigma$ is disjoint from some cylinder curve on $(S,q)$.
\end{prop}

The proof of the above proposition shall be given over the next two subsections.
Before going into the details, let us give an immediate corollary.

\begin{cor}[Maximal infinite links simplices] \label{MIL}
 A simplex $\sigma\in\A(S,q)$ is a triangulation
 away from a cylinder on $(S,q)$ if and only if ${\sigma \in \MIL(\A(S,q))}$.
\end{cor}

\proof
If $\sigma$ is a triangulation away from some cylinder $C$, then 
$\#\lk(\sigma) = \infty$ by the above proposition. 
Any $\alpha \in \lk(\sigma)$ must be a transverse arc of $C$,
and so $\sigma \cup \{\alpha\}$ has a polygon as its unique non-triangular
region. Therefore, $\lk(\sigma \cup \{\alpha\})$ is finite and so
$\sigma \in \MIL(\A(S,q))$.

Conversely, if $\sigma \in \MIL(\A(S,q))$ then there exists some cylinder
$C$ disjoint from $\sigma$ by \cref{IL}. Note that $\sigma$ cannot intersect any saddle
connection in $\partial C$ transversely. By the maximality assumption, 
we deduce that $\partial C \subseteq \sigma$, and so $C$ is a non-triangular
region of~$S - \sigma$. If there exists any $\alpha \in \lk(\sigma)$
not contained in $C$, then $\lk(\sigma \cup \{\alpha\})$ is infinite,
violating the maximality assumption. Therefore, $\sigma$ must be a triangulation
away from $C$.
\endproof

\begin{rem}[Triangulations away from the same cylinder]\label{sameMIL}
 Observe that if $\sigma \in \MIL(\A(S,q))$ then the vertices of $\lk(\sigma)$ are precisely the transverse arcs in the cylinder that is the unique region of $S - \sigma$.

 Consequently, if $\sigma, \sigma' \in \MIL(\A(S,q))$ then $\lk(\sigma) = \lk(\sigma')$ if and only if they are triangulations away from the same cylinder.
\end{rem}

In order to use the Cylinder Rigidity Theorem (\cref{thm:cyl_rigid}) to prove our main theorem,
we need to detect cylinder curves
on $(S,q)$ using only combinatorial information from $\A(S,q)$.
By \cref{rem:curves}, the set of cylinder curves $\cyl(q)$ on $(S,q)$
can be regarded as a subset of either $\C(S,\P)$ or $\C(S)$;
in other words, it does not matter whether or not we permit isotopies of curves to pass
through~$\P$.
\cref{MIL} can be rephrased as follows.

\begin{cor}[Cylinder Test] \label{MIL_cyl}
 Let $\gamma$ be (an isotopy class of) a simple closed curve on~$S$.
 Then $\gamma \in \cyl(q)$ if and only if there exists some $\sigma\in\MIL(\A(S,q))$
 disjoint from $\gamma$.
 $\hfill \square$
\end{cor}

\subsection{A limiting geodesic}

Assume $\sigma \in \IL(\A(S,q))$.
Since $\lk(\sigma)$ is infinite,
there exists a region $R$ of $\sigma$ containing
infinitely many saddle connections.
Our goal is to prove that $R$ contains a cylinder curve.
As $R$ contains only finitely many singularities, we may choose a singularity $p \in R$ (possibly on the boundary~$\partial R$)
that forms an endpoint
of infinitely many saddle connections in $R$. Since $\RP^1$ is compact,
the slopes of these saddle connections ending at $p$ must accumulate; we shall rotate $(S,q)$
so that they accumulate on the horizontal slope.

Suppose $p$ has cone angle $k\pi$ and choose a sufficiently small $\epsilon > 0$
so that every
saddle connection on $(S,q)$ has length greater than $2 \epsilon$.
Cutting the open $\epsilon$--neighbourhood of $p$ on~$(S,q)$ along the $k$
horizontal line segments of length $\epsilon$ emanating from $p$
yields $k$ \emph{half-discs} (with horizontal boundary) centred at $p$.
Each half-disc has the form
 \[\{ z \in \CC ~|~ |z| < \epsilon, ~0 \leq \arg(z) \leq \pi \}\]
for some suitable choice of local co-ordinates, with $p$ identified with $0\in\CC$.
Say a saddle connection $a \colon [0,l] \rightarrow (S,q)$ \emph{begins}
at a half-disc $H$ centred at $p$ if (up to possibly reversing the orientation) $a(0) = p$
and $a \cap H$ contains $a([0,\epsilon))$.

 Consider an infinite sequence $a_i \in \lk(\sigma)$ of saddle connections beginning at
 a common half-disc~$H$, whose slopes $\theta_i \in \RP^1$ converge to the horizontal slope.
 Without loss of generality, we may assume that 
 $0 < \theta_{i+1} < \theta_i < \frac{\pi}{2}$ for all $i$, that is,
 the slopes ``decrease'' to the horizontal slope $\theta = 0$. 
 (If this is not possible, then we can instead assume
 $\pi > \theta_{i+1} > \theta_i > \frac{\pi}{2}$ for all $i$
 and argue similarly.) 
 We shall define a ``limiting'' geodesic $\beta$ of the $a_i$.

 Let $\tilde H$ be a lift of $H$ to the universal cover $(\tilde S, \tilde q)$,
 with centre $\tilde p$ descending to $p$. Choose local co-ordinates so that
 $\tilde H$ is identified with 
 \[\{ z \in \CC ~|~ |z| < \epsilon,~ 0 \leq \arg(z) \leq \pi \} \subset \CC.\]
 For any $r > 0$, the open $r$--ball $B_r(\tilde p)$ in $(\tilde S, \tilde q)$
 centred at $\tilde p$ contains finitely many singularities.
 Therefore, given $r > 0$ there exists some $\psi_r > 0$ such that the sector 
 \[U(r,\psi_r) \coloneqq \{ z \in \CC ~|~ |z| < r,~ 0 \leq \arg(z) \leq \psi_r \} \subset \CC\]
 isometrically embeds into $(\tilde S, \tilde q)$, with the
 embedding agreeing with that of $\tilde H$ on 
 ${\tilde H \cap U(r,\psi_r)}$ (compare \cref{fig:sector}). In particular, the image of 
 the interior of $U(r, \psi_r)$ in $(\tilde S, \tilde q)$ contains no singularities.
 Therefore, any saddle connection starting at $\tilde H$ with
 slope $0 < \theta < \psi_r$ has length at least $r$. It follows
 that the lengths $|a_i| \rightarrow \infty$ as $i \rightarrow \infty$.

 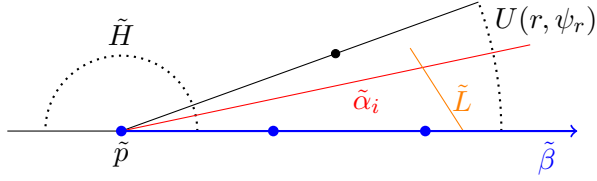
\begin{figure}
  \begin{center}
   \begin{tikzpicture}[scale=1]
      \tikzset{dot/.style={draw,shape=circle,fill=black,scale=0.3}}
   \node (A) at (0,0)[below] {$\tilde p$};
   \node (H) at (0,1)[above] {$\tilde H$};
   \node (U) at (17:5)[right] {$U(r, \psi_r)$};
   \draw[dotted, thick] ([shift=(0:1)]0,0) arc (0:180:1);
   \draw[dotted, thick] ([shift=(0:5)]0,0) arc (0:20:5);
   \draw (0,0) -- (20:3) node[dot] {} -- (20:5);
   \draw[red] (0,0) -- node[pos=0.6,below] {$\tilde \alpha_i$} (12:5.5);
   \draw[orange] (4.5,0) -- node[pos=0.4,right] {$\tilde L$} (3.8,1.1) node[right] {};
   \draw (0,0) -- (-1.5,0);
   \draw[blue,->,thick] (0,0) node[circle,fill,minimum size=4pt, inner sep=0pt] {} -- (2,0) node[circle,fill,minimum size=4pt, inner sep=0pt] {} -- (4,0) node[circle,fill,minimum size=4pt, inner sep=0pt] {} -- node[pos=0.8, below] {$\tilde \beta$} (6,0);
   \end{tikzpicture}
   \caption{A half-disc $\tilde H$ and sector $U(r, \psi_r)$ in $\CC$.
   Any line segment $\tilde L$ (as in \cref{crossbeta}) starting on
   the ray $\tilde \beta$ and leaving to its left will cross some $\tilde \alpha_i$.}
   \label{fig:sector}
  \end{center}
 \end{figure}
 
 Let $\tilde\beta \colon [0, \infty) \rightarrow (\tilde S, \tilde q)$
 be the horizontal unit-speed geodesic ray defined so that $\tilde \beta(t)$
 coincides with the image of $t \in U(r, \psi_r)$ in $(\tilde S, \tilde q)$
 whenever $0 \leq t < r$. Let $\beta$ be its projection to $(S,q)$.
 Note that whenever $\beta$ passes through a singularity, it has an angle of $\pi$ to its left; in particular, if~$\beta$ passes through a simple pole then it must immediately backtrack.
 The saddle connections~$a_i$ converge to $\beta$ in the following sense:
 
 \begin{lem}
  Let $\tilde a_i$ be the lift of $a_i$ beginning at $\tilde H$.
  Then for all $r>0$, the sequence $\tilde a_i \cap B_r(\tilde p)$ converges
  to $\tilde \beta \cap B_r(\tilde p)$ as subsets of $(\tilde S, \tilde q)$
  under the Hausdorff topology.
 \end{lem}
 
 \proof
 Fix some $r > 0$. Then for all~$i$ sufficiently large, we have $\theta_i < \psi_r$ 
 and so $\tilde a_i \cap B_r(\tilde p)$ can be identified with the straight-line
 segment in $U(r, \psi_r)$ defined by $\arg(z) = \theta_i$ (together with the origin).
 Then the Hausdorff distance between $\tilde a_i \cap B_r(\tilde p)$ and
 $\tilde \beta \cap B_r(\tilde p)$ is at most $r|\sin\theta_i|$.
 The result follows since $r|\sin\theta_i| \rightarrow 0$ as $i \rightarrow \infty$.
 \endproof

 \begin{lem} \label{crossbeta}
  Let $L$ be a non-horizontal line segment on $(S,q)$ that starts at a point
  on $\beta$, and leaves $\beta$ on its left.
  Then $L$ intersects some $a \in \lk(\sigma)$ transversely.
  In particular, no saddle connection in $\sigma$ intersects $\beta$ transversely. 
  
 \begin{proof}
  Suppose $L$ starts at $\beta(t)$ for some $t \geq 0$.
  Let $\tilde L$ be the lift of $L$ to $(\tilde S, \tilde q)$ that starts at~$\tilde \beta(t)$. For a fixed $r > t$, the pre-image of 
  $\tilde L$ in $U(r, \psi_r)$ contains a straight line segment connecting
  $t\in\CC$ to some point $z_0 \in \CC$ satisfying $|z_0| < r$ and 
  $0 < \arg(z_0) < \psi_r$; refer to \cref{fig:sector}.
  For~$i$ sufficiently large, the saddle connection
  $a_i$ has slope $0 < \theta_i < \arg(z_0)$ and length at least~$r$.
  Therefore, the lift $\tilde a_i$ starting at $\tilde H$
  must intersect $\tilde L$ transversely, and the first part of the statement follows.

  In particular, the segment $L$ cannot be a segment of a saddle connection in $\sigma$.
 \end{proof} 
  
 \end{lem}

\subsection{Cylinders on minimal components}
 
 We shall use some results concerning measured foliations and their components.
 These can be defined for any foliation with constant slope on $(S,q)$ but,
 for simplicity, we shall assume we are always working with the horizontal
 foliation.
 
 Every non-singular leaf is either closed or is dense in a subsurface, called a \emph{minimal component} of the horizontal foliation.
 The horizontal foliation decomposes the surface $S$ into a finite union of maximal horizontal cylinders and minimal components. The boundary of each such component is formed by a set of horizontal saddle connections. Refer to \cite{fathi_laudenbach_poenaru_79} for more~background.

 Our goal is to prove that there is some cylinder curve on $(S,q)$ disjoint from $\sigma \in \IL(\A(S,q))$ from the previous section.
 
 Recall that the limiting geodesic $\beta$ always has an angle of $\pi$ on its left whenever
 it passes through a singularity. This property,
 together with any segment $\beta([t_1, t_2])$ for any $0 \leq t_1 < t_2$,
 uniquely determines $\beta$ for all $t \geq t_2$.
 Therefore, if there exists some half-disc $H$ centred at a singularity
 $p$ such that~$\beta$ passes through $p$ with $H$ on its left more than once,
 then it must eventually repeatedly run
 over some finite set of horizontal saddle connections.
 In this case, since $\beta$ always has an angle of $\pi$ to its left,
 there exists a closed leaf of the horizontal foliation on
 the left of $\beta$.
 Therefore, $\beta$ forms a boundary component of some horizontal cylinder $C$. 
 If some saddle connection $\gamma\in\sigma$ intersects
 a core curve of $C$ transversely, then $\gamma \cap C$
 contains a straight line segment starting at $\beta$,
 and leaving on its left. But this contradicts \cref{crossbeta}.
 Therefore, $\sigma$ is disjoint from a cylinder curve as desired.
 
 Now assume $\beta$ does not pass through a singularity with a given half-disc on
 its left more than once.
 Since there are finitely many such half-discs on
 $(S,q)$, there is some largest value of~$t_0 \geq 0$ such that
 $\beta(t_0)$ is a singularity. Then the geodesic ray
 $\beta' = \beta([t_0, \infty))$ is a non-compact leaf of the
 horizontal foliation on $(S,q)$.
 Let $X \subseteq S$ be the minimal
 component containing $\beta'$.

 \begin{lem} \label{lem:non-horizontal_saddle_connection_disjoint_cylinder}
  If $\gamma\in\sigma$ is a non-horizontal saddle connection then it
  is disjoint from the interior of $X$.
  \end{lem}
  
 \proof
  Suppose $\gamma\in\sigma$ is not horizontal. If $\gamma$ meets
  the interior of $X$, then it must intersect~$\beta'$ transversely
  since $\beta'$ is dense in $X$. But this contradicts \cref{crossbeta}.
 \endproof
 
 The rest of this section is devoted to proving the following result.
 Together with the above lemma, this will complete the proof of \cref{IL}.
 
\begin{prop}[Subsurfaces with horizontal boundary contain cylinders] \label{prop:folcyl}
 Let $\zeta \in \A(S,q)$ be a set of horizontal saddle connections and $X$ be a connected component of~$S - \zeta$.
 Then there exists a cylinder curve on~$X$ that is disjoint
 from every horizontal saddle connection.
\end{prop}

In particular, \cref{prop:folcyl} implies that every vertex of $\A(S,q)$ has infinite valence, as stated in \cref{sec:sccomplex}: every saddle connection is disjoint from a maximal cylinder, and hence is disjoint from infinitely many saddle connections.

 The proposition follows immediately if $X$ is a cylinder, so we assume otherwise.
 Our strategy is to adapt a theorem of Vorobets to the case of
 half-translation surfaces with horizontal boundary.
 
 \begin{thm}[Cylinders of definite width \cite{vorobets_03}] \label{thm:vorobets}
  Let $\Sigma$ be a finite-area translation surface (without boundary),
  possibly with a finite set of removable singularities.
  Then there exists a Euclidean cylinder on $\Sigma$
  with width at least $W\sqrt{\area(\Sigma)}$,
  where~$W$ is a constant depending only on the genus and the cone angles of singularities of $\Sigma$. $\hfill\square$
 \end{thm}

 \begin{proofof}{\cref{prop:folcyl}}
 First, double $X$ along its boundary to obtain a closed half-translation surface $X'$ (we regard any point arising from a boundary singularity with interior angle $\pi$ as a removable singularity on $X'$).
 The boundary arcs of $\partial X$
 give rise to a set $\delta$ of disjoint horizontal saddle connections
 on $X'$ which are fixed under the natural involution of $X'$.
 Next, take the canonical translation double branched
 cover of $X'$ to obtain a translation surface~$\hat X$ (with simple poles lifting to removable singularities).
 Let $\hat \delta$ be the pre-image of $\delta$ on~$\hat X$.
 Observe that any saddle connection on $\hat X$ that does not intersect $\hat \delta$
 transversely descends to a saddle connection of the same length and slope on $X$.
 (Note that a saddle connection on $\hat X$ descends to a boundary saddle connection on $X$
 if and only if it belongs to $\hat \delta$.)
 Conversely, every (non-boundary) saddle connection on $X$ has exactly
 four saddle connections in its pre-image on~$\hat X$, also with the same
 slope and length.

 By applying an $\SL(2,\R)$--deformation that shrinks in the horizontal direction, we can assume that all horizontal saddle connections on $\hat X$
 have length strictly less than $W\sqrt{\area( \hat X)}$.
 Let $\hat C$ be a Euclidean cylinder on~$\hat X$ as given by \cref{thm:vorobets}.
 Then any horizontal saddle connection on $\hat X$ must be disjoint from the
 interior of $\hat C$. Thus, any core curve $\hat \eta$ of $\hat C$
 is disjoint from every horizontal saddle connection on $\hat X$.
 It follows that its image $\eta$ on $X$ is also disjoint from $\partial X$ and
 all horizontal saddle connections. But $\eta$ is a closed geodesic
 of constant slope, and so must be  a cylinder curve on $X$, as desired.
   $\hfill\square$
 \end{proofof}

\section{Flowing results}\label{sec:flow}

In this section, we shall establish some results concerning straight-line flows. A \emph{straight-line flow} on $(S,q)$ in direction $\theta$ is an action of $\mathbb{R}_{\geq 0}$ on $(S,q)$ such that for every $p\in S$, the map $t \mapsto t\cdot p$ is a unit-speed geodesic starting from $p$ with direction $\theta$.
Strictly speaking, this action may not be defined for all times for some points as trajectories can hit singularities, however, this only occurs on a measure zero set. Thus, we will still call this action a flow as it is standard terminology in the context of translation surfaces.

For our purposes, we consider a flow~$\varphi^t$ emanating from a given saddle
connection $a \in \A(S,q)$ in some direction $\theta$ not parallel to~$a$.
To simplify the exposition, we shall assume that~$a$ is vertical and the
direction of the flow~$\varphi^t$ is horizontal.
However, the results in this section work in general
by applying an appropriate $\SL(2,\R)$--deformation.
The \emph{height} of a saddle connection $a' \in \A(S,q)$ with respect to~$\varphi^t$
is given by taking its length measured orthogonally to the flow direction.
Equivalently, $\height(a')$ is the intersection number
between $a'$ and the foliation on $(S,q)$ with slope $\theta$.
Call $a\in\sigma$ a \emph{tallest} saddle connection in $\sigma$ if
it has maximal height among all saddle connections~in~$\sigma$.
Analogously, we define the \emph{width} of a saddle connection, measuring the length along the flow direction.

Throughout this section, we shall assume that $a$ is a vertical saddle connection
in a simplex $\sigma \in \A(S,q)$, with $\height(a) = h > 0$.
Equip $a$ with a unit-speed parameterisation and orientation
$a \colon [0,h] \rightarrow (S,q)$ so that the flow $\varphi^t$ emanates from
the right of $a$. 

 If the endpoints of $a$ coincide, then the flow $\varphi^t$ is not uniquely
 defined at the common endpoint. This will not cause us any problems;
 we shall work in the universal cover to define the two trajectories
 $\varphi^t(a(0))$ and $\varphi^t(a(h))$ in the next subsection.
 
 A notion we shall use many times in the rest of this paper is that of
 \emph{visibility} with respect to~$\varphi^t$ and $\sigma$;
 a precise formulation is given in \cref{sec:vis}.
 Informally, a singularity $z$ is \emph{visible} if some trajectory of $\varphi^t$
 starting on the interior of $a$ hits $z$ at some time $t = t_0 > 0$ without
 intersecting any saddle connection of $\sigma$ transversely.
 (Under an additional condition, we also allow for trajectories to start at an endpoint of $a$.)

 The following two propositions will be proven in the subsequent subsections.
 
 \begin{prop}[Tallest saddle connections see singularities] \label{prop:tall}
 Suppose $a$ is a tallest saddle connection in $\sigma$.
 Then there exists a visible singularity with respect to
 $\sigma$ and $\varphi^t$.
 In particular, if $a$ is strictly tallest in $\sigma$ and does not lie in a horizontal cylinder then there exists a visible singularity along some trajectory of $\varphi^t$ that starts in the interior of $a$.
 \end{prop}

\begin{prop}[Visible singularities yield visible triangles] \label{prop:vis}
 Suppose that for some $y \in [0,h]$, the trajectory $\varphi^t(a(y))$ hits a visible singularity at $t = t_0 > 0$.
 Then there exists a triangle $T$ on $(S,q)$, with $\height(T) = h$, such that:
 \begin{itemize}
  \item $\partial T$ has no transverse intersections with $\sigma$, and
  \item $T$ has $a$ as one of its sides, and appears on the side of $a$ from
  which $\varphi^t$ emanates.
 \end{itemize}
 Furthermore, $T$ can be chosen so that the corner opposite of $a$ is
 $\varphi^{t_1}(a(y_1))$, where $t_1 \leq t_0$ and $y_1 \in (0,h) \cup \{y\}$.
 In particular, if $0 < y < h$ then $T$ has no horizontal sides.
\end{prop}

In other words, if a singularity on $(S,q)$ is \emph{visible}
(not blocked by $\sigma$) along some trajectory
of $\varphi^t$ starting from $a$, then 
there exists a simplex $\sigma' \supseteq \sigma$
such that $(S - \sigma',q)$
has a triangular region meeting $a$ on its right.

As many readers would be familiar with these flowing-style arguments for translation
surfaces, let us give an informal sketch of the main ideas;
the formal proofs of the above propositions shall be given in the following two subsections.

To prove \cref{prop:tall}, suppose that the flow $\varphi^t$ emanating from $a$ does not see any visible singularities.
Then for all $y \in [0,h]$, the flow $\varphi^t(a(y))$ will hit the interior of
some saddle connection in~$\sigma$ before any singularities.
We then argue that there exists a saddle connection~$a' \in \sigma$ that blocks every trajectory.
But this implies that $a'$ is strictly taller than $a$.

For \cref{prop:vis}, suppose that $\varphi^t(y)$ hits a visible singularity at time $t = t_0$.
Consider the two (oriented) topological arcs $\eta^+, \eta^-$ that start at an endpoint of $a$,
run vertically along~$a$ until~$a(y)$, and then follow the horizontal flow trajectory
until they hit $z$. Straighten $\eta^\pm$ to obtain its geodesic representative $\t_q(\eta^\pm)$.
Let $z_1^\pm$ be the terminal endpoint of the first saddle
connection appearing along $\t_q(\eta^\pm)$.
Then at least one of $z_1^+$ or $z_1^-$ will be a corner of a triangle~$T$ satisfying the desired properties.

\subsection{Constructing a triangle}

Choose local complex co-ordinates for $(S,q)$
so that $a(y)$ is given by $\imagi y \in \CC$. Thus, in these co-ordinates,
$\varphi^t(a(y))$ is given by $t + \imagi y$ for small $t \geq 0$. 
There are only countably many $y_i \in [0,h]$ such that
the trajectory $\varphi^t(a(y_i))$ is singular: it (first) hits
a singularity at some
time $t = t_i > 0$ and so $\varphi^t$ cannot be defined for $t > t_i$. 
Let $Z \subset \CC$ denote the
countable set of points
$z_i = t_i + \imagi y_i$ arising in this manner.
We claim that the set $Z$ is non-empty. By Poincar\'e recurrence, almost every trajectory starting in $[0,h]$ has to hit a copy of $a$. If some trajectory starting in $[0,h]$ hits a singularity without hitting a copy of $a$ then we are done, so let us assume otherwise. Given $y \in [0,h]$, let $a'$ be the first copy of $a$ hit by the trajectory starting at $y$. Consider the maximal subinterval $I \subseteq [0,h]$ containing $y$ such that every trajectory starting on $I$ hits $a'$ first among all copies of $a$. If the trajectory starting at some $y' \in I$ hits the interior of $a'$ then, by our assumption, so must every trajectory in an open neighbourhood of $y'$ in $[0,h]$. Therefore, at least one of the trajectories starting an endpoint of $I$ must hit an endpoint of $a'$, a singularity.

The trajectories $\varphi^t(y)$ for all other $y \in [0,h]$ can be defined
for all $t \geq 0$.
Therefore, the map~$\varphi$ with $\varphi(t + \imagi y) \coloneqq \varphi^t(a(y))$ is defined on
\[ Y \coloneqq \{z \in \CC ~|~ \Re(z) \geq 0, 0 \leq \Im(z) \leq h\} \setminus \bigcup_{z_i \in Z} \{z_i + s ~|~ s > 0 \} \subset \CC. \]
This domain is an infinite horizontal strip with countably many horizontal rays deleted (see \cref{fig:flow}).
The map $\varphi \colon Y \rightarrow (S,q)$ is a locally isometric embedding
(with respect to the induced Euclidean path metric on $Y$) compatible
with the half-translation structure on $(S,q)$.
Furthermore, $\varphi$ lifts to a map
$\tilde{\varphi} \colon Y \rightarrow (\tilde{S}, \tilde{q})$.
Since the set of singularities on the universal cover $(\tilde{S}, \tilde{q})$
is discrete, it follows that $Z\subset \CC$ is discrete also.

\begin{figure}
 \begin{center}
  \begin{tikzpicture}[yscale=0.95]
    \tikzset{dot/.style={draw,shape=circle,fill=black,scale=0.4}}
    \tikzset{dotgreen/.style={draw,shape=circle,fill=green!40!black,scale=0.4}}
    \fill[blue!10,draw=black] (0,0) -- (0,5) -- (2,3.2) -- (0,0);
    \draw[red] (0.8,5) -- (2,3.2) -- (1.5,4.5) -- (4.8,3.9);
    \draw[red] (9,0) -- (7.5,1.5);
    \draw[red] (6,5) -- (8,4.5);
    \draw[red] (10,3.9) -- (11,4.5);
    \draw[] (0,5) -- (12,5);
    \draw (0,0) -- (6,0);
    \draw[gray, dashed] (6,0) node[dot] {} -- (12,0);
    \draw[dashed,gray] (2,3.2) -- (12,3.2);
    \draw[dashed,gray] (4.5,1.5) -- (12,1.5);
    \draw[dashed,gray] (5,2.5) -- (12,2.5);
    \draw[dashed,gray] (10,2.2) -- (12,2.2);
    \draw[dashed,gray] (1.5,4.5) node[dotgreen] {} -- (12,4.5);
    \draw[dashed,gray] (4.8,3.9) node[dot] {} -- (12,3.9);
    \draw[very thick] (0,5) node[dot] {} node[left] {$J_0^+$} -- node[left] {$J_0$} (0,0) node[dot] {} node[left] {$J_0^-$};
    \draw[blue,middlearrow={latex},thick] (0,0) -- node[above] {$\eta^-(z)$} (4.5,1.5) node[circle,fill,minimum size=4pt, inner sep=0pt] {} node[below] {$z_1^-$} -- (10,2.2) node[circle,fill,minimum size=4pt, inner sep=0pt] {};
    \draw[blue,middlearrow={latex},thick] (0,5) -- (2,3.2) node[circle,fill,minimum size=4pt, inner sep=0pt] {} node[above right] {$z_1^+$} -- node[pos=0.25,below] {$\eta^+(z)$} (5,2.5) node[circle,fill,minimum size=4pt, inner sep=0pt] {} -- (10,2.2) node[circle,fill,minimum size=4pt, inner sep=0pt] {}; 
    \draw[blue,->] (0,2.2) -- node[pos=0.3,below] {$L_z$} (10,2.2) node[below] {$z$};
    \draw[orange,dashed,thick] (0.8,4.7) -- (1.5,0.2);
    \draw[red] (2,0) -- (4.5,1.5) -- (10,2.2);
    \node[blue] (A) at (0.6,2.7) {$T(z)$};
  \end{tikzpicture}
  \caption{The domain $Y$ in $\CC$, with the deleted horizontal rays indicated by dashed
  lines. The visible singularity~$z \in Z$
  determines two paths $\eta^+(z)$ and $\eta^-(z)$ which can be used to construct a visible
  triangle $T(z)$. Line segments belonging to $\J$ are drawn in red. The green point is an element of $Z$ which has minimal real part but is not visible (in the sense of \cref{sec:vis}). The orange dashed line is as in \cref{cross_visible}.}
  \label{fig:flow}
 \end{center}
\end{figure}

Let $J_0 = a([0,h]) \subset Y$, that is $J_0$ is a vertical line segment,
with $J_0^- = 0$ and $J_0^+ = \imagi h$ as its endpoints.
Given $z \in Z$, consider the finite sets
\begin{eqnarray*}
Z^+(z) &=& \{z' \in Z ~|~ \Im(z') \geq \Im(z), \Re(z') \leq \Re(z)\} \cup J_0^+ \quad\textrm{and}\\
Z^-(z) &=& \{z' \in Z ~|~ \Im(z') \leq \Im(z), \Re(z') \leq \Re(z)\} \cup J_0^-.
\end{eqnarray*}
These are the points on the left and above $z$ respectively on the left and below $z$.
Let $Y^\pm(z)$ be the convex hull of $Z^\pm(z)$ in $\CC$.
If $Y^\pm(z)$ has non-empty interior, then there are two polygonal paths
in $\partial Y^\pm(z)$ from $J_0^\pm$ to $z$.
Let $\eta^\pm(z) \subseteq \partial Y^\pm(z)$ be the ``left'' path:
the one so that $Y^\pm(z)$ lies entirely to the right (see \cref{fig:flow}).
If $Y^\pm(z)$ is degenerate (which occurs precisely when the points in $Z^\pm(z)$ are collinear),
then it is a straight line segment from $J_0^\pm$ to $z$;
we take $\eta^\pm(z)$ to be this path in this situation.
In either case, $\eta^\pm(z)$ is a concatenation of straight line segments connecting
consecutive points in some sequence
\[J_0^\pm = z^\pm_0, z^\pm_1, \ldots, z^\pm_{k^\pm} = z,\]
where each $z^\pm_i \in Z^\pm(z)$ and $k^\pm \geq 1$.

Let $L_z$ be the horizontal line segment in $Y$ connecting
$\Im(z)$ to $z$.
Note that $\eta^\pm(z)$ is path homotopic within $Y$ to the concatenation of
the vertical path from $J^\pm_0$ to $\Im(z) \in J_0$ with~$L_z$.
Moreover, $\varphi(\eta^\pm(z))$ is the unique geodesic path on $(S,q)$
from $\varphi(J^\pm_0)$ to $\varphi(z)$ in its path homotopy class.
Note that $\varphi(\eta^\pm(z))$ is not necessarily simple on $(S,q)$,
that is, homotopic to an arc with embedded interior.
Observe that every point in $Z$ either lies on $\eta^+(z) \cup \eta^-(z)$,
or to its right. Therefore, the polygonal region $R(z)\subset Y$ bounded by 
$J_0$, $\eta^+(z)$, and $\eta^-(z)$ contains no points of~$Z$ in its interior, nor any
points lying directly to the right of points in $Z$.

Let $T^\pm(z)$ be the triangle formed by taking the convex hull
of $J_0 \cup z^\pm_1$ in $\CC$, and
choose ${T(z) \in \{T^+(z), T^-(z)\}}$ to be a triangle with the
lesser width.
Then
\[R(z) \cap \{w\in \CC ~|~ \Re(w) \leq \width(T(z)) \}\]
is a trapezium
containing $T(z)$, and so it follows that $T(z)$ 
intersects $Z$ precisely at its three corners.
Applying \cref{poly_embed}, the restriction of $\varphi$ to
$T(z)$ is injective, except for possibly at
 the corners.
Note that the corner $z_1 \in \{z_1^+,z_1^-\}$ of $T(z)$ opposite $J_0$
satisfies $\Im(z_1) = 0$ or $h$ only if $\Im(z) = 0$ or $h$ respectively.
We have thus constructed a triangle satisfying the following.

\begin{lem}[Producing triangles] \label{lem:good_tri} 
 For any $z \in Z$,
 the triangle $T = \varphi(T(z))$ on $(S,q)$ has $a$ as one of its sides,
 appears on the side of $a$ from which $\varphi^t$ emanates,
 and satisfies $\height(T) = h$ and $\width(T) \leq \Re(z)$.
 Moreover, the corner $z_1 \in Z$ of $T(z)$ opposite $J_0$ has imaginary part satisfying
 $\Im(z_1) \in (0,h) \cup \{\Im(z)\}$. $\hfill\square$.
\end{lem}

\subsection{Visibility}\label{sec:vis}

We now wish to determine conditions so that
$\varphi(T(z))$ has no sides intersecting any saddle connection in the
simplex $\sigma \in \A(S,q)$ transversely.
Consider the pre-image $\varphi^{-1}(\sigma)$ in
$Y$. This is a countable collection $\J'(\sigma)$ of (maximal)
straight line segments
(or singletons contained in~$Z$, but we may safely ignore these).
Recall that in the construction of $Y$, the (open) horizontal ray starting from any point in $Z$ is deleted. Therefore, if there is a horizontal line segment in $\J'(\sigma)$, it must have non-trivial intersection with $J_0$. Since $a$ is disjoint from any other saddle connection in $\sigma$, any horizontal line segments in $\J'(\sigma)$ must have an endpoint at $J_0^\pm$.
Let $\J = \J(\sigma)$ be the subset of $\J'(\sigma) \setminus \{J_0\}$
consisting of all non-horizontal line segments (see \cref{fig:flow}).
Applying a Poincar\'{e} recurrence argument to the flow~$\varphi^t$
on $(S,q)$ starting from the saddle connection $a$, we deduce that
any trajectory that does not hit a singularity must eventually cross
$a$ transversely. Therefore the set $\J$ is non-empty.

\begin{rem}[Types of line segments] \label{rem:segments}
Each line segment $J \in \J$ is homeomorphic to either an open,
half-open, or a closed interval. Any open endpoint of $J$ must
lie on some deleted horizontal ray, and so is of the form
$z + s$ for some $z\in Z$ and $s > 0$.
Any closed endpoint of~$J$
either belongs to $Z$, or has imaginary part equal to $0$ or $h$.
Note that $J \setminus Z$ maps into the interior of a saddle
connection under $\varphi$. 
\end{rem}

Given a line segment $J \in \J$,
observe that $Y \setminus J$ has two connected components;
let $Y(J) \subset Y$ be the component comprising of all points
lying directly to the right of some point on $J$. Define
\[V = V(\sigma) \coloneqq Y \setminus \bigsqcup_{J \in \J} Y(J). \]

\begin{definition}[Visibility] \label{def:vis}
Call a point $p \in Y$ \emph{visible} (with respect to $\J$ and $\varphi^t$)
if $p \in V$. We also call its image $\varphi(p)$ on~$(S,q)$
\emph{visible} with respect to $\sigma$ and $\varphi^t$.
More generally, we call any subset $U \subseteq V$ and its image \emph{visible}.
\end{definition}

Let us point out here the difference of $p\in Y$ being visible to the informal definition given at the beginning of the section. If $\Im(p) \in (0,h)$, then the formal and informal definition are equivalent as
the horizontal line segment from $\Im(p)$ to $p$ cannot cross any line segment $J \in \J$ transversely if $p\in Y$.
If $\Im(p) \in \{0,h\}$, the characterisation by transverse crossings does not work. Indeed, suppose that $J$ is a line segment in $Y$ with positive slope, having an endpoint at~$J_0^-$.
Then the trajectory of $\varphi^t$ starting from
$J_0^-$ does not cross $J$ transversely for small $t$, but the corresponding points lie in $Y(J)$ and hence are
not visible.

A point in $Z$ having minimal real part is not necessarily visible (see the green point in \cref{fig:flow}).

\begin{lem} \label{cross_visible} 
 If $z \in Z$ is visible then the triangle $T(z)$ is also visible.
\end{lem}

\proof
 We shall prove the contrapositive.
 Suppose that $T(z)$ is not visible.
 Then $T(z)$ contains some point lying directly to the right of some
 non-horizontal line segment $J \in \J$. It follows that $J$ intersects the interior of $T(z)$;
 see \cref{fig:flow}.
 Since $T(z) \subseteq R(z)$, the line segment $J$ also intersects the interior of $R(z)$. Therefore
 $J \cap R(z)$ must be a line segment connecting two points on $\partial R(z)$.
 Note that $J$ cannot intersect $J_0$, except possibly at the endpoints.
 Furthermore, the endpoints of $J \cap R(z)$ cannot both lie on the same path $\eta^\pm(z)$, for
 otherwise $J \cap R(z) \subseteq Y^\pm(z)$, which is disjoint from the interior of $R(z)$.
 Therefore, $J \cap R(z)$ must connect some point on $\eta^+(z) \setminus \{z\}$
 to a point on $\eta^-(z) \setminus \{z\}$. It follows that $J$ must cross the line $L_z$,
 and so $z$ is not~visible. 
\endproof

We now complete the proof of \cref{prop:vis}.
Suppose that for some $0 \leq y \leq h$ and $t_0 > 0$,
the point $\varphi^{t_0}(a(h))$ on $(S,q)$ is a visible singularity
with respect to $\varphi^t$ and $\sigma$.
Then $z = t_0 + \imagi y \in Z$ is a visible point descending to this
singularity. By the above lemma and \cref{lem:good_tri},
taking $T = \varphi(T(z))$ gives a triangle on $(S,q)$
satisfying the required properties.
This completes the proof.

\subsection{Visible singularities exist}

Throughout the rest of this section, we assume that $a$ is a tallest saddle connection of $\sigma$.
Under this assumption, we want to show that some point of $Z$ is visible. We shall prove some
stronger results.

For each $J \in \J$, observe that $Y(J)\subset Y$ is open under the subspace topology.
Therefore $\Im(J \setminus Z) = \Im(Y(J))$ is a connected open subset of $[0,h] \subset \R$ under the subspace topology. We shall always work this
topology in this section.
Given distinct $J,J' \in \J$, observe that $Y(J)$ and $Y(J')$ are either
nested or disjoint, and so the intervals
 $\Im(J \setminus Z)$ and $\Im(J' \setminus Z)$ are either nested or disjoint.
 In particular, if some point on $J'$ lies strictly to the right
 of a point on $J \setminus Z$, then $Y(J') \subset Y(J)$.
 We deduce the following.

\begin{lem}[Visibility and nesting]
 Let $J \in \J$. If some point on $J \setminus Z$ is visible, then $J$ is visible. Furthermore, $J$ is visible if and only if for all $J'\in \J$ such that $J' \neq J$, either $Y(J') \cap Y(J) = \emptyset$
 or $Y(J') \subset Y(J)$ holds. $\hfill\square$
\end{lem}

Write $\vis \subseteq \J$ for the set of non-horizontal visible line segments.
This is precisely the set of $J\in\J$ for which $Y(J)$ is maximal
with respect to inclusion. In particular, $\vis$ is non-empty.

For every $y\in[0,h]$, the trajectory $\varphi^t(a(y))$
terminates at some point in $Z$, or eventually crosses some $J \in \J$
by Poincar\'e recurrence. (In the former case, the trajectory could possibly
cross some $J\in\J$ before terminating.) By the above lemma, the first line
segment in $\J$ crossed by the trajectory (if it exists) must be visible.
Since $Z$ is countable, we deduce that the disjoint~union 
\[\bigsqcup_{J \in \vis} \Im(J \setminus Z)\]
is dense in $[0,h]$. Our goal is to show that this union is not equal to $[0,h]$.

\begin{lem}[Tall saddle connections] \label{biginterval}
 Let $J \in \J$ be a straight line segment.
 If $\Im(J \setminus Z) \supseteq [0,h)$ then one endpoint of~$J$ is at $J_0^- = 0$, and
 the other has imaginary part~$h$.
 If $\Im(J \setminus Z) \supseteq (0,h]$ then one endpoint of~$J$ is at $J_0^+ = \imagi h$,
 and the other is on the real axis.
 In addition, $\Im(J \setminus Z) \neq [0,h]$.
\end{lem}

\proof
 Assume that $\Im(J \setminus Z) \supseteq [0,h)$. 
 Then $J$ is a line segment that connects a point $p$ with $\Im(p)=0$
 to a point $p'$ with $\Im(p') = h$. Note that $p \in J$, while $p'$
 either lies in $J$, or is an open endpoint of $J$.
 If $p \neq 0$, then $\varphi$ maps $J$ into some saddle connection
 $a' \in \sigma$ on $(S,q)$, with $\varphi(p)$ contained in the interior of $a'$.
 But this implies that $\height(a') > \height(J) = h$, contradicting
 the assumption that $a$ is tallest in $\sigma$, and so~$p = 0$.
 By a similar argument, we deduce that if $\Im(J \setminus Z) \supseteq (0,h]$
 then $\imagi h$ is an endpoint of $J$. 
 Finally, if $\Im(J \setminus Z) = [0,h]$ then the endpoints of $J$ are $0$ and $\imagi h$ and hence
 $J$ must coincide with $J_0\notin\J$, a contradiction.
\endproof

It follows that for every $J \in \vis$, the interval $\Im(J\setminus Z)$ has
some open endpoint $y \in [0,h]$. In fact, every point of
$[0,h] \setminus \bigsqcup_{J \in \vis} \Im(J \setminus Z)$ arises in this manner.
Note that if $\Im(J \setminus Z) = (0,h)$, then the closure of $\varphi(J)$ in $(S,q)$ is a
tallest saddle connection of $\sigma$. In particular, if $a$ is strictly tallest then any such $J$ must map onto $a$ under $\varphi$, in which case $a$ is contained in a horizontal cylinder. Therefore, if $a$ is strictly tallest and not contained in a horizontal cylinder then $\Im(J \setminus Z)$ has an open endpoint $y \in (0,h)$ for all $J \in \vis$.
 
By \cref{rem:segments}, there exists some $z \in Z$ satisfying $\Im(z) = y$.
Note that $z$ must be visible, for otherwise $z \in Y(J)$ for some $J\in\J$,
and so $y = \Im(z) \in \Im(Y(J)) = \Im(J\setminus Z)$, a contradiction.
The results in this section can be summarised as follows.

\begin{prop}[Structure of the set of visible line segments] \label{intervals}
 The set $\{\Im(J \setminus Z) ~|~ J \in \vis\}$ is a collection of pairwise disjoint connected
 open sets in $[0,h]$, whose union is dense in $[0,h]$. Furthermore, the complement of this
 union is non-empty, and comprises precisely of all $y \in [0,h]$ for which there exists
 a visible $z \in Z$ satisfying $\Im(z) = y$.
 In particular, if $a$ is strictly tallest and not contained in a horizontal cylinder, then there exists a visible $z\in Z$ where $0 < \Im(z) < h$.
 $\hfill\square$
\end{prop}

\cref{prop:tall} now follows from the above proposition.
We conclude this section with one final result.
 
 \begin{lem}[Visibility after extending simplices]
  Suppose $\sigma \subseteq \sigma'$ are simplices in which $a$
  is a tallest saddle connection. Assume $J \in \vis(\sigma)$ is
  visible. Then either $J \in \vis(\sigma')$, or there exists some
  $J' \in \vis(\sigma')$ such that $\Im(J' \setminus Z) \supseteq \Im(J \setminus Z)$.
  In the latter case, $J'$ maps into some saddle connection in
  $\sigma' \setminus \sigma$ under~$\varphi$.
 \end{lem}
 
 \proof
 Suppose $J$ is not visible with respect to $\J(\sigma')$.
 Then there exists some $J' \in \vis(\sigma')$
 such that $Y(J') \supset Y(J)$, and so 
 $\Im(J' \setminus Z) \supseteq \Im(J \setminus Z)$.
 Observe that $J'\notin\J(\sigma)$, for otherwise~$J$ would not be visible with respect
 to $\J(\sigma)$. Therefore, $\varphi(J')$ is contained in some saddle connection
 in $\sigma' \setminus \sigma$.
 \endproof

\section{Extending triangles}\label{sec:extend}

Recall from \cref{non-fold_flip} that any non-folded (topological) triangle $T$ on $(S,\P)$
has all three sides flippable in any triangulation $\T \in \F(S,\P)$
containing $T$. In contrast, this is far from true in the case of triangulations
on half-translation surfaces (see \cref{ex:bad_triangle}).
Nevertheless, knowing that some sides of a triangle are flippable will help us to detect neighbouring triangles.
Let us begin with an observation that will be relevant in the next example.

\begin{rem}[Flipping within a pentagon]
 Let $T$ be a triangle on $(S,q)$ with sides $a,b,c \in \A(S,q)$.
 Assume $T$ is not $(1,1)$--annular.
 Suppose $\T$ is a triangulation containing $T$ in which both $a$ and~$b$ are flippable.
 Appealing to \cref{codim2}, we deduce that $\T - \{a,b\}$ has an almond pentagon
 as its unique non-triangular region, in which
 $a$ and $b$ form non-intersecting diagonals.
 Consequently, the saddle connections $a'$ and $b'$ obtained by respectively flipping $a$ and $b$
 in $\T$ can only intersect in the interior of $T$.
\end{rem}

\begin{exa}[Triangle with flipping difficulties] \label{ex:bad_triangle}
 Let $T$ be the triangle with sides $a,b,c$ as shown in \cref{fig:bad_triangle}. 
 Observe that $T$ cannot be $(1,1)$--annular.
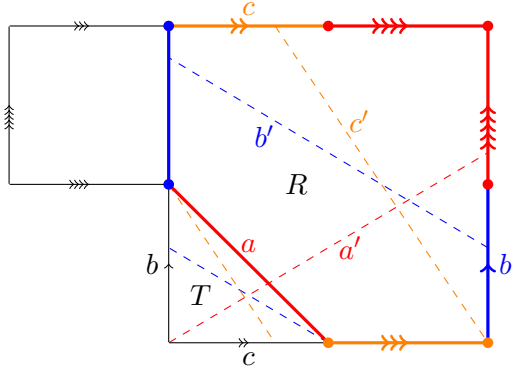
\begin{figure}
  \begin{center}
   \begin{tikzpicture}[scale=0.7]
    \foreach \x in {0,1,2,3} {
      \foreach \y in {0,1,2} {
    \node[circle,fill=none,minimum size=0pt, inner sep=0pt] (A\x-\y) at (3*\x, 3*\y) {};
    }}
     \draw[middlearrow={>>>>}] (A0-1) -- (A1-1);
     \draw[middlearrow={>>>}] (A0-2) -- (A1-2);
     \draw[middlearrow={>>>>>}] (A0-1) -- (A0-2);
     \draw[middlearrow={>}] (A1-0) -- node[left] {$b$} (A1-1);
      \draw[red,dashed] (A1-0) -- node[right] {$a'$} (3*3, 3*1.2) {};
     \draw[blue,dashed] (A2-0) -- (3*1, 3*0.6) {};
     \draw[blue,dashed] (3*3, 3*0.6) -- node[pos=0.7,below] {$b'$} (3*1, 3*1.8) {};
     \draw[orange,dashed] (A1-1) -- (3*5/3, 3*0) {};
     \draw[orange,dashed] (3*5/3, 3*2) -- node[pos=0.3,right] {$c'$} (A3-0) {};
     \draw[middlearrow={>},blue,very thick] (A3-0) -- node[right] {$b$} (A3-1);
     \draw[middlearrow={>>}] (A1-0) -- node[below] {$c$} (A2-0);
     \draw[middlearrow={>>},orange,very thick] (A1-2) -- node[above] {$c$} (A2-2);
     \draw[red,very thick] (A1-1) --  node[above] {$a$} (A2-0);
     \draw[middlearrow={>>>},orange,very thick] (A2-0) node[circle,fill,minimum size=4pt, inner sep=0pt] {} -- (A3-0)
     node[circle,fill,minimum size=4pt, inner sep=0pt] {};
     \draw[blue,very thick] (A1-1) node[circle,fill,minimum size=4pt, inner sep=0pt] {} -- (A1-2)
     node[circle,fill,minimum size=4pt, inner sep=0pt] {};
     \draw[middlearrow={>>>>>},red,very thick] (A3-1) node[circle,fill,minimum size=4pt, inner sep=0pt] {} -- (A3-2)
     node[circle,fill,minimum size=4pt, inner sep=0pt] {};
     \draw[middlearrow={>>>>},red,very thick] (A2-2) node[circle,fill,minimum size=4pt, inner sep=0pt] {} -- (A3-2);
     \node (T) at (3*1.2, 3*0.3) {$T$};
     \node (R) at (3*1.8, 3*1) {$R$};
   \end{tikzpicture}
   \caption{All horizontal and vertical saddle connections have the same length.
   There is a heptagon $R$ on this half-translation surface bounded by the saddle connections indicated
   in thick lines. The saddle connections $a',b',c'$ must intersect one another transversely inside $R$.
   (The gluings are indicated by the number of arrows, not~by~colour.)}
   \label{fig:bad_triangle}
  \end{center}
 \end{figure}
 Let $R$ be the heptagonal region bounded by the thick saddle connections.
 The boundary $\partial R$ is partitioned into six subintervals:
 as one follows the boundary, these intervals cycle through the colours red, orange, and blue,
 while also alternating between being open and closed.
 Any saddle connection $a'$ obtained by flipping~$a$ in any triangulation
 containing $T$ must contain a line segment in $R$ connecting the two red intervals.
 Similarly, any $b'$ or $c'$ obtained by flipping $b$ or $c$ will contain
 a line segment respectively connecting two blue intervals or two orange intervals.
 Then $a',b',c'$ must pairwise intersect in the interior of~$R$.
 By the above remark, we deduce that $T$ has at most one flippable side in any
 triangulation containing $T$.
\end{exa}

The goal of this section is to prove the following. In light of the above example,
this is the strongest result one could hope for in general. Recall that a pentagon
is almond if it has at most one broken diagonal, that is, at most one topological diagonal that cannot be realised by a saddle connection.

\begin{prop}[Extending triangles] \label{prop:extension}
 Suppose $T$ is a triangle on $(S,q)$ with sides ${a,b,c \in \A(S,q)}$. 
 Then there exists a triangulation $\T \supseteq \partial T$ in which $a$
 is flippable, and such that  $\lk(\T \setminus \{a,b\})$ contains $\PG_3$
 as an induced~subgraph.
 \end{prop}

If $b$ is also flippable in $\T$, then the non-triangular region of $\T \setminus \{a,b\}$ is either a $(1,1)$--annulus or an almond pentagon and $\lk(\T \setminus \{a,b\})$ contains
 ~\begin{tikzpicture}[scale = 0.8]
   \tikzstyle{every node}=[draw,circle,fill=black,minimum size=4pt, inner sep=0pt]
    \draw (0,0) node (a) {}
    -- (1,0) node (b) [label=above:$a$] {}
    -- (2,0) node (e) [label=above:$b$] {}
    -- (3,0) node (f) {};
    \end{tikzpicture}~
as a subgraph;
otherwise~$b$ is a non-flippable diagonal of an almond pentagon of $\T \setminus \{a,b\}$, in which case $\lk(\T \setminus \{a,b\})$ is equal to
  \begin{tikzpicture}[scale = 0.8]
   \tikzstyle{every node}=[draw,circle,fill=black,minimum size=4pt, inner sep=0pt]
    \draw (0,0) node (a) [label=above:$a$] {}
    -- (1,0) node (b) [label=above:$b$] {}
    -- (2,0) node (e)  {}
    -- (3,0) node (f) {};
    \end{tikzpicture}~.
 
We can also obtain stronger extension results for certain triangles.
 
\begin{prop}[Extending major triangles] \label{prop:major_ext}
 Suppose $T$ is a triangle on $(S,q)$ with sides $a,b,c \in \A(S,q)$.
 Assume $Q \supset T$ is a strictly convex
 quadrilateral that has $c$ as a diagonal, and satisfying
 $\area(T) \geq \frac{1}{2}\area(Q)$.
 Then the triangulation
 $\T \supseteq \partial T$ in \cref{prop:extension}
 can be chosen so that both $a$ and $c$ are flippable.
 Furthermore, if $T$ is not $(1,1)$--annular, then $\T$ can be chosen
 so that $\partial Q \subseteq \T$.
\end{prop}

\begin{definition}[Major triangle]
Call $T$ a \emph{major triangle} if there exists a strictly convex
quadrilateral $Q \supset T$ such that $\area(T) \geq \frac{1}{2}\area(Q)$;
if this holds, call the side of $T$ that forms a diagonal of $Q$ a \emph{base}
of~$T$. A major triangle may have more than one base.
Any $(1,1)$--annular triangle is major.
\end{definition}

The proofs of \cref{prop:extension} and \cref{prop:major_ext} are given
in \cref{subsec:extension} and \cref{subsec:major_ext} respectively.
In these subsections we shall continue using the following notation as defined in the previous section:
the collection of line segments $\J$, the segment $J_0$ and its endpoints $J_0^\pm$,
the singularities $Z \subset \CC$, the domain $Y\subset\CC$, and the locally isometric embedding
$\varphi \colon Y \to (S,q)$.

\subsection{Extending triangles to almond pentagons or \texorpdfstring{$(1,1)$--annuli}{(1,1)-annuli}} \label{subsec:extension}

Let $T$ be a triangle on $(S,q)$ with sides $a,b,c \in \A(S,q)$.
Apply an $\SL(2,\R)$--deformation to~$(S,q)$ to make $a$ vertical and $c$ horizontal.
This ensures that $a$ and $b$
are both tallest saddle connections in the simplex $\tau = \{a,b,c\} \in \A(S,q)$,
with height $h = \height(T) > 0$.
Let $\varphi^t$ be the horizontal unit-speed flow emanating from $a$, and
flowing away from $T$.
Equip $a$ with a unit-speed parameterisation
$a \colon [0,h] \rightarrow (S,q)$ so that $\varphi^t$ emanates from its right.
There is a unique isometric identification of $T$ with a Euclidean triangle
in~$\CC$ so that $a(y)$ maps to $\imagi y$ for $0 \leq y \leq h$.
Since $T$ appears to the left of $a$, the map
$\varphi \colon Y \rightarrow (S,q)$ can be extended to a locally isometric
embedding defined on $T \cup Y \subset \CC$.

Without loss of generality, we shall assume that the sides $a,b,c$ appear
in anticlockwise order around $T$ in $\CC$, and so $c$ is a line segment
lying on the negative real axis. 
Our goal is to extend~$\tau$ to a triangulation in which $a$ is flippable.
Let us prove a more general statement for any simplex $\sigma \supseteq \{a,b\}$ with tallest saddle connections $a$ and $b$ which characterises the non-flippability of $a$ in all triangulations
$\T \supseteq \sigma$.

\begin{lem}[Awning lemma] \label{lem:awning}
 Let $T$ be a triangle with $\partial T = \{a,b,c\}$ as above and assume
 $a$ is a tallest saddle connection in a simplex $\sigma \supseteq \{a,b\}$. 
 Then the following are equivalent:
 \begin{enumerate}
  \item $a$ is non-flippable in every triangulation $\T \supseteq \sigma$,
  \item There is exactly one visible singularity $z \in Z$ with respect to $\vis(\sigma)$,
  and this point satisfies $\Im(z) = 0$,
  \item There is a unique visible non-horizontal line segment $J \in \vis(\sigma)$,
  and $J$ has one endpoint at $J_0^+ = \imagi h \in Y$ and the other on the real axis.
 \end{enumerate}
\end{lem}

\proof
We proceed via a cycle of implications.
\begin{itemize}
 \item (i) $\implies$ (ii): Suppose there exists a visible singularity
 $z\in Z$ satisfying $\Im(z) > 0$. Then by \cref{prop:vis}, there exists a triangle
 $T' = T(z) \subset Y$ meeting $a$ on the right, of $\height(T') = h$, such that its
 corner $z'\in Z$ opposite $a$ satisfies $\Im(z') > 0$.
 Gluing $T$ and $T'$ along $a$ produces a strictly convex quadrilateral,
 and so $a$ can be flipped in any triangulation containing $\sigma \cup \partial T'$.
 
 \item (ii) $\implies$ (iii): 
 Suppose $z\in Z$ is the unique visible singularity, and that $\Im(z) = 0$.
 Then by \cref{intervals}, we have
 \[\bigsqcup_{J\in\vis(\sigma)} \Im(J\setminus Z) = (0,h].\]
 Since this is a disjoint union of open intervals (under the subspace
 topology on $[0,h]\subset \R$), there exists only one visible non-horizontal line segment
 $J \in \vis(\sigma)$. In particular, we have $\Im(J\setminus Z) = (0,h]$.
 Applying \cref{biginterval}, we deduce that
 $J$ has one endpoint at $J_0^+$, and the other on the real axis as desired.
  
 \item (iii) $\implies$ (i): Let $J$ be as given in Condition (iii), and let $a'\in \sigma$ be its image on~$(S,q)$. Note that there are no points of $Z$ lying in the interior of the region $R \subset Y$ bounded by $J_0$, $J$, and the real axis.
 Suppose $\T \supseteq \sigma \supseteq \{a,b\}$ is a triangulation in which $a$ is flippable.
 Let $d$ be the diagonal obtained by flipping $a$ in $\T$.
 Since $d$ cannot intersect $b$, its pre-image $\varphi^{-1}(d) \subset T \cup Y$ contains a straight line segment with positive slope intersecting the interior of $a$ (see \cref{fig:awning}).
 Since $R$ has no interior singularities, this line must also cross $J$ transversely. This implies that $d$ intersects $a' \in \T \setminus \{a\}$, a contradiction.
\end{itemize}
  
\begin{figure}
 \begin{center}
  \begin{tikzpicture}[scale=1]
    \tikzset{dot/.style={draw,shape=circle,fill=black,scale=0.3}}
    \draw (2,0) -- node[above left] {$b$} (4,2);
    \draw[gray, dashed] (4,2) -- (8,2);
    \draw[gray, dashed] (7,0) -- (8,0);
    \draw[red,thick] (4,2) -- node[above right,pos=0.7] {$J$} (7,0);
    \draw (2,0) -- node[below] {$c$} (4,0);
    \draw[dashed,blue] (2.8,0) -- node[below,pos=0.6] {$d$} (5.8,1.4);
    \draw (4,2) node[dot] {} node[above] {$J_0^+$} -- node[right,pos=0.4] {$a$} (4,0) node[dot] {} node[below] {$J_0^-$};
    \draw (4,0) -- (6,0) node[dot] {};
    \draw[thick, dotted] (6,0) node[below] {$z$} -- (7,0);
    \node (A) at (3.5,0.9) {$T$};
    \node (R) at (5.2,0.5) {$R$};
  \end{tikzpicture}
  \caption{The line segment $J$ is an awning for $a$ with respect to $\J$ and
  the flow $\varphi^t$. The unique visible singularity $z\in Z$ may possibly
  lie on an endpoint of $J$.}
  \label{fig:awning}
 \end{center}
\end{figure}
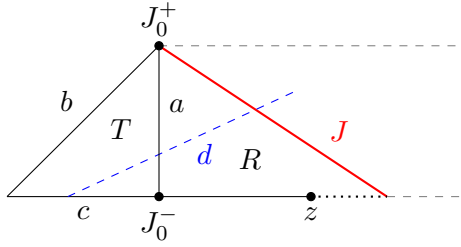
  
This completes the proof.
\endproof

\begin{definition}[Awning]
An \emph{awning} for~$a$ (with respect to $\J$ and the flow $\varphi^t$) is a visible non-horizontal line segment $J \in \vis(\sigma)$, with one endpoint at $J_0^+ = \imagi h \in Y$ and the other on the real axis (as in Condition~(iii) in the above lemma).
If $J$ is an awning for $a$, then $\varphi(J)$
maps into a tallest saddle connection $a'$ of $\sigma$, 
which we shall also refer to as an \emph{awning}. 
\end{definition}

Note that any awning $a'$ for $a$ cannot be parallel to $a$, and so $a'\neq a$.
Furthermore, $a$ and $a'$ must have the same height since $a$ is tallest in $\sigma$.

Let us now return our attention to \cref{prop:extension}. For this situation, we choose $\sigma = \tau$.
We will construct a triangulation that contains $\tau$ and in which $a$ is flippable.
Any awning $a'\in\tau$ is a tallest saddle connection with negative slope.
But the tallest saddle connections in $\tau = \{a,b,c\}$ are $a$ and~$b$, neither of which
have negative slope, and so $a$ does not have any awnings.
By the above lemma, there exists a visible singularity $z \in Z$
such that $\Im(z) > 0$.
This gives a visible triangle $T' = T(z)$ which can be glued to $T$
along $a$ to form a strictly convex quadrilateral~$Q$.
Moreover, $\height(T') = \height(a)$.
Therefore, $a$ and $b$ remain tallest in the simplex
$\tau' = \partial T \cup \partial T' \supset \tau$. 

If $b$ is also a side of $T'$, then gluing $Q$ along the two copies
of $b$ yields a $(1,1)$--annulus in which $a$ and $b$ are both
transverse arcs. This gives a link as desired.

Now suppose $b$ is not a side of $T'$.
Consider the horizontal unit-speed flow emanating from~$b$
away from $T$. By Propositions \ref{prop:tall} and \ref{prop:vis},
there exists a triangle $T''$ meeting $Q$ along $b$, of
$\height(T'') = h$, and such that $\partial T''$ has no transverse
intersections with $\tau'$ (see~\cref{fig:pentagon_flow}).
The isometric embedding $T \cup Y \rightarrow \CC$ can be uniquely
extended to an isometric embedding $T'' \cup T \cup Y \rightarrow \CC$.
Let $z' \in \CC$ be the corner of $T''$ opposite $b$.
Note that $\Re(z') < 0$ and $0 \leq \Im(z') \leq h$.

\begin{figure}
 \begin{center}
    \begin{tikzpicture}[scale=0.6]
   \draw (0,0) -- node[below] {$c$} (2,0) -- (4.5,2) node[right] {$z$} -- (2,4)
    -- (-2,1) node[left] {$z'$} -- (0,0) -- node[left] {$b$} (2,4) -- node[right] {$a$}  (2,0);
   \draw[red,dashed] (0,0) -- (4.5,2);
   \draw[red,dashed] (2,0) -- (-2,1);
   \node (A) at (1.5,1.9) {$T$};
   \node (B) at (3.2,2.2) {$T'$};
   \node (C) at (-0.2,1.5) {$T''$};
  \end{tikzpicture}
     \begin{tikzpicture}[scale=0.6]
   \draw (0,0) -- node[below] {$c$} (2,0) -- (4.5,2) node[right] {$z$} -- (2,4)
    -- (-3,0) node[left] {$z'$} -- (0,0) -- node[left] {$b$} (2,4) -- node[right] {$a$}  (2,0);
   \draw[red,dashed] (0,0) -- (4.5,2) -- (-3,0);
   \node (A) at (1.5,1.9) {$T$};
   \node (B) at (3.2,2.2) {$T'$};
   \node (C) at (-0.2,1.5) {$T''$};
  \end{tikzpicture}
  \caption{The two cases corresponding to when $\Im(z') > 0$ or $\Im(z') = 0$.
  In either case, the pentagon $P$ formed by gluing $T$, $T'$, and $T''$ along
  $a$ and $b$ is almond.}
  \label{fig:pentagon_flow}
 \end{center}
\end{figure}
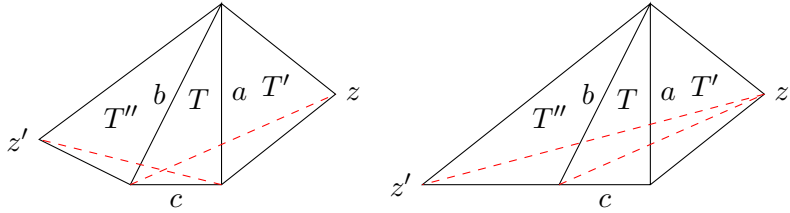

Let $P$ be the pentagon formed by gluing $Q$ and $T''$ along $b$.
Note that $a$, $b$, and the diagonal of~$Q$ obtained by flipping $a$ are all diagonals of $P$.
If $\Im(z') > 0$, then $b$ is also flippable in $P$ which implies that $P$ is an almond pentagon. Therefore, $\lk(\T \setminus \{a,b\})$ is either $\CG_5$ (if $P$ is strictly convex) or
 ~\begin{tikzpicture}[scale = 0.8]
  \tikzstyle{every node}=[draw,circle,fill=black,minimum size=4pt, inner sep=0pt]
  \draw (0,0) node (a) {}
    -- (1,0) node (b) [label=above:$a$] {}
    -- (2,0) node (e) [label=above:$b$] {}
    -- (3,0) node (f) {};
 \end{tikzpicture}~
(if $P$ has a broken diagonal, occurring precisely when $\Im(z')=h$). If $\Im(z') = 0$,
the straight line segment in $\CC$ connecting $z'$ to $z$ gives a fourth diagonal of $P$ and so $P$ is an almond pentagon. In this case, $\lk(\T \setminus \{a,b\})$ is equal to
 ~\begin{tikzpicture}[scale = 0.8]
  \tikzstyle{every node}=[draw,circle,fill=black,minimum size=4pt, inner sep=0pt]
  \draw (0,0) node (a) [label=above:$a$] {}
    -- (1,0) node (b) [label=above:$b$] {}
    -- (2,0) node (e)  {}
    -- (3,0) node (f) {};
 \end{tikzpicture}.

This completes the proof of \cref{prop:extension}.

\begin{rem}[Specific link after extending triangles] \label{rem:link_extension}
 From the above proof, we see that $\lk(\T \setminus \{a,b\})$ contains
 ~\begin{tikzpicture}[scale = 0.8]
  \tikzstyle{every node}=[draw,circle,fill=black,minimum size=4pt, inner sep=0pt]
  \draw (0,0) node (a) [label=above:$a$] {}
    -- (1,0) node (b) [label=above:$b$] {}
    -- (2,0) node (e)  {}
    -- (3,0) node (f) {};
 \end{tikzpicture}~
 unless $\Im(z') = h$.
 In this case, if $z'$ is the only visible singularity (seen from $b$), then $b$ is contained in a horizontal cylinder $C$ of height $h$. The lower boundary of $C$ can only contain one singularity and, moreover, $C$ contains at least three disjoint saddle connections of height $h$. Hence, $T$ is a $(m,1)$--annular triangle for some $m \geq 2$.
\end{rem}

\subsection{Extending major triangles} \label{subsec:major_ext}

In this section, we assume that $T$ is a major triangle with sides
$a,b,c\in\A(S,q)$, having~$c$ as a base.
Thus, there exists another triangle $T'$ having $c$ as a side such that gluing $T$ and $T'$ along~$c$ yields a strictly convex quadrilateral $Q$.
Apply an $\SL(2,\R)$--deformation to make~$c$ horizontal and $a$ vertical.
Thus, $\height(T) \geq \height(T')$ and so
$a$ and $b$ are tallest saddle connections in~$\partial Q \cup \{c\}$.
Note that $c$ is flippable in any triangulation containing
$\partial Q \cup \{c\}$.

As in the previous subsection, we isometrically identify $T$ with a triangle
in $\CC$ so that $a$ lies on the positive imaginary axis, $c$ on the negative
real axis, with their common corner at the origin.
The triangle $T'$ can be uniquely isometrically identified with a Euclidean
triangle in the third quadrant of $\CC$, so that it glues to the copy
of $T$ in $\CC$ along $c$ with the correct orientation.
The map $\varphi \colon Y \rightarrow (S,q)$ extends to a locally isometric embedding
on $\varphi \colon T \cup T' \cup Y \rightarrow (S,q)$, which has a lift
$\tilde{\varphi}$ to the universal cover. 

Our goal is to prove the existence of a triangulation containing $\partial T$
in which both $a$ and~$c$ are flippable.

\begin{lem}[Major triangles and flippability] \label{major_flip}
 Suppose $a$ is non-flippable in
 every triangulation $\T \supseteq \partial Q \cup \{c\}$.
 Then there exists a $(1,1)$--annulus $A$ 
 in which $a$ and $c$ are transverse arcs, and $b$ is contained in $\partial A$.
 In particular, $a$ and $c$ are flippable in any triangulation
 containing $\partial A \cup \{a,c\}$.
 In this case, we also have
 $\height(T) = \height(T')$.
\end{lem}

\proof
By \cref{lem:awning}, there exists an awning $J\in\J$ for $a$ which descends
to a tallest saddle connection $a' \in \sigma = \partial Q \cup \{c\}$
via $\varphi$. Since $a'$ has negative slope, it cannot coincide with
$a$, $b$, nor $c$. Furthermore, the side $b' $ of $Q$
opposite $b$ has positive slope since $Q$ is strictly convex, and so
does not coincide with $a'$ neither.
Therefore, $a'$ must be the side of~$Q$ opposite $a$.
We shall identify $a'$ with its copy in $\partial T' \subset \CC$, lying in the
third quadrant. Furthermore, we have $\height(T') = \height(a')=\height(T)$.

Since $a'$ and $J$ have the same image on $(S,q)$, there is a deck transformation
$\tilde{g}$ of the universal cover sending $\tilde{\varphi}(a)$ to $\tilde{\varphi}(J)$.
We can realise $\tilde{g}$ in local co-ordinates: there exists a map
$g \colon \CC \rightarrow \CC$, of the form $g(w) = \pm w + w_0$ for some $w_0 \in \CC$,
such that $(\tilde{g} \circ \tilde{\varphi}) (w) = (\tilde{\varphi} \circ g)(w)$
whenever $w, g(w) \in T \cup T' \cup Y$. Refer to \cref{fig:major_flip}.

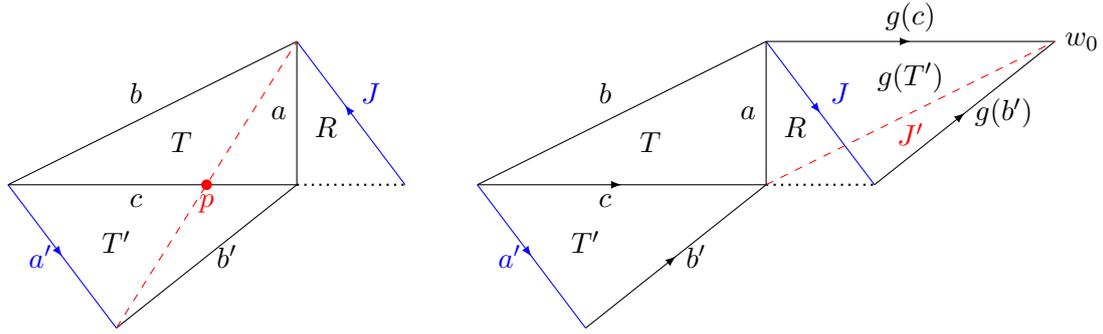
\begin{figure}
 \begin{center}
  \begin{tikzpicture}[scale=0.95]
  \begin{scope}
    \draw (0,0) -- node[above left] {$b$} (4,2);
    \draw[middlearrow={latex},blue] (0,0) -- node[left] {$a'$} (1.5,-2);
    \draw[middlearrow={latex reversed},blue] (4,2) -- node[above right] {$J$} (5.5,0);
    \draw (0,0) -- node[below left] {$c$} (4,0);
    \draw (1.5,-2) -- node[right] {$b'$} (4,0);
    \draw (4,2) -- node[left] {$a$} (4,0);
    \draw[thick, dotted] (4,0) -- (5.5,0);
    \node (A) at (2.4,0.6) {$T$};
    \node (B) at (1.5,-0.8) {$T'$};
    \node (R) at (4.4,0.8) {$R$};
    \draw[red,dashed] (1.5,-2) -- (2.75,0) node[circle,fill,minimum size=4pt, inner sep=0pt] {} node[below] {$p$} -- (4,2);
   \end{scope}
   \begin{scope}[xshift=6.5cm]
    \draw (0,0) -- node[above left] {$b$} (4,2);
    \draw[middlearrow={latex},blue] (0,0) -- node[left] {$a'$} (1.5,-2);
    \draw[middlearrow={latex},blue] (4,2) -- node[above right] {$J$} (5.5,0);
    \draw[middlearrow={latex}] (0,0) -- node[below left] {$c$} (4,0);
    \draw[middlearrow={latex}] (4,2) -- node[above] {$g(c)$} (8,2) node[right]{$w_0$};
    \draw[middlearrow={latex}] (1.5,-2) -- node[right] {$b'$} (4,0);
    \draw[middlearrow={latex}] (5.5,0) -- node[right] {$g(b')$} (8,2);
    \draw (4,2) -- node[left] {$a$} (4,0);
    \draw[thick, dotted] (4,0) -- (5.5,0);
    \node (A) at (2.4,0.6) {$T$};
    \node (B) at (1.5,-0.8) {$T'$};
    \node (C) at (6,1.5) {$g(T')$};
    \node (R) at (4.4,0.8) {$R$};
    \draw[red,dashed] (4,0) -- node[below] {$J'$} (8,2);
   \end{scope}
  \end{tikzpicture}
  \caption{The two cases in the proof of \cref{major_flip}.
  Left: $a'$ and $J$ are related by a rotation by~$\pi$ about $p$.
  Right: $a'$ and $J$ are related by a translation $g$. There may
  be singularities on the horizontal boundary of $R$ (black dotted line).}
  \label{fig:major_flip}
 \end{center}
\end{figure}

We claim that $g$ must be a translation; in which case we have
$w_0 = \length(c) + \imagi\length(a)$. Suppose otherwise, for a contradiction.
The unique half-translation of $\CC$ mapping $a'$ to $J$
is a rotation by $\pi$, with centre at the intersection point $p$ of the two
diagonals of $Q = T \cup T'$ in $\CC$. Then $g(T')$ and $T$ have overlapping but not coinciding
interiors. But this implies that $T$ and~$T'$ have overlapping interiors
(as triangles on $(S,q)$), which is impossible.

Let $R \subset \CC$ be the region bounded by $J_0$, $J$, and the real axis.
Since the awning $J$ is visible, the interior of $R$ is also visible
and hence contains no singularities.
Now consider the triangle $g(T') \subset \CC$. This triangle has
$J$ as one of its sides, and the point $w_0$ as the opposite corner.
Furthermore, $R$ and $g(T')$ can be glued along $J$ to form
a strictly convex quadrilateral $Q' \subset \CC$ with no interior singularities.
Note that $Q'$ is not necessarily
a subset of $Y$, since the unique visible singularity
$z\in Z$ could lie strictly to the left of the endpoint of $J$ on the real axis.
However, the points in $Q'$ strictly to the right of $z$ are the only points in $Q\setminus Y$.
Therefore, the line segment $J'$ connecting $0$ to $w_0$ in $\CC$
lies in~$Y$.
The triangle $T'' \subset Q' \cap Y$ with sides $a$, $g(c)$, and $J'$
descends to a triangle on $(S,q)$ with $a$ and $c$
as two of its sides. This triangle cannot coincide with $T$ on $(S,q)$,
since its interior overlaps that of $T'$.
Gluing the triangles $T$ and $T''$ along $a$ and $c$
gives the desired $(1,1)$--annulus $A$ on $(S,q)$.
\endproof

This completes the proof of \cref{prop:major_ext}.
We conclude this section with a lemma which will be useful in Sections \ref{sec:flip_pairs_bad_kites} and \ref{sec:triangle}.

\pagebreak

\begin{lem}[Annular triangles and flippability] \label{lem:ann_flip}
 Let $T$ be a triangle on $(S,q)$ with sides $a,b,c \in \A(S,q)$. Suppose that $A$ is an $(m,1)$--annulus that contains $b$ and $c$ as transverse arcs, where $m\geq 1$.
 Let $\sigma = \partial A \cup \{b,c\}$. Then there exists a triangulation $\T \supseteq \sigma$ in which $a$, $b$, and~$c$ are flippable.
 If $A$ is a $(1,1)$--annulus, then $\T$ can be chosen such that
 every saddle connection in $\sigma$ is flippable.
 
\end{lem}

\proof
Observe that $b$ and $c$ are flippable in any triangulation $\T \supseteq \sigma$.
As before, assume $a$ is vertical and $c$ is horizontal.
Without loss of generality, suppose~$b$ has positive slope.
Note that~$a$ and $b$ are tallest in $\sigma$.

Consider the horizontal straight-line flow emanating from
$a$ away from $A$.
By \cref{lem:awning}, the only obstruction for the flippability of $a$
in every triangulation $\T \supseteq \sigma$ is the presence of an awning $J \in \J$ for $a$, which maps to a tallest
saddle connection of $\sigma$ via $\varphi$. Such an awning must have negative slope,
but this does not hold for any saddle connection of $\sigma$.
Therefore, $a$ does not have an awning. Applying \cref{prop:extension},
there exists a triangle $T''$ of $\height(T'') = \height(a)$
that can be glued to $A$ along $a$ so that $a$ is flippable in
any triangulation containing $\sigma' = \sigma \cup \partial T''$.

Assume now that $A$ is a $(1,1)$--annulus and let $a$ and $d$ be the saddle connections forming~$\partial A$. Then $d$ is also a tallest saddle connection in $\sigma$.
As $a$ is flippable in $\sigma'$, $\partial T''$ cannot contain any saddle
connection of negative slope that is tallest in $\sigma'$.
Consider now the horizontal straight-line flow emanating from
$d$ away from $A$. As before, an awning for $d$ must have negative slope.
But such an awning cannot exist in $\sigma'$.
Therefore, there exists a triangulation $\T \supseteq \sigma'$ in
which $d$ is also flippable.
\endproof

\section{Flip pairs and convex quadrilateral boundaries} \label{sec:flip_pairs_bad_kites}

In this section, we attempt to detect the boundary of a strictly convex quadrilateral $Q$
given its diagonals $c$ and $d$, using only the combinatorial properties of $\A(S,q)$.
First, we define the sets of \emph{barriers} $\B(c,d)$ and of \emph{flippable barriers} $\FB(c,d)$ in $\A(S,q)$ satisfying 
$\FB(c,d) \subseteq \partial Q \subseteq \B(c,d)$. However, these inclusions may be strict as we will show in \cref{ex:oct}. The main result of this section is the careful
definition of a set $\KFB(c,d) \subseteq \partial Q$, which we call \emph{kite-or-flippable barriers},
that is guaranteed to contain at least three sides of $Q$ (see \cref{cor:kfb}).
This will play an important role in our proof of the Triangle Test in the following section.

\subsection{Barriers}

Recall from \cref{codim1} that a simplex $\sigma \in \A(S,q)$ has link
$\lk(\sigma) \cong \NG_2$
if and only if the unique non-triangular region of $S - \sigma$ is a strictly convex
quadrilateral. Moreover, the two diagonals of the quadrilateral are precisely
the vertices of $\lk(\sigma)$.

\begin{definition}[Flip pair] \label{def:fp}
 The set of \emph{flip pairs} of $\A(S,q)$ is
 \[\FP(\A(S,q)) \coloneqq \{\lk(\sigma) ~|~ \sigma \in \A(S,q), \lk(\sigma) \cong \NG_2 \}.\]
\end{definition}

Given a flip pair $\{c, d\}$, there exists a (unique) strictly convex quadrilateral
$Q = Q(c,d)$ in which they form the diagonals.
We would like to recover the simplex $\partial Q \in \A(S,q)$ of boundary saddle connections 
purely combinatorially from $\{c, d\}$. However, this is not so straightforward.
As a first approximation, we define the following.

\begin{definition}[Barrier] \label{def:barriers}
The set of \emph{barriers} of a flip pair $\{c,d\} \in \FP(\A(S,q))$ is
\[\B(c,d) \coloneqq \{\gamma \in \A(S,q) ~|~ \gamma \neq c,d \text{ and } \forall \delta \in \A(S,q) \text{ with } \delta \pitchfork \gamma \text{ we have } \delta \pitchfork c \textrm{ or } \delta \pitchfork d \}.\] 
\end{definition}

Observe that any saddle connection $\gamma$ intersecting $\partial Q$ transversely
must necessarily intersect at least one of its diagonals, and so
$\partial Q \subseteq \B(c, d)$.
However, the barriers do not always coincide with $\partial Q$; see \cref{ex:oct} below.
This disparity can be characterised as follows.

\begin{lem}[Cordons appear as barriers] \label{barriers}
The saddle connections in $\B(c, d) \setminus \partial Q$ are 
precisely the cordons of $\partial Q$.
Consequently, any triangulation $\T$ containing $\partial Q$
also contains $\B(c,d)$.
\end{lem}

\proof
Suppose $\gamma \notin \partial Q$ is a saddle connection.
Recall from \cref{def:cordon} that $\gamma$ is a cordon of~$\partial Q$
if and only if $\gamma \in \lk(\partial Q)$ and
$\gamma$ does not intersect any saddle connection in $\lk(\partial Q)$.
This is equivalent to saying that whenever some $\delta \in \A(S,q)$ intersects $\gamma$,
then $\delta$ also intersects some saddle connection of $\partial Q$, and hence
at least one of $c$ or $d$. By definition, this holds precisely when $\gamma \in B(c,d) \setminus \partial Q$.
\endproof

To exclude the cordons of $\partial Q$, we use the fact that they cannot
be flipped in any triangulation containing $\partial Q$. This motivates a
second approximation.

\begin{definition}[Flippable barrier] \label{def:flip_bar}
The set of \emph{flippable barriers} of a flip pair $\{c,d\}$ is
\[\FB(c, d) = \{\gamma \in B(c, d) ~|~ \exists \textrm{ a triangulation } \T\supseteq \B(c, d) \textrm{ in which } \gamma \textrm{ is flippable} \}. \]
\end{definition}

This set can be characterised purely combinatorially, as the flippability
condition can be restated as $\lk(\T \setminus \{\gamma\}) \cong \NG_2$.
Since a cordon for $\partial Q$ cannot be flipped in any triangulation
containing $\partial Q$, we deduce that $\FB(c,d) \subseteq \partial Q$.

\begin{lem}[Over--under]
Given any flip pair $\{c,d\}$, we have
$\FB(c,d) \subseteq \partial Q(c,d) \subseteq \B(c,d)$. $\hfill\square$
\end{lem}

Unfortunately, these inclusions may be strict, as the following example demonstrates.

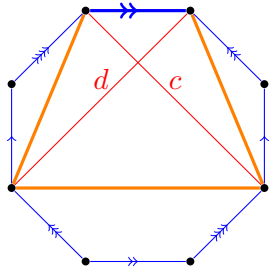
\begin{figure}
  \begin{center}
   \begin{tikzpicture}[scale=0.6]
    \foreach \x in {0,1,...,7} {
    \node[circle,fill=black,minimum size=3pt, inner sep=0pt] (A\x) at (22.5+45*\x:3) {};
    }
    \draw[red] (A1) -- node[above] {$d$} (A4);
    \draw[red] (A2) -- node[above] {$c$} (A7);
    \draw[orange,very thick, densely dashed] (A1) -- (A7) -- (A4) -- (A2);
    \draw[blue,middlearrow={>}] (A7) -- (A0);
    \draw[blue,middlearrow={>}] (A4) -- (A3);
    \draw[blue,middlearrow={>>}] (A2) -- (A1);
    \draw[blue,very thick] (A2) -- (A1);
    \draw[blue,middlearrow={>>}] (A5) -- (A6);
    \draw[blue,middlearrow={>>>}] (A0) -- (A1);
    \draw[blue,middlearrow={>>>}] (A5) -- (A4);
    \draw[blue,middlearrow={>>>>}] (A6) -- (A7);
     \draw[blue,middlearrow={>>>>}] (A3) -- (A2);
   \end{tikzpicture}
   \caption{The flip pair $\{c,d\}$ forms the diagonals of a strictly convex quadrilateral $Q(c,d)$,
   whose sides are indicated by thick line segments. The saddle connections in $\FB(c,d)$ are
   given as dashed orange lines, while those in $\B(c,d) \setminus \FB(c,d)$ are in blue.}
   \label{fig:octagon}
  \end{center}
 \end{figure}

 \begin{exa}[Non-flippable barriers] \label{ex:oct}
 \cref{fig:octagon} shows a genus $2$ translation surface formed by gluing opposite sides of a regular octagon. The flip pair $\{c,d\}$ forms the diagonals of a strictly convex quadrilateral $Q(c,d)$. The flippable barriers are indicated as dashed orange lines, while the non-flippable barriers are given in blue.
 There are three cordons of $\partial Q(c,d)$; these are the non-horizontal blue saddle connections.
 In this example, we have the strict inclusions $\FB(c,d) \subsetneq \partial Q(c,d) \subsetneq \B(c,d)$.
\end{exa}

\subsection{Cylindrical kites}
 
We now show that it is possible for a flip pair to only have two flippable barriers,
and prove that this happens precisely when the flip pair satisfies the following conditions.

\begin{definition}[Cylindrical kite]
 Suppose $Q$ is a strictly convex quadrilateral, with diagonals $c$ and~$d$,
 that can be made into a (non-rhombus)
 Euclidean kite under $\SL(2,\R)$--deformations; this occurs if and only
 if exactly one of its diagonals bisect the other. Without loss of generality,
 suppose that $c$ is horizontal and~$d$ is vertical, and that $c$ bisects $d$ (see \cref{fig:bad_kite}).
 Let $T$ and $T'$ be the triangles obtained by cutting $Q$ along $c$.
 We say $Q$ is a \emph{cylindrical kite} if there exists a horizontal Euclidean cylinder~$C$ such
 that:
 \begin{itemize}
  \item $C$ contains $T$ and $T'$, and thus all saddle connections 
  of $\partial Q$ are transverse arcs of~$C$,
  \item $C - \partial Q$ has exactly four regions, among them $T$ and $T'$, and 
  \item for each such region $R$ with $T \neq R \neq T'$, we require the two non-horizontal sides to be adjacent in $R$ and opposite in $Q$.
 \end{itemize}
 We also call a flip pair $\{c,d\}$ a \emph{cylindrical kite pair} if the quadrilateral
 $Q(c,d)$ is a cylindrical kite.
 \end{definition}
 
 \begin{rem}\label{rem:kite_pair}
  If $\{c,d\}$ is a cylindrical kite pair with a cylinder $C$ as above, then exactly one of~$c$ or $d$ belongs to $\partial C$.
 \end{rem}

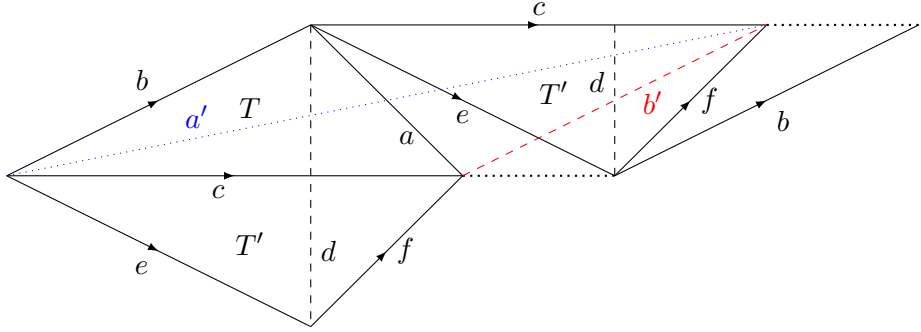
\begin{figure}
 \begin{center}
  \begin{tikzpicture}[scale=1.0]
   \draw[dotted, blue] (0,0) -- node[above] {$a'$} (5,1) -- (10,2);
   \draw[middlearrow={latex}] (0,0) -- node[above left] {$b$} (4,2);
   \draw[middlearrow={latex}] (8,0) -- node[below right] {$b$} (12,2);
   \draw[middlearrow={latex}] (0,0) -- node[below left] {$e$} (4,-2);
   \draw[middlearrow={latex}] (4,2) -- node[below] {$e$} (8,0);
   \draw[middlearrow={latex}] (0,0) -- node[below left] {$c$} (6,0);
   \draw[middlearrow={latex}] (4,2) -- node[above] {$c$} (10,2);
   \draw[middlearrow={latex}] (4,-2) -- node[right] {$f$} (6,0);
   \draw[middlearrow={latex}] (8,0) -- node[right] {$f$} (10,2);
   \draw (4,2) --(5,1) -- node[left] {$a$} (6,0);
   \draw[dashed] (4,2) -- (4,0) -- node[right] {$d$} (4,-2);
   \draw[dashed] (8,0) -- (8,0.5) -- node[left] {$d$} (8,2);
   \draw[thick, dotted] (6,0) -- (8,0);
   \draw[thick, dotted] (10,2) -- (12,2);
   \node (A) at (3.2,0.9) {$T$};
   \node (B) at (3.2,-0.9) {$T'$};
   \node (C) at (7.2,1.1) {$T'$};
   \draw[dashed, red] (6,0) -- (7, 0.5) -- node[below] {$b'$} (10,2);
  \end{tikzpicture}
  \caption{A cylindrical kite $Q(c,d)$ with sides $a$, $b$, $e$, and $f$.
  The triangles $T$ and $T'$ are obtained by cutting $Q(c,d)$ along $c$, and are
  contained in a horizontal cylinder~$C$.
  There may be singularities on the horizontal dotted lines
  lying on the boundary of~$C$. The saddle connections $b'$
  (in red, dashed) and $a'$ (in blue, dotted)
  appear in the proof of \cref{prop:kb}.}
  \label{fig:bad_kite}
 \end{center}
\end{figure}
 
 \begin{lem}[At most two flippable barriers in cylindrical kites] \label{lem:bad_kite}
  If $\{c,d\} \in \FP(\A(S,q))$ is a cylindrical kite pair then $\# \FB(c,d) \leq 2$.
 \end{lem}

 \proof
 Let $Q = Q(c,d)$ be a cylindrical kite, and assume that $c$ is horizontal, $d$ is vertical,
 and that $c$ bisects $d$. Label the sides of $Q$ by $a,b,e,f$ in cyclic order
 so that $a,b,c$ bounds a triangle~$T$, the side $a$ has negative slope,
 and $\width(a) < \width (b)$ (see \cref{fig:bad_kite}).
 Let $T'$ be the triangle bounded by $c,e,f$. Let $C$ be a horizontal cylinder
 which fulfils the conditions for $Q$ to be a cylindrical kite.
 The triangles $T$ and $T'$ form exactly two of the regions of $C - \partial Q$.
 
 Let $R$ be the region of $C - \partial Q$ meeting $T$ along $a$. 
 Then $a$ and $e$ form the adjacent non-horizontal sides of $R$.
 It follows that all other sides of $R$ lie on one boundary component of~$C$.
 Since each region of $C - \partial Q$ is a Euclidean planar polygon,
 $R$ is isometric to a Euclidean triangle,
 with its horizontal ``side'' possibly subdivided into several saddle connections. 
 
 Now consider the horizontal straight-line flow emanating from $a$ away from $T$.
 Any trajectory starting from the interior of $a$ will cross $R$ and then hit
 the interior of $e$. Appealing to \cref{intervals}, we deduce that $e$ is the only
 visible line segment with respect to the flow and~$\partial Q$.
 Since $e$ shares an endpoint with $a$ at the corner of $T$ opposite~$c$,
 it follows that $e$ is an awning for $a$. Therefore, by \cref{lem:awning},
 $a$ cannot be flippable in any triangulation $\T \supset \partial Q$, hence
 $a \notin \FB(c,d)$.
 
 Using a similar argument, we deduce that $f\notin \FB(c,d)$, and so $\# \FB(c,d) \leq 2$.
\endproof

However, there will always be at least two flippable barriers, and at least three if we are not in the cylindrical kite case.

\begin{lem}[Flip pair boundaries]
 Let $\{c, d\} \in \FP(\A(S,q))$ be a flip pair. Then $\#\FB(c,d) \geq 2$.
 Furthermore, equality occurs if and only if $\{c,d\}$ is a cylindrical kite pair.
 
 \begin{proof}
  Apply an $\SL(2,\R)$--deformation to make $c$ horizontal and $d$ vertical.
  If $Q$ has a pair of opposite sides identified, then it forms a
  $(1,1)$--annulus by gluing those sides. Applying \cref{lem:ann_flip},
  there exists a triangulation containing $\partial Q$ in which every saddle
  connection of~$\partial Q$ is flippable.
  In this case, we have $\FB(c,d) = \partial Q$.
  Note that $\# \partial Q = 3$, as identifying the second pair of opposite sides would yield a torus which we do not consider here. 
  
  Let us now assume $Q$ has four distinct sides $a$, $b$,
  $e$, $f$ appearing in the given cyclic order, with $a,b,c$ bounding
  a triangle $T$ in $Q$. Suppose that $b,e,d$ bound a triangle $T''$ in~$Q$.
  Without loss of generality, assume that both $T$ and $T''$ take up at
  least half the area of $Q$ (and are hence major), and that
  $a$ and $e$ have negative slope, while $b$ and $f$ have positive slope.

  First, we show that $b$ must be a flippable barrier. If $b \not\in \FB(c,d)$,
  then by \cref{lem:awning} there exists an awning for $b$
  with respect to $\partial Q$ and
  the horizontal flow away from $T$. This awning starts at the corner where $a$ and $b$ meet and hence has positive slope, which means it must be $f$. It follows that
  $\width(f) > \width(b)$. But this implies that $T''$ takes up less
  than half the area of $Q$, a contradiction.
   
  Next, we prove that at least one of $a$ or $e$ is also a flippable barrier.
  Suppose $a \not\in \FB(c,d)$.
  By considering slopes, we deduce that $e$ is an awning for $a$ with
  respect to $\partial Q$ and the horizontal flow direction.
  It follows that $\width(e) > \width(a)$, and so $T''$ is a
  strictly major triangle. 
  Now consider the vertical flow emanating from $e$ away
  from $T''$. Any awning for $e$ with respect
  to~$\partial Q$ and the vertical flow must start at the corner where $b$ and $e$ meet. Hence, it must
  be a widest saddle connection of negative slope (and taller than $e$).
  But such a saddle connection cannot exist, since $\partial Q$
  has $b$ and $e$ as its widest saddle connections. 
  Therefore, by \cref{lem:awning},
  $e$~is flippable in some triangulation containing 
  $\partial Q$ and so $e \in \FB(c,d)$.
  Note that in this case,
  the heights of $a,b,e,f$ are all equal,
  and $Q$ is a kite.

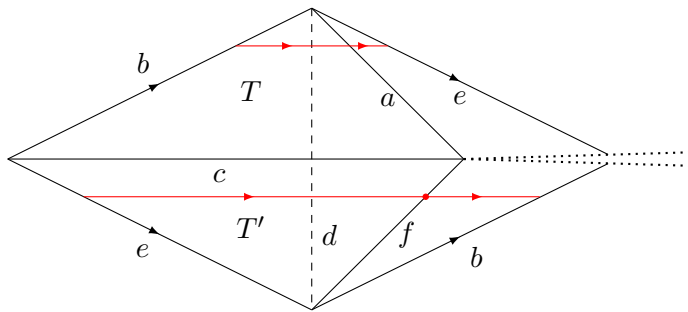
\begin{figure}
 \begin{center}
  \begin{tikzpicture}[scale=1.0]
   \draw[middlearrow={latex}] (0,0) -- node[above left] {$b$} (4,2);
   \draw[middlearrow={latex}] (4,-2) -- node[below right] {$b$} (7.9,-0.06);
   \draw[middlearrow={latex}] (0,0) -- node[below left] {$e$} (4,-2);
   \draw[middlearrow={latex}] (4,2) -- node[below] {$e$} (7.9,0.06);
   \draw (0,0) -- node[below left] {$c$} (6,0);
   \draw (4,-2) -- node[right] {$f$} (6,0);
   \draw (4,2) -- node[below] {$a$} (6,0);
   \draw[dashed] (4,2) -- (4,0) -- node[right] {$d$} (4,-2);
   \draw[thick, dotted] (6,0) -- (9,0.09);
   \draw[thick, dotted] (6,0) -- (9,-0.09);
   \node (A) at (3.2,0.9) {$T$};
   \node (B) at (3.2,-0.9) {$T'$};
   \draw[fill, color=red] (5.5,-0.5) circle (1pt);
   \draw[middlearrow={latex}, color=red] (5.5,-0.5) -- (7,-0.5);
   \draw[middlearrow={latex}, color=red] (3,1.5) -- (4.5,1.5);
   \draw[middlearrow={latex}, color=red] (4.5,1.5) -- (5,1.5);
   \draw[middlearrow={latex}, color=red] (1,-0.5) -- (5.5,-0.5);
  \end{tikzpicture}
  \caption{A kite $Q(c,d)$ in which only the sides $b$ and $e$ are flippable.
  Note that the copies~of~$b$ and $e$ on the right may not have the same endpoints
  (the dotted lines indicate~a~``cut'').}
  \label{fig:two_flippable_boundaries_kite}
 \end{center}
\end{figure}
  
  Finally, we characterise when there are only two flippable barriers.
  Assume that we are in the same setting as in the previous paragraph and suppose $a, f \not\in\FB(c,d)$. Let $T'$ be the triangle bounded by
  $c,e,f$.
  Then $b$ is an awning for $f$ with respect to $\partial Q$ and the
  horizontal flow away from $T'$. 
  Note that $Q$ cannot be a rhombus as $b$ and $f$ have different slopes.
  Each trajectory starting at an interior point of~$f$ will hit $b$ and
  continue across $T$ until it hits an interior point of $a$.
  The trajectory will then coincide with some trajectory starting from
  $a$ as described above (since it now flows away from~$T$);
  these in turn will hit $e$ and cross $T'$,
  until they reach the initial starting point on~$f$ (see \cref{fig:two_flippable_boundaries_kite}).
  Since this holds for every trajectory starting from the interior of $f$,
  the triangles~$T$ and~$T'$ are contained in a common horizontal cylinder $C$.
  Moreover, as $e$ is an awning for $a$, they share a common corner in some
  region $R \neq T,T'$ of $C - \partial Q$ and hence are the non-horizontal adjacent sides of $R$.
  Similarly, $b$ and $f$ are adjacent sides of the fourth region of $C - \partial Q$.
  It follows that $Q$ is a cylindrical kite.
  \end{proof}
\end{lem}

We have thus shown that $\#\FB(c,d) = 2$ if and only if $\{c,d\}$ is a cylindrical kite pair.
Moreover, when $c$ and $d$ are perpendicular, then $\FB(c,d)$ are the two longer
sides of $\partial Q(c,d)$.
Cylindrical kites only arise in a very particular set of circumstances
which we can use to our advantage.
We propose the following purely combinatorial procedure for recovering cylindrical kite boundaries.
As the two diagonals $c$ and $d$ of a cylindrical kite play different roles in the definition, we write $\kappa$ instead of $\{c,d\}$ and then, in the process of recovering the kite, determine which element of $\kappa$ plays the role of $c$ (or $d$).

\begin{prop}[Preparations to define kite barriers] \label{prop:kb}
Let $\kappa \in \FP(\A(S,q))$ be a cylindrical kite pair and let $\{b,e\} = \FB(\kappa)$ be the flippable barriers of~$\kappa$.
\begin{enumerate}
 \item There exists $\sigma \in \MIL(\A(S,q))$ satisfying $\FB(\kappa) \subset \lk(\sigma)$.
  Furthermore, $\kappa \cap \sigma$ consists of exactly one saddle connection. Call this saddle connection $c$ and call $d$ the other saddle connection in $\kappa$.
 \item There exists $\sigma' \in \MIL(\A(S,q))$ satisfying $\sigma' \supseteq (\sigma \cup \{b\}) \setminus \{c\}$.
  Furthermore, there is a unique saddle connection in $\lk(\sigma' \cup \{c\})$ disjoint from $d$.
  Call this saddle connection $a$.
 \item There exists $\sigma'' \in \MIL(\A(S,q))$ satisfying $\sigma'' \supseteq (\sigma \cup \{e\}) \setminus \{c\}$.
 Furthermore, there is a unique saddle connection in $\lk(\sigma'' \cup \{c\})$ disjoint from $d$.
 Call this saddle connection $f$.
 \item There exists a quadrilateral $Q \coloneqq Q(c,d)$ with diagonals $c$ and $d$. Moreover, the set $\{a,b,e,f\}$ is equal to $\partial Q(c,d)$. In particular, $\{a,b,e,f\}$ does not depend on the choices of $\sigma$, $\sigma'$, and $\sigma''$.
\end{enumerate}
\end{prop}

\begin{proof}
 By definition, there exists a quadrilateral $Q$ associated to $\kappa$. Let us assume, without loss of generality, that the diagonals of $Q$ are perpendicular, with the horizontal one bisecting the vertical one.
  The reader should refer to \cref{fig:bad_kite} as a guide throughout this proof,
  but imagine that there are initially no labels on the diagram.
  We recover $\FB(\kappa) = \{b,e\}$, yielding the two long sides of $Q$.
  For the sake of concreteness, choose $b$ to have positive slope, and $e$ negative.
  
  By definition, there exists a cylinder $C$ containing $\FB(\kappa) = \{b,e\}$ as transverse arcs. We claim that $C$ is the unique cylinder with this property and, furthermore, that $C$ is horizontal.
  
  If there exists a cylinder with slope $\theta\in\RP^1$
  having $b,e$ as transverse arcs, then every straight-line trajectory with slope $\theta$
  emanating from the interior of $b$ and $e$ will form closed (non-singular) loops.
  Consider the straight-line flows emanating from $b$ and $e$
  and pointing into the kite $Q$.
  For positive slopes, some trajectory emanating from the interior of $e$
  will hit a singularity which is an endpoint of the opposite side to $e$ in $Q$.
  Similarly, for negative slopes, some trajectory starting from
  the interior of $b$ will hit an endpoint of its opposite side.
  Finally, for the vertical slope, there is a downward trajectory starting from an interior point of $b$ that crosses $e$,
  and then hits the corner of $Q$ opposite the common corner of $b$ and $e$.
  Therefore, such a cylinder exists only for the horizontal slope, yielding the claim.
  
  Now we can choose $\sigma \in \MIL(\A(S,q))$ to be any triangulation away from $C$. By definition, it satisfies $\FB(\kappa) \subset \lk(\sigma)$.
  By \cref{rem:kite_pair}, exactly one of the saddle connections in $\kappa$
  belongs to $\partial C \subseteq \sigma$.
  Thus we have proven statement (i) and can assign $c$ to be the horizontal diagonal of~$Q$,
  and $d$ the vertical.
 
  Cutting $C$ along $b$ produces a polygonal region $R$.
  Exactly one of the two copies of $b$ on $\partial R$ has the two copies
  of $c$ as its adjacent sides (in \cref{fig:bad_kite}, this is
  the left copy of $b$).
  Let $b'$ be the topological arc in $R$ 
  connecting the endpoints of the two copies of $c$ that are not endpoints of $b$.
  Since $b'$ is a transverse arc of $C$, it can be realised as a saddle connection
  (it is shown in red in \cref{fig:bad_kite}).
  Thus, there is a quadrilateral $Q' \subset R$ bounded by $b,c,b',c$.
  Gluing $Q'$ along the two copies of $c$ yields a $(1,1)$--annulus $C'$, which
  must be a cylinder.
  
  We can choose $\sigma' \in \MIL(\A(S,q))$ to be any triangulation away from $C'$. Note that it satisfies $\sigma' \supseteq (\sigma \cup \{b\}) \setminus \{c\}$.
  The unique non-triangular region of $\sigma' \cup \{c\}$ is the parallelogram~$Q'$.
  Therefore $\lk(\sigma' \cup \{c\})$ has exactly two vertices:
  one that intersects $d$, and another that does not.
  The latter is the side of~$Q$ opposite $e$.
  Therefore, statement (ii) recovers the saddle connection~$a$ as in \cref{fig:bad_kite}
  (the other diagonal of this parallelogram appears as $a'$).
  
  Using a similar argument with $e$ taking the role of $b$, we can define a cylinder $C''$ and deduce that statement (iii) correctly recovers the saddle connection~$f$ as given in \cref{fig:bad_kite}.
  
  Finally, the cylinders $C$, $C'$, and $C''$ are uniquely defined
  for $\kappa$. Therefore, by \cref{sameMIL}, any choice of
  $\sigma$, $\sigma'$, and $\sigma'$ satisfying statements (i)--(iii)
  will produce the same set $\{a,b,e,f\}$.
\end{proof}

\begin{definition}[Kite barriers] \label{def:kb}
 Let $\kappa \in \FP(\A(S,q))$ be a cylindrical kite pair and choose $a,b,c,d,e,f$ as in \cref{prop:kb}. We then define the \emph{kite barriers} of $\kappa$ to be $\KB(\kappa) = \KB(c,d) \coloneqq \{a,b,e,f\}$.
\end{definition}

\begin{definition}[Kite-or-flippable barriers] \label{def:kfb}
Given a flip pair $\{c,d\}$, define its set of \emph{kite-or-flippable barriers} to be
${\KFB(c,d) \coloneqq \KB(c,d)}$ if $\{c,d\}$ is a cylindrical kite pair, and $\KFB(c,d) \coloneqq \FB(c,d)$ otherwise.
\end{definition}

The results of this section can be summarised as follows.

\begin{cor}[Three out of four ain't bad] \label{cor:kfb}
 If $\{c,d\}$ is a flip pair then $\#\KFB(c,d) \geq 3$ and $\KFB(c,d) \subseteq \partial Q(c,d)$.
 $\hfill\square$
\end{cor}

\section{Detecting triangles}\label{sec:triangle}

 Throughout this section, assume
  $\tau = \{a_1, a_2, a_3\} \in \A(S,q)$ is a $2$--simplex.
 Our goal is to give a purely combinatorial criterion to detect
 whether $\tau$ bounds a triangle on $(S,q)$.
 This will provide an analogue of the Topological Triangle Test
 (\cref{top_tri_test}) for the saddle~connection~complex.
 
 The Triangle Test comprises three subtests: the First Triangle Test, the Bigon Test, and the Second Triangle Test. Each test enlarges the pool of triangles that can be conclusively detected; moreover, this procedure also simplifies the case analysis in each of the tests.

 \begin{prop}[First Triangle Test] \label{prop:first_triangle_test}
  Assume $\tau =\{a_1, a_2, a_3\} \in \A(S,q)$ is a {$2$--simplex}
  satisfying the following two conditions.
  \begin{enumerate}[(T1)]
   \item There is a triangulation $\T \supseteq \tau$ 
 such that
    \begin{enumerate}[(i)]
     \item $\lk(\T \setminus \{a_k\}) = \{a_k\} \sqcup \{b\}$
    for some $b \in \A(S,q)$,
     \item $a_i, a_j \in \KFB(a_k, b)$, and
     \item $\lk(\T \setminus \{a_i, a_j\})$ contains
     \begin{tikzpicture}[scale = 0.8]
   \tikzstyle{every node}=[draw,circle,fill=black,minimum size=4pt, inner sep=0pt]
    \draw (0,0) node (a) [label=above:$a_i$] {}
    -- (1,0) node (b) [label=above:$a_j$] {}
    -- (2,0) node (e)  {}
    -- (3,0) node (f) {};
    \end{tikzpicture}
     as an induced subgraph,
    \end{enumerate}
  for some choice of distinct $i,j, k\in\{1,2,3\}$, and
    \item If $\lk(\T' \setminus \tau)$ is infinite for some triangulation $\T' \supseteq \tau$,
    then $\lk(\T' \setminus \{a_i, a_j\})$ is also infinite for some choice of distinct $i,j \in \{1,2,3\}$.
  \end{enumerate}
  Then $\tau$ bounds a triangle on $(S,q)$.
  
  As a partial converse, if $T$ is a major triangle then $\partial T$
  satisfies Conditions (T1) and~(T2).
 \end{prop}
 
 Note that if $\tau$ satisfies the First Triangle Test, then it bounds a triangle $T$
 contained in a triangulation in which at least two sides of $T$ are flippable.
 \cref{ex:bad_triangle} shows a triangle where at most one side is flippable, and so
 not all triangles can be detected by this test.
 Thus, more general tests are needed.

\begin{prop}[Bigon Test] \label{prop:bigon_test}
 Assume $\tau = \{a_1, a_2, a_3\} \in\A(S,q)$ is a $2$--simplex that does not satisfy the First Triangle Test.
 Assume furthermore that there exists a choice of distinct $i,j,k \in \{1,2,3\}$ and a triangulation $\T \supseteq \tau$ satisfying the following five conditions.
 \begin{enumerate}[(B1)]
  \item $\lk(\mathcal{T} \setminus \{a_k\}) = \{a_k\} \sqcup \{b\}$ for some $b \in \mathcal{A}(S,q)$,
  \item $\lk(\mathcal{T} \setminus \{a_i,a_j\}) = $
  \begin{tikzpicture}[scale = 0.8]
  \tikzstyle{every node}=[draw,circle,fill=black,minimum size=4pt, inner sep=0pt]
   \draw (0,0) node (a) [label=above:$a_i$] {}
   -- (1,0) node (b) [label=above:$a_j$] {};
 \end{tikzpicture}
  \item $\lk( (\mathcal{T} \setminus \tau) \cup \{b\} ) = $
  \begin{tikzpicture}[scale = 0.8]
   \tikzstyle{every node}=[draw,circle,fill=black,minimum size=4pt, inner sep=0pt]
    \draw (0,0) node (a) {}
    -- (1,0) node (b) [label=above:$a_i$] {}
    -- (2,0) node (e) [label=above:$a_j$] {}
    -- (3,0) node (f) {};
    \end{tikzpicture}, 
  \item $\lk(\mathcal{T} \setminus \tau) =$
  \begin{tikzpicture}[scale = 0.8, baseline=0]
  \tikzstyle{every node}=[draw,circle,fill=black,minimum size=4pt, inner sep=0pt]
   \draw (0,0) node {}
   -- (1,0) node (a) [label=above:$a_i$] {}
   -- (2,1) node (c) [label=below:$a_k$] {}
   -- (3,0) node (b) [label=above:$a_j$] {}
   -- (4,0) node {}
   -- (2,-1) node {}
   -- (0,0) {};
   \draw (3,0) node {} 
   -- (1,0) node {}
   -- (2,-1) node (c') [label=above:$b$] {}
   -- (3,0) node {};
    \end{tikzpicture}, and
  \item There does not exist a $2$--simplex $\tau' \subseteq \T$ that satisfies the First Triangle Test,
  and contains $\{a_i,a_k\}$ or $\{a_j,a_k\}$.
 \end{enumerate}
  Then $\tau$ bounds a triangle contained in a bigon with a simple pole in the interior.
  
  Conversely, if $T$ is a triangle contained in a bigon then $\partial T$ satisfies the First Triangle Test or Conditions (B1) -- (B5). In the latter case, the interior arcs of the bigon are $a_i$ and $a_j$.
 \end{prop}
 
 Before stating the Second Triangle Test, we need the following definition.
 
 \begin{definition}[Compatible flip partner]\label{def:cfp}
 Define the \emph{$\tau$--compatible flip partners} of a saddle connection $a_i \in \tau$ to be
 \[\Ft(a_i) \coloneqq \{b \in \A(S,q) ~|~ \{a_i, b\} \in \FP(\A(S,q)),~
   b \textrm{ is disjoint from both } a_j, a_k\}. \]
 This is precisely the set of saddle connections that can be obtained by
 flipping $a_i$ in some triangulation containing $\tau$.
 \end{definition}
 
 \begin{prop}[Second Triangle Test]
  Assume $\tau = \{a_1, a_2, a_3\} \in\A(S,q)$ is a $2$--simplex that neither satisfies the conditions of the First Triangle Test nor of the Bigon Test. Then $\tau$ bounds a triangle on $(S,q)$ if and only
  if the following three conditions hold.
    \begin{enumerate}
    \item[(T3)] If $\tau' \in \A(S,q)$ is a $2$--simplex such that $\#(\tau \cap \tau') \geq 2$, 
    then $\tau'$ neither satisfies the conditions of the First Triangle Test nor of the Bigon Test,
    \item[(T4)] For all distinct $i, j, k\in\{1,2,3\}$, there exists a triangulation $\T_i \supseteq \tau$ such that ${\lk(\T_i \setminus \{a_j, a_k\})}$ contains $\PG_3$
 as an induced subgraph, and
    \item[(T5)]For all distinct $i, j, k\in\{1,2,3\}$ and $b\in\Ft(a_i)$, we have 
  $\KFB(a_i,b) \cap \{a_j, a_k\} \neq \emptyset$.
  \end{enumerate}
 \end{prop}

 \begin{cor}[Triangle Test] \label{tri_test}
   A $2$--simplex $\tau \in \A(S,q)$ bounds a triangle on $(S,q)$ if and only if
   it satisfies the First Triangle Test, the Bigon Test, or the Second Triangle Test. $\hfill\square$
 \end{cor}

 The proofs of necessity for the three tests
 essentially combine results from the preceding sections.
 The main content of this section is in the proof of sufficiency.
 We show that the First Triangle Test can detect many triangles, including
 major triangles.
 The Bigon Test, as the name suggests, detects triangles contained in bigons such as in \cref{fig:codim2_annulus} on the right.
 Our general strategy for the Second Triangle Test is as follows.
 By \cref{detectable},
 Condition (T4) is equivalent to saying that
 for all distinct $i,j,k$, 
 there exists a triangulation $\T_k \supseteq \tau$ such that
 $\T_k \setminus \{a_i, a_j\}$ either has an almond pentagon or a (1,1)--annulus
 as its unique non-triangular region, in which $a_i$ and $a_j$ are disjoint
 saddle connections. This implies that there exists a triangle~$T_k$ of $\T_k$ having
 $a_i$ and $a_j$ as two sides. The goal is to use the other properties to
 prove that $T_i$, $T_j$, and $T_k$ are indeed the same triangle.
 The fact that major triangles and triangles in bigons can already be detected by the First Triangle Test and the Bigon test
 plays a key role in simplifying our case analysis.

 \subsection{Proof of the First Triangle Test}

 \subsubsection{Necessity for major triangles}
 Assume $\tau = \{a_1, a_2, a_3\}$ bounds a major triangle $T$ with base $a_3$.
 Depending on whether $T$ is annular or not, we choose a triangulation $\T \supseteq \tau$ to satisfy Condition (T1) as follows.
 
 If $T$ is an $(m,1)$--annular triangle with transverse arcs $a_1$ and $a_2$ for some $m\geq 1$, then by \cref{lem:ann_flip}, we may take $\T$ to be a triangulation in which $a_1, a_2, a_3$ are all flippable.
 This implies that (T1)(i) is fulfilled for every choice of $(i,j,k)$.
 If $m=1$, we choose $(i,j,k) = (1,2,3)$ and (T1)(iii) is fulfilled by \cref{codim2}.
 If $m\geq 2$, at least one of the unique non-triangular region of $\T - \{a_1, a_3\}$ and of $\T - \{a_2, a_3\}$ is a strictly convex pentagon. Therefore, (T1)(iii) is fulfilled for the corresponding choice of $(i,j,k)$. Let $Q$ be the unique non-triangular region of~$\T - \{a_k\}$.
 
 If $T$ is not $(m,1)$--annular with transverse arcs $a_1$ and $a_2$, consider a strictly convex quadrilateral $Q$ containing $T$, with $a_3$ as a diagonal, in which $T$ takes up at least half the area.
 By \cref{prop:major_ext}, we may choose a triangulation $\T \supseteq \partial Q \cup \{a_3\}$ such that $a_1$ and $a_3$ are flippable, and $\lk(\T \setminus \{a_1, a_2\})$ contains $\PG_3$ as an induced subgraph.
 Recall from \cref{rem:link_extension} that the subgraph is in fact 
 \begin{tikzpicture}[scale = 0.8]
   \tikzstyle{every node}=[draw,circle,fill=black,minimum size=4pt, inner sep=0pt]
    \draw (0,0) node (a) [label=above:$a_1$] {}
    -- (1,0) node (b) [label=above:$a_2$] {}
    -- (2,0) node (e)  {}
    -- (3,0) node (f) {};
 \end{tikzpicture}
 in our situation. Hence, (T1)(i) and (iii) are fulfilled for $(i,j,k) = (1,2,3)$.
 Another application of \cref{prop:major_ext} shows that there exists some triangulation (not necessarily equal to $\T$) containing $\partial Q \cup \{a_3\}$ in
 which $a_2$ and $a_3$ are~flippable.
 
 We show that (T1)(ii) is fulfilled for all of the previously discussed cases at once.
 Let~$b$ be the other diagonal of $Q$.
 If $\{a_k, b\}$ is a cylindrical kite pair then $$a_i, a_j \in \partial Q = \KB(a_k, b) = \KFB(a_k, b)$$ by \cref{prop:kb}.
 Otherwise, we have that $a_i$ and $a_j$ are flippable (in some triangulation containing $\partial Q \cup \{a_k\}$) and hence by definition
 $$a_i, a_j \in \FB(a_k,b) = \KFB(a_k, b)~.$$ 

 To verify the necessity of Condition (T2), let us assume that
 $\lk(\T' \setminus \{a_i, a_j\})$ is finite for all distinct $i,j\in\{1,2,3\}$.
 Let $T_i$ be the triangle of $\T'$ meeting $T$ along $a_i$.
 Then $\T' \setminus \{a_i, a_j\}$ has a pentagon, formed by gluing $T$ to $T_i$
 and $T_j$ along $a_i$ and $a_j$, as its unique non-triangular region.
 In particular, the triangles $T_1, T_2, T_3$ are distinct.
 Therefore, $\T' \setminus \tau$ has a hexagon formed by gluing $T$ to $T_1,T_2,T_3$
 along $a_1, a_2, a_3$ as its unique non-triangular region.
 This hexagon contains only finitely many diagonals, and so
 $\lk(\T' \setminus \tau)$ is finite.

 \subsubsection{Sufficiency}
 Assume $\tau$ satisfies Conditions (T1) and (T2).
 Without loss of generality, assume Condition (T1) holds for $(i,j,k) = (1,2,3)$.
 
 Let $Q_i$ be the unique quadrilateral region of $\T \setminus \{a_i\}$ for $i=1,2,3$.
 By Condition (T1)(i) and \cref{codim1}, $Q_3$ is strictly convex, with diagonals $a_3$ and $b$.
 By Condition (T1)(ii), we have $a_1, a_2 \in \KFB(a_3, b)$ and hence
 $a_1, a_2 \in \partial Q_3$ by \cref{cor:kfb}.
 
 If $\lk(\T \setminus \{a_i, a_3\})$ is infinite for some $i = 1,2$, then
 by \cref{codim2},
 the unique non-triangular region of $\T \setminus \{a_i, a_3\}$ is a
 $(1,1)$--annulus. This annulus is obtained by gluing~$Q_3$ along two copies of $a_i$,
 and so $a_i$ must form two opposite sides of $Q_3$.
 It follows that $\tau$ bounds a triangular region of $Q_3 - a_3$
 and we are done. We may henceforth assume 
 $\lk(\T \setminus \{a_i, a_3\})$ is finite for $i = 1,2$.
 
 By Condition (T1)(iii) and \cref{detectable},
 the unique non-triangular region $R$ of $\T \setminus \{a_1,a_2\}$
 is either an almond pentagon or a $(1,1)$--annulus;
 this region is respectively cut into three or two triangles by the saddle connections $a_1,a_2$.
 In particular, these are the only triangles
 of $\T$ meeting at least one of $a_1$ or $a_2$.

 Let us consider the case where $R$ is a $(1,1)$--annulus;
 this occurs precisely when $\lk(\T \setminus \{a_1, a_2\})$ is infinite.
 Since $a_1 \in \partial Q_3$, there is a triangle of $\T$
 obtained by cutting $Q_3$ along $a_3$ having $a_1$ and $a_3$
 as two of its sides. The only triangles of $\T$ that meet $a_1$
 are the two triangles obtained by cutting $R$ along $a_1$ and $a_2$,
 and so $a_3$ must form a side of $R$.
 It follows that $\tau$ bounds a $(1,1)$--annular triangle as desired.
 
 We shall henceforth assume that $R$ is an almond pentagon.
 In particular, $\lk(\T \setminus \{a_i, a_j\})$ is finite for all $i \neq j$,
 and so $\lk(\T \setminus \tau)$ is also finite by Condition (T2).
 Let $T$ be the triangle in $R$
 meeting both $a_1$ and $a_2$; and let $T'$ and $T''$ be the other triangles of $R$
 that meet only $a_1$ or $a_2$ respectively.
 The third side of $T$ is a saddle connection $c \in \partial R$; see \cref{fig:first_sufficient}.
 Our goal is to prove that $c = a_3$.
 If $T$ is contained in $Q_3$ then we are done,
 since~$a_3$ will form a side of $T$.

\begin{figure}
 \begin{center}
  \begin{tikzpicture}[scale=0.5]
   \draw (0,0) -- node[below] {$c$} (3,0) -- (5,3) -- (1.5,5)
    -- (-2,3) -- (0,0) -- node[above left] {$a_2$} (1.5,5) -- node[above right] {$a_1$}  (3,0);
   \node (A) at (1.5,1.5) {$T$};
   \node (B) at (3.2,2) {$T'$};
   \node (C) at (-0.2,2) {$T''$};
  \end{tikzpicture}
  \qquad
  \begin{tikzpicture}[scale=0.5]
   \draw (2,4) --  (5,5) -- (6,4);
   \draw[middlearrow={latex}] (0,0) -- node[left] {$a'_3$} (2,4);
   \draw[middlearrow={latex}]  (4,0) -- node[right] {$a_3$} (6,4);
   \draw[red] (2,4) node[left] {$p'$} -- node[below] {$e$} (6,4) node[right] {$p$};
   \draw (0,0) -- node[below] {$d$} (4,0);
   \draw[orange] (0.5,1) -- node[below] {$\gamma$} (4.5,1);
  \end{tikzpicture}
  \qquad
  \begin{tikzpicture}[scale=0.5]
   \draw (2,4) --  (4,3.5) -- (6,4);
   \draw[middlearrow={latex}] (0,0) -- node[left] {$a'_3$} (2,4);
   \draw[middlearrow={latex}]  (4,0) -- node[right] {$a_3$} (6,4);
   \draw (0,0) -- node[below] {$d$} (4,0);
   \draw[blue] (0,0) -- (6,4);
   \draw[blue] (2,4) -- (4,0);
   \draw[red, densely dashed] (2,4) node[left] {$p'$} to [bend right=40] node[pos=0.3, below] {$e$} (6,4) node[right] {$p$};
   \draw[orange] (0.5,1) -- node[below] {$\gamma$} (4.5,1);
  \end{tikzpicture}
  \caption{Left to right: an almond pentagon formed by gluing $T$ to $T'$ and $T''$
  along $a_1$ and $a_2$; the case where $e$ is a straight diagonal; the case where
  $e$ is a broken diagonal. In either case, there is a cylinder curve $\gamma$ (in orange).}
  \label{fig:first_sufficient}
 \end{center}
\end{figure}
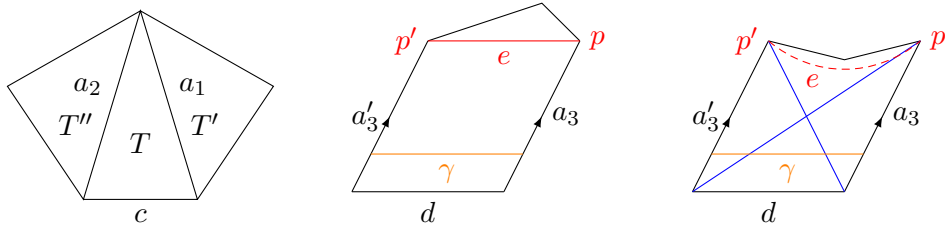 

 Let us suppose otherwise for a contradiction (this is equivalent to assuming $c \neq a_3$).
 Since $a_1, a_2 \in \partial Q_3$, and $T', T''$ are the only triangles of $\T$
 other than $T$ meeting $a_1$ or $a_2$, it follows that $T'$ and $T''$ are the
 two triangular regions of $Q_3 - a_3$.
 Therefore $a_3$ must form two sides of the pentagon $R$, one meeting $T'$ and one meeting $T''$.
 We rule out the case where these two sides are adjacent on the pentagon~$R$. If this were so, the angle between these two sides would be either $\pi$ or $2\pi$.
 In both cases, the diagonal that would cross both $a_1$ and $a_2$ is a broken diagonal. This implies that the link of $\mathcal{T} \setminus \{a_1,a_2\}$ is an induced subgraph of
\begin{tikzpicture}[scale = 0.8]
   \tikzstyle{every node}=[draw,circle,fill=black,minimum size=4pt, inner sep=0pt]
    \draw (0,0) node (a) {}
    -- (1,0) node (b) [label=above:$a_1$] {}
    -- (2,0) node (e)  [label=above:$a_2$] {}
    -- (3,0) node (f) {};
 \end{tikzpicture} which contradicts (T1)(iii).

 Thus, $a_3$ forms two non-adjacent sides of $R$ and $\T \setminus \tau$ has a topological $(2,1)$--annulus as its unique non-triangular region;
 it is formed by gluing $R$ along the two copies of~$a_3$.
 We claim that there is a cylinder curve contained in this annulus.
 This will imply that $\lk(\T \setminus \tau)$ is infinite, by \cref{IL},
 contradicting Condition (T2).
 
 Let $d \in \partial R$ be the side of $R$ adjacent to both copies
 of $a_3$. To avoid potential confusion, we shall
 label them $a_3$ and $a'_3$ respectively.
 Let $p$ and $p'$ respectively be the corners of $R$ that lie on
 an endpoint of $a_3$ and $a'_3$, but are not endpoints of $d$.
 Note that $p$ and $p'$ are not adjacent corners of $R$.
 Let $e$ be the diagonal of $R$ connecting $p$ and $p'$.
 
 If $e$ is a straight diagonal,
 then there is a quadrilateral $Q \subset R$ bounded by
 $a_3, d, a'_3, e$. Gluing~$Q$ to itself along the two copies of $a_3$ yields
 a $(1,1)$--annulus, which must be a cylinder. The core curve $\gamma$
 of this cylinder is shown in \cref{fig:first_sufficient}.
 
 We are left with the case where $e$ is a broken diagonal.
 Since $R$ is an almond pentagon, $e$ must be the only broken diagonal. 
 Therefore, there exists a triangle contained in $R$ with $a_3$ and~$d$ as
 two of its sides, and so $\angle(d, a_3) < \pi$. Similarly, we deduce
 $\angle(d, a'_3) < \pi$. Since~$a_3$ and~$a'_3$ are parallel,
 a small regular neighbourhood of $d$ within $R$ is a
 Euclidean parallelogram. This parallelogram forms a cylinder upon gluing
 the two copies of $a_3$ in $\partial R$; thus there is a cylinder curve $\gamma$ that is parallel to $d$ and contained in $R$ (see \cref{fig:first_sufficient}).
 
 This completes the proof of the First Triangle Test.

\subsection{Proof of the Bigon Test}\label{sec:bigon_test}

\subsubsection{Necessity for triangles contained in bigons}

 We shall begin with a useful lemma for triangles sharing a pair of sides.

 \begin{lem}[Triangles with two common sides] \label{two_tri}
  Let $T$ and $T'$ be distinct triangles with two sides in common.
  Then $T$ and $T'$ are contained in a common bigon, or $\partial T$ and $\partial T'$ both satisfy the First Triangle Test.
 \end{lem}
 
 \proof
 Suppose $a,b$ are the common sides of $T$ and $T'$.
 By applying an $\SL(2,\R)$--deformation, we may assume that $a$ and
 $b$ are respectively vertical and horizontal, and both have length $h > 0$.
 Then $T$ and $T'$ are both right-angled isosceles
 triangles. There are two possible configurations (up to swapping the roles of $a$ and $b$)
 depending on the sides of $a$ and $b$ on which the triangles~$T$ and~$T'$ appear.
 
 \begin{enumerate}
  \item $T$ and $T'$ appear on opposite sides of $a$, and on opposite sides of $b$;
  \item $T$ and $T'$ appear on the same side of $a$, and on opposite sides of $b$;
 \end{enumerate}
 Note that $T$ and $T'$ cannot appear on the same sides of both $a$ and $b$, for
 otherwise they would coincide.
 The different cases appear as shown in \cref{fig:triangles_with_two_sides_in_common} (up to reflection).

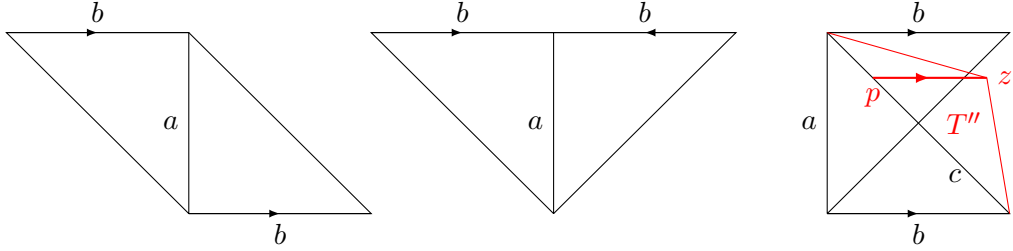
\begin{figure}[b]
 \begin{center}
  \begin{tikzpicture}[scale=0.6]
   \draw (0,0) -- node[left]{$a$} (0,4) -- node[above]{$b$} (-4,4) -- (0,0);
   \draw (0,0) -- node[below]{$b$} (4,0) -- (0,4);
   \draw[middlearrow={latex}] (-4,4) -- (0,4);
   \draw[middlearrow={latex}] (0,0) -- (4,0);
   
   \begin{scope}[xshift=8cm]
    \draw (0,0) -- node[left]{$a$} (0,4) -- node[above]{$b$} (-4,4) -- (0,0);
    \draw (0,4) -- node[above]{$b$} (4,4) -- (0,0);
    \draw[middlearrow={latex}] (-4,4) -- (0,4);
    \draw[middlearrow={latex}] (4,4) -- (0,4);
   \end{scope}
   
   \begin{scope}[xshift=14cm]
    \draw (0,0) -- node[left]{$a$} (0,4) -- node[above]{$b$} (4,4) -- (0,0);
    \draw (0,0) -- node[below]{$b$} (4,0) -- node[pos=0.3, below]{$c$} (0,4);
    \draw[middlearrow={latex}] (0,0) -- (4,0);
    \draw[middlearrow={latex}] (0,4) -- (4,4);
    \draw[color=red, middlearrow={latex}, thick] (1,3) node[below] {$p$} -- (3.5,3)
   node[right] {$z$};
    \draw[red] (0,4) -- (3.5,3) -- (4,0); 
    \node[red] (T) at (3,2) {$T''$};
   \end{scope}
  \end{tikzpicture}
  \caption{The first two configurations appear in Case (i) and the last in Case (ii) of the proof of \cref{two_tri}.}
  \label{fig:triangles_with_two_sides_in_common}
 \end{center}
\end{figure}
 
 In Case (i), the triangles $T$ and $T'$ can be glued along $a$ and $b$ to form either a $(1,1)$--annulus or a bigon. In either situation, the desired result holds.
 
 For Case (ii), let $c$ and $c'$ respectively be the third side of $T$ and $T'$.
 Let $\varphi^t$ be the horizontal unit-speed flow
 emanating from $c$ away from $T$ (strictly speaking, we are working in the universal cover).
 Consider the trajectory that begins at the common corner of $a$ and $c$ in $T$.
 This trajectory runs along  $b$ until its hits the other endpoint at time $t=h$.
 This endpoint is a visible singularity with respect to $\varphi^t$
 and $\partial T$. Applying \cref{prop:vis}, there exists a triangle~$T''$ meeting $c$ on the side from which the flow emanates, of $\height(T'') = h$,
 with a visible singularity $z = \varphi^{h'}(p)$ as its corner opposite~$c$,
 for some $0 < h' \leq h$ and a point $p \in c$. Then $T$ and $T''$ can be glued
 along~$c$ to form a strictly convex quadrilateral in which $T$ takes up at least
 half the area. Thus~$T$ is a major triangle. Using a similar flowing
 argument with $c'$, we prove that $T'$ is also major.
 In particular, by \cref{prop:first_triangle_test},
 $\partial T$ and $\partial T'$ both satisfy the First Triangle Test.
 \endproof

  Assume that $\tau = \{a_i,a_j,a_k\}$ bounds a triangle $T$ contained in a bigon with a simple pole in its interior and~$a_i$ and $a_j$ as interior arcs.
  We may assume that $\tau$ does not satisfy the First Triangle Test.
  Then~$a_k$ forms a side of this bigon; call the other side $c_k$.
  By applying an $\SL(2,\R)$--deformation, we can assume that $a_k$ is vertical and $c_k$ is horizontal as in  \cref{fig:triangle_in_bigon_vertical_horizontal}.
  Observe that $a_k$ is strictly tallest in $\sigma \coloneqq \tau \cup \{c_k\}$ and is not contained in a horizontal cylinder.
  Consider the horizontal flow emanating from $a_k$ away from $T$.
  By \cref{prop:tall} and \cref{prop:vis}, there exists an acute-angled triangle $T'$ with $a_k$ as one of its sides such that $\partial T' \cup \sigma$ spans a simplex in $\A(S,q)$;
  let $\T$ be any triangulation containing~$\partial T' \cup \sigma$.
  The unique non-triangular region of $\T \setminus \{a_i, a_j\}$ is a bigon containing a simple pole, and so Condition~(B2) holds by \cref{codim2}.
  Gluing $T$ to $T'$ along $a_k$ forms a strictly convex quadrilateral $Q$ with $a_i$ and $a_j$ appearing as adjacent sides.
  Setting $b$ to be the other diagonal of $Q$, we see that $\T$ satisfies Condition (B1).

\begin{figure}
 \begin{center}
  \begin{tikzpicture}
   \draw (0,0) -- node[below]{$c_k$} (3,0) -- node[right]{$a_k$} (3,3) -- node[above left]{$a_i$} (1.414,1.414) -- node[above left]{$a_i$} (0,0);
   \draw (1.414,1.414) -- node[left]{$a_j$} (3,0);
   \draw (2.5,1.5) node{$T$};
   
   \begin{scope}[xshift=6cm]
    \draw (0,0) -- node[below]{} (3,0) -- node[right]{} (3,3) -- node[above left]{} (1.414,1.414) -- node[above left]{} (0,0);
    \draw (1.414,1.414) -- node[below left]{} (3,0);
    \draw (2.5,1.5) node{};
    
    \draw (3,0) -- (7,2) -- (3,3);
    \draw[densely dashed] (1.414,1.414) -- node[]{$b$} (7,2);
    \draw[densely dashed] (7,2) -- (2.5,2.5);
    \draw[densely dashed] (0.5,0.5) -- (3,0);
    \draw[densely dashed] (7,2) -- (0,0);
   \end{scope}
  \end{tikzpicture}
  \caption{The situation for the bigon before and after flowing. In the right one, all diagonals of $\T \setminus \tau$ are drawn (dashed or full).}
  \label{fig:triangle_in_bigon_vertical_horizontal}
 \end{center}
 \end{figure}
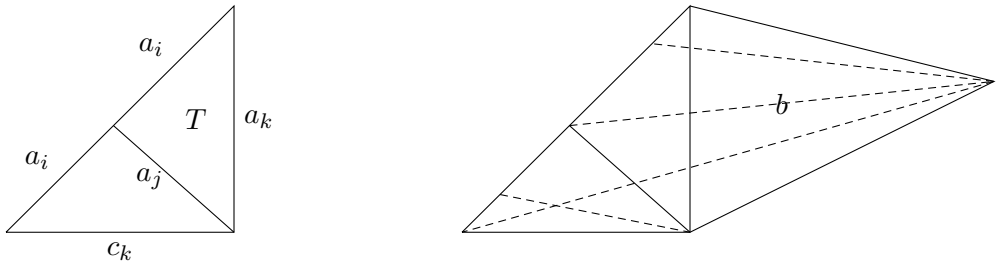
  
  Let $\T'$ be the triangulation obtained from $\T$ by flipping $a_k$.
  Then $(\T\setminus \tau) \cup \{b\} = \T' \setminus \{a_i,a_j\}$ has an almond pentagon with one corner having angle $\pi$ as its unique non-triangular region.
  Since the unique broken diagonal is the topological arc intersecting $a_i$ and $a_j$, it follows that Condition (B3) also holds.
  Observe that $\T \setminus \tau$ has a triangle containing a simple pole as its unique non-triangular region.
  All saddle connections belonging to $\lk(\T \setminus \tau)$ appear in \cref{fig:triangle_in_bigon_vertical_horizontal} and it can be checked there that $\lk( \T \setminus \tau)$ is as in Condition (B4).
  
  Finally, to check Condition (B5), note that there are exactly two triangles in $\mathcal{T}$ that contain~$a_i$ and exactly two triangles that contain $a_j$.
  These are precisely the triangles bounded by $\{a_i,a_j,a_k\}$ and $\{a_i,a_j,c_k\}$, respectively.
  Hence, Condition~(B5) has only to be checked for $\{a_i,a_j,a_k\}$ for which it is true by assumption.
  
  Thus all conditions hold and we have proven the statement.

\subsubsection{Sufficiency}

  Assume $\tau = \{a_1,a_2,a_3\}$ is a $2$--simplex that does not satisfy the First Triangle Test.
  Let $i,j,k$ be a permutation of $1,2,3$ and $\T \supseteq \tau$ be a triangulation satisfying Conditions (B1) -- (B5).
  Condition (B1) and \cref{cor:flip} imply that~$a_k$ can be flipped in $\T$, with flip partner $b$, to obtain a new triangulation~$\T'$.
  Condition (B3) and \cref{codim2} imply that $\T' \setminus \{a_i,a_j\} = (\T \setminus \tau) \cup \{b\}$ has a pentagon~$P$ with one broken diagonal as its unique non-triangular region.
  Moreover, the broken diagonal of $P$ is the arc intersecting both~$a_i$ and $a_j$ (which are themselves straight diagonals of $P$).
  Let $T'$ be the triangle contained in $P$ with $a_i$ and $a_j$ as two of its sides.
  Condition (B4) implies that $\T' \setminus \{a_i,a_j,b\} = \T \setminus \tau$ has exactly one non-triangular region $R$ and hence $b$ must appear as at least one side of $P$.
  We claim that~$b$ appears as a pair of adjacent sides of $P$ forming an angle of $\pi$.
  
  First, we rule out the case where $b$ appears as only one side. Suppose first that $b$ appears exactly as the third side of $T'$.
  Then $a_k$ and $b$ form the diagonals of a strictly convex quadrilateral having $a_i$ and $a_j$ as adjacent sides. This quadrilateral contains a major triangle $T''$ having $a_k$ and one of $a_i$ or $a_j$ among its sides. But then $\partial T'' \subseteq \T$ satisfies the First Triangle Test and contains $\{a_i,a_k\}$ or $\{a_j, a_k\}$, contradicting Condition (B5).
  Without loss of generality, we may thus assume that~$b$ appears exactly as a side of the triangle obtained by cutting $P$ along $a_i$. Then $a_k$ does not intersect the interior of the strictly convex quadrilateral obtained by cutting~$P$ along $a_i$. It follows that the non-triangular region of $\T \setminus \{a_i,a_j\}$ contains the flip partner of~$a_j$ in $\T'$ as straight diagonal. This contradicts Condition (B2).
    
  Suppose now that $b$ forms two non-adjacent parallel sides of $P$.
  If they are identified by a translation, then we can find a cylinder curve in $R$ parallel to the side of~$P$ adjacent to both copies of $b$.
  This contradicts Condition (B4) combined with \cref{IL}.
  If they are identified by a half-translation, then $P$ will have at least two broken diagonals as can be seen in the left of \cref{fig:pentagon_with_half_identified_sides}, a contradiction.
  
 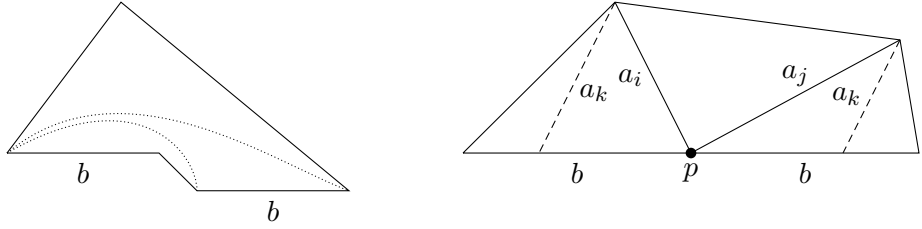
\begin{figure}
 \begin{center}
  \begin{tikzpicture}
  \begin{scope}
    \draw (0,0) -- node[below]{$b$} (2,0) -- (2.5,-0.5) -- node[below] {$b$} (4.5,-0.5) -- (1.5,2) -- (0,0);
   
    \draw[densely dotted] (0,0) .. controls +(30:2cm) and +(90:0.5cm) .. (2.5,-0.5);
    \draw[densely dotted] (0,0) .. controls +(40:2cm) and +(160:1cm) .. (4.5,-0.5);
   \end{scope}

  \begin{scope}[xshift=6cm]
    \draw (0,0) -- node[below]{$b$} (3,0) -- node[below] {$b$} (6,0) -- (5.75,1.5) -- (2,2) -- (0,0);
    \fill (3,0) circle (2pt) node[below]{$p$};
    
    \draw (3,0) -- node[left]{$a_i$} (2,2);
    \draw (3,0) -- node[above]{$a_j$} (5.75,1.5);
    
    \draw[densely dashed] (2,2) -- node[pos=0.6, right]{$a_k$} (1.0,0);
    \draw[densely dashed] (5.75,1.5) -- node[left]{$a_k$} (5.0,0);
   \end{scope}
  \end{tikzpicture}
  \caption{In these pentagons, the sides that are labelled by $b$ are identified by a half-translation. Left: The dotted arcs cannot be realised as saddle connections. Right: The corner $p$ is a simple pole and the dashed diagonal $a_k$ is the flip partner of $b$.}
  \label{fig:pentagon_with_half_identified_sides}
 \end{center}
 \end{figure}
 
  Thus, $b$ forms two adjacent parallel sides of $P$. Since $P$ is an almond pentagon, the corner $p$ formed by these two sides has angle $\pi$. Therefore the broken diagonal of $P$ cuts off a (topological) triangle with $b$ as two of its sides.
  Since $a_i$ and $a_j$ intersect the broken diagonal, they both have an endpoint at $p$. As $a_k$ is the flip partner of $b$, it is contained in $R$. It follows that $a_i,a_j,a_k$ form three sides of a Euclidean triangle $T$ (see \cref{fig:pentagon_with_half_identified_sides}).
  Gluing $T$ to $T'$ along $a_i$ and $a_j$ yields a bigon having $a_i$ and $a_j$ as interior arcs as~desired.

  This completes the proof of the Bigon Test.

\subsection{Proof of the Second Triangle Test}

 \subsubsection{Necessity}

 Assume $\tau = \{a_1, a_2, a_3\}$ bounds a triangle on $(S,q)$ that neither satisfies
 the First Triangle Test nor the Bigon Test. Condition (T4) is an immediate consequence of
 \cref{prop:extension}.
 
 To verify Condition (T3), suppose $\tau'\in\A(S,q)$ is a $2$--simplex
 sharing at least two vertices with $\tau$.
 If $\tau'$ satisfies the First Triangle Test or the Bigon Test, it must bound a triangle.
 Then by \cref{two_tri}, $\tau$ must also satisfy the First Triangle Test or the Bigon Test which is a contradiction.
 
 To verify Condition (T5), suppose $b$ is obtained by flipping
 $a_i$ in some triangulation $\T \supseteq \tau$. Then $a_j,a_k$ form
 two sides of the strictly convex quadrilateral $Q(a_i, b)$ with diagonals $a_i$ and $b$.
 By \cref{cor:kfb}, $\KFB(a_i, b)$ contains at least three sides of
 $Q(a_i, b)$. Hence, at least one of~$a_j$ and $a_k$ must belong to~$\KFB(a_i, b)$.

\subsubsection{Sufficiency}
 
  Assume $\tau = \{a_1, a_2, a_3\}$ satisfies the Second Triangle Test,
  but neither the First Triangle Test nor the Bigon Test.
  We shall use the indices $i,j,k$ to stand for a fixed permutation of $1,2,3$.
   
  \begin{lem}[(T3) and (T4) give three almond pentagons]
   For all distinct $i,j,k \in \{1,2,3\}$,
   $\T_i \setminus \{a_j, a_k\}$ has an almond pentagon $P_i$ as its unique non-triangular region.
  \end{lem}

  \proof
   By Condition (T4) and \cref{detectable}, ${\T_i \setminus \{a_j, a_k\}}$ has
   either a $(1,1)$--annulus or an almond pentagon as its unique non-triangular region.
   If this region is a $(1,1)$--annulus, then $a_j$ and $a_k$ form two
   sides of some $(1,1)$--annular triangle $T'$. But then $\partial T'$
   is a $2$--simplex sharing at least two vertices with~$\tau$, and
   satisfying the First Triangle Test. This contradicts Condition~(T3).
  \endproof
  
  The pentagon $P_i$ contains a triangle
  $T_i$ with $a_j$ and $a_k$ as two of its sides;
  let $c_i \in \partial P_i$ denote the third side of $T_i$.  
  Our goal is to prove that $c_i = a_i$ for some $i$.
  Note that $c_i \in \T_i$ and so it cannot intersect any of $a_i, a_j, a_k$ transversely.

  \begin{lem}[Properties of $T_i$] \label{lem:unique_tri}
  The triangle $T_i$ is not major.
  Moreover, $T_i$ is the only triangle on $(S,q)$¸ having $a_j$ and $a_k$ as two of its sides.
  \end{lem}
  
  \proof
  Since $\partial T_i$ shares at least two vertices with $\tau$,
  it does not satisfy the First Triangle Test by Condition (T3).
  Therefore $T_i$ cannot be major by \cref{prop:first_triangle_test}.
  If $T' \neq T_i$ is another triangle having
  $a_j$ and $a_k$ as its sides, then, by \cref{two_tri}, $\partial T'$ and
  $\partial T_i$ both satisfy the First Triangle Test or the Bigon Test,
  a contradiction.
  \endproof
  
  Next, orient the triangles $T_i$ by choosing the cyclic order
  of the subscripts to agree with~$(1,2,3)$. To be explicit, set
   $\TV_1 = [c_1, a_2, a_3]$, $\TV_2 = [a_1, c_2, a_3]$, and
  $\TV_3 = [a_1, a_2, c_3]$.
  Our goal is to show that these triangles are consistently oriented on $(S,q)$.
  Before doing so, let us introduce some terminology.  
  In each $T_i$, direct the sides $a_j$ and $a_k$ so that they point
  towards their common corner. Say that $T_i$ and $T_j$ meet \emph{concordantly}
  along~$a_k$ if the directions on $a_k$ coming from $T_i$ and $T_j$ agree,
  otherwise we say that they meet \emph{discordantly}; see \cref{fig:4_cases}. 
  Write $\angle_{T_i}(a_j, a_k)$ for the angle between $a_j$ and $a_k$,
  measured inside the triangle $T_i$.
  
    \begin{figure}
 \begin{center}
    \begin{tikzpicture}[scale=0.5]
   \draw (0,0) -- node[left] {$c_i$} (-3,-5) -- node[below] {$a_j$} (0,-4) -- node[below left] {$a_k$} (0,0)
    -- node[right] {$c_j$} (3,-5) -- node[below] {$a_i$} (0,-4);
  \end{tikzpicture}
  \hspace*{0.5cm}
  \begin{tikzpicture}[scale=0.5]
   \draw (0,1) -- node[above left] {$a_i$} (0,-4) -- node[left] {$c_j$} (3,2) -- node[above left] {$a_k$} (0,1)
    -- node[right] {$c_i$} (5,-3) -- node[above right] {$a_j$} (3,2);
  \end{tikzpicture}
  \hspace*{0.5cm}
  \begin{tikzpicture}[scale=0.5]
   \draw (0,1) -- node[above left] {$a_i$} (0,-4) -- node[left] {$c_j$} (3,2) -- node[above left] {$a_k$} (0,1)
    -- node[right] {$a_j$} (5,-3) -- node[above right] {$c_i$} (3,2);
  \end{tikzpicture}
  \hspace*{0.5cm}
  \begin{tikzpicture}[scale=0.5]
   \draw (0,0) -- node[left] {$a_j$} (-3,-5) -- node[below] {$c_i$} (0,-4) -- node[below right] {$a_k$} (0,0)
    -- node[right] {$c_j$} (3,-5) -- node[below] {$a_i$} (0,-4);
    \draw[red, dashed] (-2.5,-1.5) -- node[above] {$b$} (1.5,-4.5);
  \end{tikzpicture}
  \caption{The four cases for gluing $T_i$ to $T_j$ along $a_k$. From left to right:
  concordant with same orientation; discordant with same orientation;
  concordant with opposite orientation; discordant with opposite orientation.}
  \label{fig:4_cases}
 \end{center}
\end{figure}
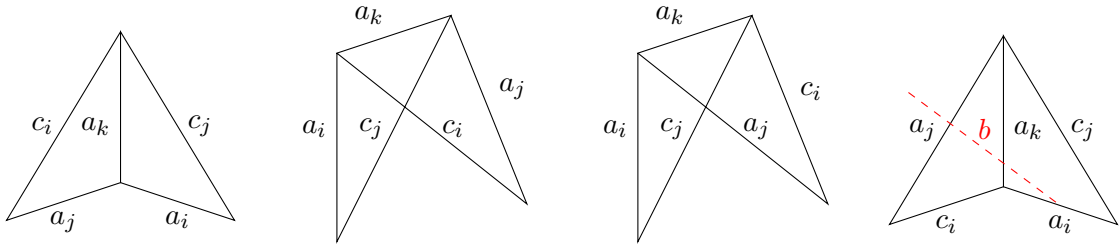
  
  \begin{lem}[Consistent orientation]
  The triangles $\TV_1$, $\TV_2$, and $\TV_3$ are consistently oriented on $(S,q)$.
  \end{lem}
  
  \proof
  Suppose $\TV_i$ and $\TV_j$ are not consistently oriented. We consider two cases,
  depending on whether they meet concordantly or discordantly along $a_k$.
  
  First, suppose $T_i$ and $T_j$ meet concordantly. Without loss of generality,
  we may suppose that $\angle_{T_i}(a_j, a_k) < \angle_{T_j}(a_k, a_i)$.
  (The angles cannot be equal, for otherwise $a_i = a_j$.)
  But then $a_j$ intersects $c_j$, or $T_j$ is strictly contained in $T_i$, which is both impossible.
  
  Next, suppose $T_i$ and $T_j$ meet discordantly. 
  Consider the quadrilateral $Q$ formed by $T_i$ and~$T_j$, glued along $a_k$.
  (Strictly speaking, this does not always result in a quadrilateral
  on $(S,q)$ with embedded interior, so we should really work in the universal cover.)
  If~$Q$ is strictly convex, then it has embedded interior on $(S,q)$ by \cref{poly_embed}.
  In this situation, at least one of $T_i$ or $T_j$ is major, which is ruled out by \cref{lem:unique_tri}.
  Thus, $Q$ is not strictly convex. Assume,
  without loss of generality, that $\angle_{T_i}(a_k, c_i) + \angle_{T_j}(a_i, a_k) \geq \pi$,
  as shown in the fourth case of \cref{fig:4_cases}.
  Recall that $P_j$ is a pentagon with at most one broken diagonal. Hence there exists at least one straight diagonal $b$ of $P_j$ that intersects $a_k$. Note that $b$ cannot intersect $c_j$, so it has to intersect~$a_j$. But $a_j$ is disjoint from the interior of $P_j$ and hence from $b$, a contradiction.
  \endproof

Next we provide a criterion for $\tau$ bounding a triangle in the case when we have at least one discordant gluing.

\begin{lem}[Parallel sides coincide] \label{lem:parallel}
   Assume $T_i$ and $T_j$ meet discordantly along $a_k$. If $a_i$ and~$c_i$ are parallel, they must coincide. Consequently,
   $\tau$ bounds a triangle.
  \end{lem}
  
  \proof
  Observe that
  $\angle_{T_i}(a_k, c_i) \leq \angle_{T_j}(a_k, a_i)$,
  for otherwise either $a_i$ intersects $a_j$ transversely
  or $T_j$ is strictly contained in $T_i$, neither of which is possible.
  Therefore $T_j \cap c_i$ contains a segment of $c_i$ starting from
  the common corner of $a_i$ and $a_k$ in $T_j$. Since this segment and $a_i$
  are contained in a common triangle, we deduce that $\angle_{T_j}(a_i,c_i) < \pi$.
  Hence if $a_i$ and $c_i$ are parallel, we have $\angle_{T_j}(a_i,c_i) = 0$.
  It follows that $a_i$ and~$c_i$ are the same saddle connection.
  \endproof

For the rest of this section, we shall assume for a contradiction that $\tau$ does not bound a triangle.
Our strategy is to prove that the triangles $T_i$, $T_j$, and $T_k$ can be placed into one of three standard configurations as depicted in \cref{fig:setting_end_if_sufficiency_second_triangle_test}, with $a_i$ strictly taller than $a_j$ and $a_k$, by applying $\SL(2,\R)$--deformations, reflections, and permutations of the indices.

\begin{figure}
 \begin{center}
  \begin{tikzpicture}[scale=0.45]   
   \begin{scope}[xshift=-9cm]
    \draw (0,0) -- node[left] {$a_i$} (0,-5) -- node[right] {$c_j$} (3,2) -- node[left] {$a_k$} (0,0);
    \draw (3,2) -- node[above] {$c_i$} (-4,2);
    \draw (0,0) -- node[right] {$a_j$} (-4,2);
    \draw (-4,2) -- node[left] {$c_k$} (0,-5);
   \end{scope}
   
   \draw (0,0) -- node[left] {$a_i$} (0,-5) -- node[right] {$c_j$} (3,2) -- node[above left] {$a_k$} (0,0);
   \draw (0,0) -- node[below] {$c_i$} (7,0);
   \draw[middlearrow={latex}] (7,0) -- node[above right] {$a_j$} (3,2);
   \draw[middlearrow={latex reversed}]  (0,0) -- node[above right] {$a_j$} (-4,2);
   \draw (-4,2) -- node[left] {$c_k$} (0,-5);
   
   \begin{scope}[xshift=-1cm]
   \draw (10,0) -- node[left] {$a_i$} (10,-5) -- node[right] {$c_j$} (13,2) -- node[above left] {$a_k$} (10,0);
   \draw (10,0) -- node[below] {$c_i$} (17,0);
   \draw[middlearrow={latex}] (17,0) -- node[above right] {$a_j$} (13,2);
   \draw (10,0) -- node[right] {$c_k$} (14,-7);
   \draw[middlearrow={latex}] (14, -7) -- node[below left] {$a_j$} (10,-5);
   \end{scope}
  \end{tikzpicture}
  \caption{Three configurations for the triangles $T_i$, $T_j$, and $T_k$. From left to right: All triangles meet concordantly; $T_i$ and $T_j$ meet discordantly along $a_k$ and $T_j$ and $T_k$ meet concordantly along $a_i$; $T_i$ and $T_j$ meet discordantly along $a_k$ and $T_j$ and $T_k$ meet discordantly along $a_i$. Note that $\height(a_i) > \height(a_j) = \height(a_k)$ in all cases.}
  \label{fig:setting_end_if_sufficiency_second_triangle_test}
 \end{center}
\end{figure}

Let us first consider the case where all gluings are concordant.

  \begin{lem}[When all gluings are concordant] \label{lem:all_gluings_concordant}
   If all three gluings among $T_i$, $T_j$, and $T_k$ are concordant, then they glue together to form a triangle containing exactly one removable singularity which, moreover, forms a common endpoint of $a_i$, $a_j$, and $a_k$. Furthermore, no two saddle connections in $\{a_1, a_2, a_3, c_1, c_2, c_3\}$ are parallel.
  \end{lem}

  \proof
  Suppose all gluings are concordant. Then $a_i$, $a_j$, and $a_k$
  will be directed towards a common singularity.
  Examining the cyclic gluing pattern of $T_i$, $T_j$, and $T_k$ along the sides $a_k$, $a_i$, and $a_j$, we see that a small neighbourhood of the common singularity is formed by gluing
  together one corner from each of $T_i, T_j, T_k$.
  By summing these corner angles, it follows that this
  singularity has cone angle strictly less than $3\pi$.
  
  Suppose this cone angle equals $\pi$. 
  Gluing $T_i$, $T_j$, $T_k$ along $a_k$, $a_i$, $a_j$ forms a triangle that contains a simple pole. By the Gauss--Bonnet theorem, the sum of the interior angles of this triangle is $2\pi$. In particular, there exists a corner of angle strictly less than $\pi$. Let $a_k$ be the saddle connection that connects this corner to the simple pole. Then $T_i$ and $T_j$ form a strictly convex quadrilateral which means that at least one of $T_i$ or $T_j$ is a major triangle.
  This contradicts \cref{lem:unique_tri}.
  
  Hence, the cone angle is $2\pi$.
  Gluing $T_i$, $T_j$, $T_k$ along $a_k$, $a_i$, $a_j$ forms a Euclidean triangle that contains a removable singularity. The sides of this triangle are $c_1$, $c_2$, and $c_3$ whereas $a_1$, $a_2$, and $a_3$ connect the corners of the triangle to the removable singularity.
  In particular, no two of these six saddle connections can be parallel.
  \endproof

  We continue first in the situation of the previous lemma, that is when all three gluings are concordant.
  Without loss of generality, we may suppose that $T_i$ has least area among the three triangles.
  Applying an $\SL(2,\R)$--deformation, we may arrange the triangles such that~$a_i$ is horizontal and $c_i$ is vertical.
  The triangles thus appear as in \cref{fig:setting_end_if_sufficiency_second_triangle_test} on the left (up to reflection).
  The least area assumption ensures that $\height(a_i) > \height(a_j) = \height(a_k)$ as~desired.
  
  We now turn to the second case, that is, where there is at least one discordant gluing.
  Without loss of generality, suppose that $T_i$ and $T_j$ meet discordantly along $a_k$.
  As $T_i$ and $T_j$ are consistently oriented and cannot be contained in one another, $c_i$ and $c_j$ must intersect.
  Apply an $\SL(2,\R)$--deformation to make $c_i$ horizontal and $c_j$ vertical.
  Note that $c_i$ and $c_j$ have a unique intersection point, for otherwise
  $T_i \cap c_j$ contains at least two distinct vertical line segments, and so
  $c_j$ would intersect $a_j$ or $a_k$ transversely.
  By \cref{lem:parallel} (and the assumption that $\tau$ does not bound a triangle),
  it follows that none of $a_i, a_j, a_k$ can be horizontal nor vertical.
  Therefore, the triangles $T_i$ and $T_j$ appear as in \cref{fig:wide_opening} (up to horizontal or vertical reflection).
  
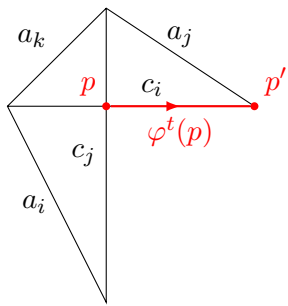
\begin{figure}
 \begin{center}
  \begin{tikzpicture}[scale=0.65]
   \draw (0,0) -- node[above right] {$c_i$} (5,0) -- node[above] {$a_j$} (2,2) -- node[above left] {$a_k$} (0,0)
   -- node[left] {$a_i$} (2,-4) -- node[left] {$c_j$} (2,2);
   \node[shape=circle,fill=red,inner sep=0pt,minimum width=3pt] (A) at (2,0) {};
   \node[shape=circle,fill=red,inner sep=0pt,minimum width=3pt] (B) at (5,0) {};
   \draw[color=red, middlearrow={latex}, thick] (A) node[above left] {$p$} -- node[below] {$\varphi^t(p)$} (B)
   node[above right] {$p'$};
   \end{tikzpicture}
  \caption{The triangles $T_i$ and $T_j$ as in the proof of \cref{lem:wide_open}.  \label{fig:wide_opening}
  The trajectory $\varphi^t(p)$ (in red) runs along $c_i$ from $p$ to $p'$. Note that $p$ is not
  a singularity.}
 \end{center}
\end{figure}

\begin{lem}[Heights and widths] \label{lem:wide_open}
 In the described situation, we have
 $\height(T_i) < \height(a_i)$ and $\width(T_j) < \width(a_j)$.
 Consequently, we have $\angle_{T_j}(a_i,a_k) + \angle_{T_i}(a_k, a_j) > \pi$.
\end{lem}

\proof
  Suppose $\width(a_j) \leq \width(T_j)$ for a contradiction. The saddle connection
  $c_j$ is vertical, and is tallest in the simplex $\tau \cup \{c_j\}$.
  Consider the horizontal unit-speed flow~$\varphi^t$ emanating from~$c_j$, away from $T_j$.
  Let $p$ be the intersection point of $c_i$ and $c_j$.
  The trajectory $\varphi^t(p)$ runs along $c_i$
  until it hits the corner $p'$ of $T_i$ formed by $c_i$ and $a_j$ at time $t = \width(a_j)$;
  see \cref{fig:wide_opening}.
  Since $c_i$ is disjoint from $\tau$, and $p$ is the only intersection point
  between $c_i$ and $c_j$, it follows that~$p'$ is visible with respect to $\varphi^t$ and $\tau \cup \{c_j\}$.
  Applying \cref{prop:vis}, there exists a visible triangle $T$
  with $c_j$ as one of its sides and satisfying
  \[\height(T) = \height(T_j) \quad \textrm{and} \quad \width(T) \leq \width(a_j) \leq \width(T_j).\]
  But then $T_j$ is a major triangle, since it has at least half the area of the 
  strictly convex quadrilateral formed by gluing $T_j$ and $T$ along $c_j$.
  This contradicts \cref{lem:unique_tri}.
  The proof of the other inequality follows in a similar manner.
  
  By considering the three right-angled triangles obtained by cutting $T_i \cup T_j$
  along $c_i$ and~$c_j$ (in the universal cover), together with the inequalities
  established above, we deduce the following.
  \begin{eqnarray*}
   \angle_{T_j}(a_i,a_k) + \angle_{T_i}(a_k, a_j) &=& \angle_{T_j}(a_i,c_i) + \angle_{T_j}(c_i,a_k) + \angle_{T_i}(a_k, c_j) + \angle_{T_i}(c_j, a_j)\\
   &>& 2\left(\angle_{T_j}(c_i,a_k) + \angle_{T_i}(a_k, c_j)\right)
   = 2 \cdot \frac{\pi}{2} = \pi.
  \end{eqnarray*}
  This completes the proof.
\endproof

 We continue in the previously described situation (as in \cref{fig:wide_opening}).
 By applying a horizontal shear to $(S,q)$, we now make $a_i$ vertical and $c_i$ horizontal,
 while maintaining the height of every saddle connection.
 Observe that the inequality $\angle_{T_j}(a_i,a_k) + \angle_{T_i}(a_k, a_j) > \pi$
 persists under $\SL(2,\R)$--deformations.
 This implies that exactly one of $a_j$ or~$a_k$ has positive slope, with the other having negative slope.
 Therefore, the triangles $T_i, T_j, T_k$ appear as in the  configuration in \cref{fig:setting_end_if_sufficiency_second_triangle_test} in the middle or on the right (up to horizontal or vertical reflections),
 depending on whether~$T_j$ and~$T_k$ are glued concordantly or discordantly along $a_i$.
  
 Summarising the discussion, we may thus assume that the triangles $T_i$, $T_j$, $T_k$ appear as one of the three possible configurations in \cref{fig:setting_end_if_sufficiency_second_triangle_test} up to horizontal or vertical reflections.
 For concreteness, let us assume that they appear exactly as shown, so that we may refer to top, bottom, left, and right.
 Observe that $\height(T_j) = \height(T_k) = \height(T_i) + \height(a_i)$.
  
\begin{lem}
 There exists a strictly convex quadrilateral $Q$ with $a_i$ as one of its diagonals, such that $\height(Q) = \height(a_i)$, and neither $a_j$ nor $a_k$ intersect $\partial Q$ transversely.
\end{lem}
  
  \proof
  Consider the leftwards unit-speed flow emanating from $a_i$.
  Since $a_i$ is strictly taller than both $a_j$ and $a_k$, we may construct
  a triangle $T$ meeting $a_i$ that is visible with respect to $\tau$ (see \cref{fig:too_short}).
  Note that $\height(T) = \height(a_i)$.
  Now let $\sigma = \tau \cup \partial T$, and consider the rightwards
  unit-speed flow emanating from $a_i$. Since $a_i$ is also tallest in $\sigma$, we
  can construct a triangle~$T'$ meeting $a_i$ that is visible
  with respect to $\sigma$. We need to show that~$T'$ can be chosen so that
  the quadrilateral $Q$ obtained by gluing $T$ and $T'$ along $a_i$ is
  strictly convex. If no sides of $T$ are horizontal, then this is immediate.
    
  Suppose the side of $T$ meeting the bottom endpoint of $a_i$ is horizontal,
  as in \cref{fig:too_short}.
  By \cref{lem:awning}, the only obstruction for the existence of
  a triangle $T'$ so that $Q$ is strictly convex
  is the presence of an awning for $a_i$. Such an awning
  must be tallest in $\sigma$ and have negative slope.
  But this cannot occur, since the tallest saddle connections in $\sigma$
  are the two non-horizontal sides of $T$, neither of which have negative slope.
  Arguing similarly for the case where a horizontal side of $T$
  meets the top endpoint of $a_i$ completes the proof.
  \endproof

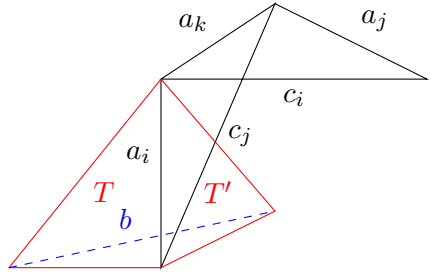
\begin{figure}
 \begin{center}
  \begin{tikzpicture}[scale=0.5]
   \draw (0,0) -- node[above left] {$a_i$} (0,-5) -- node[right] {$c_j$} (3,2) -- node[above left] {$a_k$} (0,0)
    -- node[below] {$c_i$} (7,0) -- node[above right] {$a_j$} (3,2);
   \draw[red]  (0,0) -- (-4,-5) -- (0,-5) -- (3,-3.5) --(0,0);
   \draw[blue, dashed] (-4,-5) -- node[above left] {$b$} (3,-3.5);
   \node[red] (A) at (-1.5,-3) {$T$};
   \node[red] (B) at (1.5,-3) {$T'$};
  \end{tikzpicture}
  \caption{A strictly convex quadrilateral $Q$ (in red), formed by gluing triangles~$T$ and~$T'$ along~$a_i$.
  The diagonals of $Q$ are $a_i$ and $b$ (in blue, dashed).}
  \label{fig:too_short}
 \end{center}
\end{figure}
  
  Let $Q$ be as in the above lemma, and $b$ be its diagonal
  obtained by flipping $a_i$.
  Since $\partial Q \cup \tau$ forms a simplex, it can be extended to a triangulation,
  and so $b\in\Ft(a_i)$.
  By Condition (T5), at least one of $a_j$ or $a_k$ belongs to
  $\KFB(a_i, b)$.
  Recall from \cref{cor:kfb} that $\KFB(a_i, b) \subseteq \partial Q$.
  Therefore, one of the triangles obtained by cutting $Q$ along $a_i$ must also have at least one of~$a_j$ or $a_k$ among its sides.
  By \cref{lem:unique_tri}, this triangle must coincide with $T_j$ or $T_k$.
  But this is impossible, since
  \[\height(T_j) = \height(T_k) = \height(T_i) + \height(a_i) > \height(Q).\]
  Hence we have a contradiction to the assumption that $\tau$ does not bound a triangle.
  
  This completes the proof of the Second Triangle Test.

\section{Orienting triangles}\label{sec:orient}

To recover the gluing pattern of a triangulation, we need
to know not only the triangles themselves but also their orientations on $(S,q)$. 
Recall that an oriented triangle is given by a
triple of sides $\TV = [a,b,c]$, considered up to cyclic permutation.
The main goal of this section is to develop a purely combinatorial test to detect whether two oriented triangles are consistently oriented, that is, when they determine the same orientation on~$S$.

\begin{prop}[Orientation Test] \label{prop:orient}
 Let $\TV$ and $\TV'$ be oriented triangles on $(S,q)$. Then $\TV$ and $\TV'$
 are consistently oriented if and only if there exists a sequence of
 distinct triangles
 \[\TV = \TV_0, \ldots, \TV_k = \TV'\]
 such that for each $0 \leq i < k$,
 the triangles $\TV_i$ and $\TV_{i+1}$ are consistently oriented and contained in a common strictly convex quadrilateral. Moreover, the property that $\TV$ and $\TV'$
 are consistently oriented can be detected purely combinatorially in $\A(S,q)$.
\end{prop}

Certainly, if such a sequence exists then $\TV$ and $\TV'$ are consistently
oriented. Our goal is to prove that there always exists such a sequence
between two consistently oriented triangles.

\begin{lem}[Detecting convex quadrilateral boundaries] \label{lem:quad_boundaries}
 Let $\{c,d\}$ be a flip pair, and $Q(c,d)$ be the strictly convex
 quadrilateral they span.
 Then $\partial Q(c,d)$ can be detected using only the combinatorial
 structure of $\A(S,q)$.
\end{lem}

\proof
 Recall from \cref{sec:flip_pairs_bad_kites} that the set of barriers $\B(c,d)$
 contains $\partial Q(c,d)$, and can be defined purely combinatorially
 in terms of the pair of vertices $\{c,d\}$.
 Let $\T$ be any triangulation containing $\B(c,d) \cup \{c\}$.
 As $\B(c,d) \setminus \partial Q(c,d)$ consists exactly of the cordons of $Q(c,d)$ (see \cref{barriers}),
 we deduce that
 $Q(c,d)$ is the unique
 non-triangular region of $\T \setminus \{c\}$. Applying the Triangle Test,
 we can detect the two triangles $T$ and $T'$ of $\T$ that meet $c$.
 Then $\partial Q(c,d) = (\partial T \cup \partial T') \setminus \{c\}$.
\endproof

 Using the above lemma and the Triangle Test, we can consistently orient
 the four triangles contained in a strictly convex quadrilateral as follows. 
 Let $\{c,d\}$ be a flip pair.
 \begin{enumerate}
   \item Use \cref{lem:quad_boundaries} to detect the sides of $Q(c,d)$.
   \item Observe that two sides of $Q(c,d)$ (which may be the same saddle connection)
  are opposite one another if and only if
    they do not form a triangle together with neither~$c$ nor $d$.
    Use this observation along with the Triangle Test to cyclically order the sides of $Q(c,d)$.
  \item Suppose that $a,b,e,f$ form the sides
 of $Q(c,d)$ in the given cyclic order, and so that $\{a,b,c\}$ forms a triangle.
 Then $[a,b,c]$, $[e,f,c]$, $[b,e,d]$, and $[f,a,d]$ are triangles
 contained in $Q(c,d)$ that all have the same orientation.
  \end{enumerate}

We now define an auxiliary graph $\G(S,q)$ as follows:
The vertices of $\G(S,q)$ are the triangles on $(S,q)$ (in some triangulation), and
two triangles $T, T'$ are connected by an edge if and only if they
are contained in a common strictly convex quadrilateral.
(The triangles $T$ and~$T'$ can be any two of the four triangles in the strictly convex quadrilateral,
in particular, they are allowed to overlap.) 
Adjacent triangles in $\G(S,q)$ can be detected purely combinatorially
using the definition of flip pairs together with the above procedure. Moreover, if $T$ and $T'$ are adjacent then we
can also assign them consistent orientations using only combinatorial data.
Thus, it suffices to show that $\G(S,q)$ is connected in order to  prove \cref{prop:orient}.

Let us first focus on pairs of triangles that can be glued along a
common side to form a quadrilateral.
These triangles can be detected purely combinatorially as follows.

\begin{lem}[Gluable triangles] \label{gluable}
 Let $\tau, \tau' \in \A(S,q)$ be $2$--simplices bounding triangles $T,T'$.
 Then there exists a quadrilateral formed by gluing $T$ and $T'$ along a
 common side if and only if $\tau \cup \tau'$ is a simplex in $\A(S,q)$ and
 $1 \leq \#(\tau \cap \tau') \leq 2$.
\end{lem}

\proof
 Observe that $\tau \cup \tau'$ is a simplex in $\A(S,q)$ if and only if
 no saddle connections of~$\tau$ intersect any saddle connection of
 $\tau'$. The condition $1 \leq \#(\tau \cap \tau') \leq 2$ is equivalent
 to saying that $T$ and $T'$ are distinct triangles that meet along 
 at least one common side.
\endproof

\begin{lem}[Connectedness of $\G(S,q)$]
 The graph $\G(S, q)$ is connected.
\end{lem}

\proof
 Suppose first that $\T$ is a triangulation of $(S,q)$, and suppose $T, T'$ are triangles 
 of $\T$ that share a common side. We show that there exists a path in $\G(S,q)$
 connecting $T$ to $T'$.
 Let~$c \in \partial T$ be the common side.
 By \cref{prop:extension}, there exists a triangulation
 $\T' \supseteq \partial T$ in which~$c$ is flippable.
 Thus, the two triangles $T, T''$ that meet $c$ in~$\T'$
 can be glued along $c$ to form a strictly convex quadrilateral.
 In particular, $T$ and~$T''$ are adjacent in $\G(S,q)$.
 
 By \cref{thm:tahar}, $\F_\tau(S,q)$ is connected, so there exists a sequence
 of triangulations
 \[\T = \T_0, \ldots, \T_j = \T' \]
 each containing $\tau$, where consecutive triangulations are related by
 a single flip. Let $T_i \neq T$ be the triangle of $\T_i$
 with $c$ as one of its sides. Note that $T_0 = T'$ and $T_j = T''$.
 If $T_i \neq T_{i+1}$ then $\T_i$ and $\T_{i+1}$ are related by a flip in a strictly convex
 quadrilateral that contains both $T_i$ and $T_{i+1}$.
 Thus, for each $0 \leq i < j$, the triangles $T_i$ and $T_{i+1}$
 either coincide or are adjacent in $\G(S,q)$. Therefore, there exists a path in
 $\G(S,q)$ connecting $T'$ to $T''$, and hence to $T$.
 
 Consequently, by connectedness of $(S,q)$, there exists a path in $\G(S,q)$
 connecting any two triangles of a given triangulation. If two triangulations
 $\T, \T'$ differ by a single flip then they contain at least one triangle
 in common (since we are assuming that $S$ is not a flat torus with exactly one marked point and hence contains at least three triangles in every triangulation).
 By \cref{thm:tahar}, $\F(S,q)$ is connected and so it follows that $\G(S,q)$ is connected.
\endproof

Note that all the steps and objects used in the previous proof can 
be stated using only the combinatorial
structure of $\A(S,q)$.
This completes the proof of \cref{prop:orient}.

\section{Rigidity}\label{sec:rigid}

We are now ready to prove our main theorem.
Let $(S,q)$ and $(S',q')$ be half-translation surfaces.
By the Triangle Test (\cref{tri_test}) and 
Orientation Test (\cref{prop:orient}), the combinatorial
structure of $\A(S,q)$ can be used to recover
the gluing pattern of any triangulation $\T$ on $(S,q)$, and
hence the underlying topological surface $(S,\P)$.
We can thus regard $\A(S,q)$ as a subcomplex of $\A(S,\P)$.
Similarly, $\A(S',q')$ can be regarded as a subcomplex of $\A(S', \P')$,
where $(S', \P')$ is the underlying topological surface
of $(S',q')$.

Given a homeomorphism $F \colon (S,\P) \rightarrow (S', \P')$,
write $F^\# \colon \A(S,\P) \to \A(S',\P')$ for the induced map
on the arc complexes. If $F$ is isotopic to an affine diffeomorphism
from $(S,q)$ to~$(S',q')$, then the restriction 
$F^\#|_{\A(S,q)} \colon \A(S,q) \to \A(S',q')$
is a simplicial isomorphism.
Our goal is to prove the converse.

\begin{restate}{thm:main}[Rigidity of the saddle connection complex]
 Let $(S,q)$ and $(S',q')$ be half-translation surfaces, neither of which are flat tori with exactly one removable singularity. Suppose $\phi : \A(S,q) \to \A(S',q')$ is a simplicial isomorphism.
 Then there exists a unique affine diffeomorphism $F: (S,q) \to (S',q')$ inducing $\phi$.
\end{restate}

Recall that in the case where at least one of $(S,q)$ or $(S',q')$ is a flat torus with one removable singularity, we have exactly two such affine diffeomorphisms as described after \cref{exa:scc_flat_torus}.
As we are not in this case, we can use the fact that triangles are uniquely determined by their sides in the following arguments.

Our strategy is to first define affine maps on individual triangles.
These maps can be used to define a piecewise affine diffeomorphisms on
$(S,q)$ associated to a given triangulation. 
We then show that all triangulations
give rise to the same piecewise affine diffeomorphism $F \colon (S,q) \to (S',q')$
yielding the map $\phi$ on the saddle connection complexes.
Finally, we use the Cylinder Rigidity Theorem (see \cref{thm:cyl_rigid}) to show that $F$ is affine.

The Triangle Test (\cref{tri_test}) uses only the combinatorial
structure of $\A(S,q)$.
Therefore, a $2$--simplex $\tau\in\A(S,q)$ bounds a triangle
on $(S,q)$ if and only if $\phi(\tau) \in \A(S',q')$ bounds a triangle
on $(S',q')$.
If $T$ is a triangle on $(S,q)$ with sides $\tau = \{a,b,c\}$ then
there is a unique affine map $F_T \colon T \to (S',q')$ such that
$F_T(a) = \phi(a)$, $F_T(b) = \phi(b)$, and $F_T(c) = \phi(c)$.
In particular, the image of $F_T$ is the unique triangle with sides
$\phi(\tau)$.

Suppose $\TV = [a,b,c]$ and $\TV' = [a',b',c']$
are oriented triangles on $(S,q)$. 
Since the Orientation Test (\cref{prop:orient})
only uses the combinatorial structure of $\A(S,q)$, it follows that
the oriented triangles $[\phi(a), \phi(b), \phi(c)]$ and 
$[\phi(a'), \phi(b'), \phi(c')]$ on $(S',q')$ are consistently oriented
if and only if $\TV, \TV'$ are consistently oriented.

We now want to define a piecewise affine diffeomorphism ${F_\T \colon (S,q) \to (S',q')}$ for a given triangulation $\T$ of $(S,q)$.

\begin{lem}[Candidate homeomorphisms] \label{tri_homeo}
 Let $\T$ be a triangulation of $(S,q)$ and define ${F_\T \colon (S,q) \to (S',q')}$ by declaring $F_\T|_T = \ F_T$ for every triangle $T$ of $\T$.
 Then $F_\T$ is a well-defined, piecewise affine diffeomorphism.
\end{lem}

\proof
Observe that $F_\T$ is well-defined on the interior of each triangle $T$ of $\T$.
We need to check that it is well-defined on the edges and vertices of $\T$.
Suppose $T,T'$ are distinct triangles of $\T$.
Since $\phi$ preserves disjointness of saddle connections,
$\partial F_\T(T)$ and $\partial F_\T(T')$ have no transverse intersections.
It follows that $F_\T(T), F_\T(T')$ have disjoint interiors.
By \cref{prop:orient} and \cref{gluable},
if $T, T'$ meet along a common side $a \in \T$,
then $F_\T(T), F_\T(T')$ meet along $\phi(a)$ with the same orientation.
Thus, $F_\T$ is well-defined on the complement of the singularities~$\P$ of~$(S,q)$.

Next, we check that $F_\T$ is well-defined on $\P$.
Note that $S$ and $S'$ are respectively the metric completions
of $S \setminus \P$ and  $S' \setminus \P'$.
Observe that each $F_T$ is a Lipschitz map.
Since $\T$ has finitely many triangles, $F_\T$ is also a Lipschitz map.
As Lipschitz maps send Cauchy sequences to Cauchy sequences, it follows that
$F_\T$ extends to a unique map between the respective metric completions.
It follows that $F_\T \colon (S,q) \to (S',q')$ is well-defined.

Finally, we show that $F_\T$ is a homeomorphism.
By construction, $F_\T$ is continuous and restricts to an affine map on each
triangle $T$ of~$\T$.
We can analogously define a piecewise affine map
$G_{\phi(\T)} \colon (S', q') \to (S,q)$ using $\phi^{-1}$ and the triangulation
$\phi(\T)$ on $(S',q')$.
By construction, $G_{\phi(\T)} \circ F_\T$ restricts to the identity map on each triangle~$T$ of~$\T$. Therefore, $F_\T$ and $G_{\phi(\T)}$ are inverses of one another.
It follows that $F_\T$ is a piecewise affine diffeomorphism.
\endproof

Observe that $F_\T$ acts as a bijection between the singularities
of $(S,q)$ and $(S',q')$. In fact, this bijection does not depend on
the choice of the triangulation $\T$.

As $F_\T$ is in particular a homeomorphism from $(S,\P)$ to $(S',\P')$, it induces a simplicial isomorphism $F_\T^\# \colon \A(S,\P) \to \A(S',\P')$.
We show first that $F_\T^\#$ is independent of the triangulation and then that it coincides with $\phi$ on the proper isotopy classes of arcs in $\A(S,\P)$ realisable as saddle connections on $(S,q)$.

\begin{lem}[Candidates are isotopic]
 For two triangulations $\T, \T'$ of $(S,q)$, we have $F_\T^\# = F_{\T'}^\#$.
\end{lem}

\proof
 By \cref{thm:tahar}, $\F(S,q)$ is connected, and so there is a finite
 sequence of flips turning~$\T$ into~$\T'$. So we may assume without loss of generality
 that $\T$ and $\T'$ differ by precisely one flip. Then $\T \cap \T'$ is a triangulation
 away from a strictly convex quadrilateral~$Q$ on $(S,q)$.
 Then $F_\T$ and~$F_{\T'}$ coincide outside the interior of $Q$.
 By Alexander's Trick, any two homeomorphisms defined on a disc that agree on the boundary
 must be isotopic. Therefore, $F_\T$ and $F_{\T'}$ are~isotopic, through isotopies fixing the singularities. Hence, $F_\T^\#(a) = F_{\T'}^\#(a)$ for every arc $a \in \A(S,\P)$.
\endproof

\begin{lem}[Candidates induce the correct isomorphism on $\A(S,q)$]
 Let $\T$ be a triangulation of $(S,q)$. Then  the restricted simplicial map $F_\T^\#|_{\A(S,q)}$ coincides with $\phi$.

\begin{proof}
 Let $a \in \A(S,q)$ be a saddle connection and let $\T'$ be a triangulation with $a \in \T'$.
 By definition, $F_{\T'}(a)$ is the unique geodesic representative of $\phi(a)$ in its isotopy class.
 But this implies $F_\T^\#(a) = F_{\T'}^\#(a) = \phi(a)$.
\end{proof}
\end{lem}

It remains to show that $F_\T$ is an affine diffeomorphism.
For this, let us fix a choice of triangulation $\T$ and write $F = F_\T$.
Let $m = F^*(q')$ be the half-translation structure on $S$ obtained 
by pulling back $q'$ via $F$. Note that the singularities of $(S,m)$
are naturally identified with $\P$.
Now, consider the map $F' \coloneqq F \circ \id_S \colon (S,m) \to (S',q')$ 
where $\id_S \colon (S,m) \to (S,q)$ is the identity map on the underlying surface $(S,\P)$. 
By definition of $m$, $F'$ maps the charts corresponding to the half-translation structure $m$ to the charts corresponding to $q'$ and hence~$F'$ is an isometry.
Therefore, a topological arc $\alpha\in\A(S,\P)$ is realisable as a 
saddle connection on~$(S,m)$ if and only if $F'(\alpha)\in\A(S',\P')$ is realisable
as a saddle connection on $(S',q')$. By the above corollary, this occurs
precisely when $F^{-1}\circ F'(\alpha) = \id_S(\alpha) = \alpha$ is
realisable as a saddle connection on~$(S,q)$.
We have thus shown the following.

\begin{lem}[Subcomplexes coincide]
 The saddle connection complexes $\A(S,q)$ and $\A(S,m)$ coincide as subcomplexes of 
 the arc complex $\A(S,\P)$. $\hfill\square$
\end{lem}

It follows that $\MIL(\A(S,q)) = \MIL(\A(S,m))$ as sets of simplices in $\A(S,\P)$.
Therefore, by \cref{MIL_cyl}, the sets of cylinder curves $\cyl(q)$ and $\cyl(m)$
coincide as sets of simple closed curves on $S$.
Applying the Cylinder Rigidity Theorem (\cref{thm:cyl_rigid}), we deduce that $m \in \GL(2,\R)\cdot q$.
Therefore $\id_S$ is an affine diffeomorphism.
Since $F'$ is an isometry, it follows that
$F = F' \circ \id_S^{-1}$ is also an affine diffeomorphism as desired.

Finally, we check that $F$ is the unique affine diffeomorphism with $F^\#|_{\A(S,q)} = \phi$.
Suppose $G \colon (S,q) \to (S',q')$ is an affine diffeomorphism
inducing the isomorphism $\phi \colon \A(S,q) \to \A(S',q')$. Then $G$ maps a triangle
$T$ on $(S,q)$ with sides $a,b,c$ to the triangle $G(T)$ on $(S',q')$ with sides
$\phi(a) = G(a)$, $\phi(b) = G(b)$, and $\phi(c) = G(c)$.
Since $F_T \colon T \rightarrow G(T)$ is the unique affine map from~$T$ to $G(T)$
which behaves correctly on the sides of $T$, it follows that $G|_T = F_T$.
As this holds for all triangles on $(S,q)$, we deduce that $G = F_\T$ for
every triangulation $\T$ of $(S,q)$. (In particular, $F_\T$ does not
depend on the choice of triangulation.) Therefore, $F = G$ and we are~done.

This completes the proof of \cref{thm:main}.

\clearpage
\appendix

\section{Classifying links of simplices of codimension \texorpdfstring{$1$ or $2$}{1 or 2}} \label{app:classification}

In the tables below, we classify the possible link types of a simplex $\sigma\in\A(S,q)$ of codimension~$1$ or $2$, together with the corresponding non-triangular regions of $S - \sigma$.
In the cases where $\lk(\sigma)$ decomposes as a non-trivial join, we indicate the factors by colouring the vertices red or blue.

\begin{center}
\begin{longtable}{>{\centering\arraybackslash} m{0.27\textwidth} >{\centering\arraybackslash} m{0.67\linewidth}} $\lk(\sigma)$ & Non-triangular regions of $S-\sigma$ \\ \hline
  \\
  \begin{tikzpicture}[scale=0.8]
   \draw[fill, color=red] (0,0) circle (3pt);
   \draw[fill, color=red] (1,0) circle (3pt);
  \end{tikzpicture}
  &
  \begin{tikzpicture}[scale=0.7, baseline=(current bounding box.center)]
   \draw (0,0) -- (1,0.8) -- (-0.2,2) -- (-2.1,1.1) -- (0,0);
   \draw[color=red] (0,0) -- (-0.2,2);
   \draw[color=red] (1,0.8) -- (-2.1,1.1);
  \end{tikzpicture}
  \\[10mm]
  \begin{tikzpicture}[scale=0.8]
   \draw[fill, color=red] (0,0) circle (3pt);
  \end{tikzpicture}
  &
  \begin{tikzpicture}[scale=0.8, baseline=(current bounding box.center)]
    \draw (0,0) -- (1,0.6) -- (2.6,0.2) -- (1.5,1.8) -- (0,0);
    \draw[color=red] (1,0.6) -- (1.5,1.8);
  \end{tikzpicture}\\
   \caption{Classification of links of codimension--1 simplices.}
   \label{tab:codim1}
 \end{longtable}
\end{center}

\vspace{-35pt}

\begin{center}
\begin{longtable}{>{\centering\arraybackslash} m{0.27\textwidth} >{\centering\arraybackslash} m{0.67\linewidth}}
  $\lk(\sigma)$ & Non-triangular regions of $S-\sigma$ \\ \hline
  \\
  \begin{tikzpicture}[scale=0.8]
   \draw (-0.3,0) -- (0,0) -- (1,0) -- (2,0) -- (3,0) -- (4,0) -- (4.3,0);
   \draw[dotted] (-0.6,0) -- (-0.3,0);
   \draw[dotted] (4.3,0) -- (4.6,0);
   \draw[fill, color=red] (0,0) circle (3pt);
   \draw[fill, color=red] (1,0) circle (3pt);
   \draw[fill, color=red] (2,0) circle (3pt);
   \draw[fill, color=red] (3,0) circle (3pt);
   \draw[fill, color=red] (4,0) circle (3pt);
   \draw (2,-0.5) node{(bi-infinite path graph)};
  \end{tikzpicture}
  &
  \begin{tikzpicture}[rotate=10, xscale=0.75, yscale=0.5]
   \draw (0,0) -- (4,0) to[in=0, out=0, looseness=0.5] (4,3) -- (0,3) to[in=0, out=0, looseness=0.5] (0,0);
   \draw (0,3) to[in=180, out=180, looseness=0.5] (0,0);
   \draw[dashed] (4,0) to[in=180, out=180, looseness=0.5] (4,3);
   \draw[color=red] (0.42,1) -- (4.42,1);
   \draw[color=red] (0.42,1) to[out=-45, in=180, looseness=0.4] (1.4,0);
   \draw[color=red, dashed] (1.4,0) to[out=0, in=180, looseness=0.4] (3.4,3);
   \draw[color=red] (3.4,3) to[out=0, in=135, looseness=0.4] (4.42, 1);
   \draw[color=red] (0.42,1) to[out=-70, in=180, looseness=0.3] (1,0);
   \draw[color=red, dashed] (1,0) to[out=0, in=180, looseness=0.3] (1.8,3);
   \draw[color=red] (1.8,3) to[out=0, in=180, looseness=0.3] (3,0);
   \draw[color=red, dashed] (3,0) to[out=0, in=180, looseness=0.3] (3.7,3);
   \draw[color=red] (3.7,3) to[out=0, in=110, looseness=0.3] (4.42, 1);   
  \end{tikzpicture}
  \\[5mm]
  \begin{tikzpicture}[scale=0.8]
   \draw (-90:1) -- (-18:1) -- (54:1) -- (126:1) -- (198:1) -- (-90:1);  
   \draw[fill, color=red] (-90:1) circle (3pt);
   \draw[fill, color=red] (-18:1) circle (3pt);
   \draw[fill, color=red] ( 54:1) circle (3pt);
   \draw[fill, color=red] (126:1) circle (3pt);
   \draw[fill, color=red] (198:1) circle (3pt);
  \end{tikzpicture}
  &
  \begin{tikzpicture}[scale=0.6]
   \draw (0,0) -- (3,0) -- (4,2) -- (2,2.3) -- (-1,1.2) -- (0,0);
   \draw[color=red] (0,0) -- (2,2.3);
   \draw[color=red] (3,0) -- (2,2.3);
   \draw[color=red] (0,0) -- (4,2);
   \draw[color=red] (3,0) -- (-1,1.2);
   \draw[color=red] (4,2) -- (-1,1.2);
  \end{tikzpicture}
  \\[5mm]
  \begin{tikzpicture}[scale=0.8]
   \draw (0,0) -- (1,0) -- (2,0) -- (3,0);  
   \draw[fill, color=red] (0,0) circle (3pt);
   \draw[fill, color=red] (1,0) circle (3pt);
   \draw[fill, color=red] (2,0) circle (3pt);
   \draw[fill, color=red] (3,0) circle (3pt);
  \end{tikzpicture}
  &
  \begin{tikzpicture}[scale=0.6]
   \draw (0,0) -- (3,0) -- (4,2) -- (2,1.3) -- (-1,1.2) -- (0,0);
   \draw[color=red] (0,0) -- (2,1.3);
   \draw[color=red] (3,0) -- (2,1.3);
   \draw[color=red] (0,0) -- (4,2);
   \draw[color=red] (3,0) -- (-1,1.2);
  \end{tikzpicture}
  \\[5mm]
  \begin{tikzpicture}[scale=0.8]
   \draw (45:1) -- (135:1) -- (-135:1) -- (-45:1) -- (45:1);  
   \draw[fill, color=blue] (45:1) circle (3pt);
   \draw[fill, color=red] (135:1) circle (3pt);
   \draw[fill, color=blue] (-135:1) circle (3pt);
   \draw[fill, color=red] (-45:1) circle (3pt);
  \end{tikzpicture}
  &
  \begin{tikzpicture}[scale=0.6]
   \draw (0,0) -- (1,0.8) -- (-0.2,2) -- (-2.1,1.1) -- (0,0);
   \draw[color=red] (0,0) -- (-0.2,2);
   \draw[color=red] (1,0.8) -- (-2.1,1.1);
   
   \begin{scope}[xshift=3cm, scale=1.0]
    \draw (0,0) -- (1,1.6) -- (-0.2,2) -- (-1.5,1.1) -- (0,0);
    \draw[color=blue] (0,0) -- (-0.2,2);
    \draw[color=blue] (1,1.6) -- (-1.5,1.1);
   \end{scope}
  \end{tikzpicture}
  \\[5mm]
  \begin{tikzpicture}[scale=0.8]
   \draw (0,0) -- (1,0) -- (2,0);
   \draw[fill, color=red] (0,0) circle (3pt);
   \draw[fill, color=blue] (1,0) circle (3pt);
   \draw[fill, color=red] (2,0) circle (3pt);
  \end{tikzpicture}
  &
  \begin{tikzpicture}[scale=0.6, baseline=(current bounding box.center)]
   \draw (0,0) -- (3,0) -- (3.3,2) -- (2,0.5) -- (-1,1.1) -- (0,0);
   \draw[color=blue] (3,0) -- (2,0.5);
   \draw[color=red] (0,0) -- (2,0.5);
   \draw[color=red] (3,0) -- (-1,1.1);
  \end{tikzpicture}
  \quad or \quad
  \begin{tikzpicture}[scale=0.7, baseline=(current bounding box.center)]
   \draw (0,0) -- (1,0.8) -- (-0.2,2) -- (-2.1,1.1) -- (0,0);
   \draw[color=red] (0,0) -- (-0.2,2);
   \draw[color=red] (1,0.8) -- (-2.1,1.1);
   
   \begin{scope}[xshift=1cm]
    \draw (0,0) -- (1,0.6) -- (2.6,0.2) -- (1.2,2) -- (0,0);
    \draw[color=blue] (1,0.6) -- (1.2,2);
   \end{scope}
  \end{tikzpicture}
  \\[5mm]
  \begin{tikzpicture}[scale=0.8]
   \draw (0,0) -- (1,0);
   \draw[fill, color=red] (0,0) circle (3pt);
   \draw[fill, color=blue] (1,0) circle (3pt);
  \end{tikzpicture}
  &
  \adjustbox{margin=10pt}{\begin{tikzpicture}[scale=0.6, baseline=(current bounding box.center)]
   \draw (0,0) -- (3,0) -- (3.3,2) -- (1,0.5) -- (-1.3,1.4) -- (0,0);
   \draw[color=red] (3,0) -- (1,0.5);
   \draw[color=blue] (0,0) -- (1,0.5);
  \end{tikzpicture}
  \quad
  or\
  \begin{tikzpicture}[scale=0.7, baseline=(current bounding box.center)]
    \draw (0,0) -- (1,0.6) -- (2.6,0.2) -- (1.5,1.8) -- (0,0);
    \draw[color=red] (1,0.6) -- (1.5,1.8);
   
   \begin{scope}[xshift=2.8cm]
    \draw (0,0) -- (1,0.6) -- (2.6,0.2) -- (1.2,2) -- (0,0);
    \draw[color=blue] (1,0.6) -- (1.2,2);
   \end{scope}
  \end{tikzpicture}
  or\
  \begin{tikzpicture}[scale=0.7, baseline=(current bounding box.center)]
    \draw (0,0) to[bend left=25] (1.5,2) to[bend left=40] (0,0);
    \draw[color=blue] (0,0) -- (1,1.2);
    \draw[color=red] (1,1.2) -- (1.5,2);
  \end{tikzpicture}
  }
  \\
   \caption{Classification of links of codimension--2 simplices.}
   \label{tab:codim2}
 \end{longtable}
\end{center}

\section{List of notation}\label{app:notation}

Let $\K$ be a simplicial complex. The following sets of simplices in $\K$ are used in this paper.

\begin{itemize}
 \item $\IL(\K)$: simplices $\sigma\in\K$ with infinite link (\cref{defin:IL})
 \item $\MIL(\K)$: simplices $\sigma \in \IL(\K)$ that are maximal among those in $\IL(\K)$ (\cref{defin:IL})
 \item $\FP(\K)$: flip pairs $\kappa \cong \NG_2$ arising as links of codimension--1 simplices in $\K$ (\cref{def:fp})
 \item $\B(\kappa)$: barriers of a flip pair $\kappa$ (\cref{def:barriers})
 \item $\FB(\kappa)$: flippable barriers of a flip pair $\kappa$ (\cref{def:flip_bar})
 \item $\KB(\kappa)$: kite barriers of a cylindrical kite pair $\kappa$ (\cref{def:kb})
 \item $\KFB(\kappa)$: kite-or-flippable barriers of a flip pair $\kappa$ (\cref{def:kfb})
 \item $\Ft(a)$: $\tau$-compatible flip partners of a saddle connection $a$ (\cref{def:cfp})
\end{itemize}

The following notation for graphs is used in the paper.
\begin{itemize}
\item $\CG_k$: the cycle graph of length $k$
\item $\NG_k$: the edgeless graph on $k$ vertices
\item $\PG_k$: the path graph of length $k$
\item $\PG_\infty$: the bi-infinite path graph
\end{itemize}

\bibliographystyle{amsalpha}
\bibliography{literature_saddle_connection_graph}

\providecommand{\bysame}{\leavevmode\hbox to3em{\hrulefill}\thinspace}
\providecommand{\MR}{\relax\ifhmode\unskip\space\fi MR }
\providecommand{\MRhref}[2]{%
  \href{http://www.ams.org/mathscinet-getitem?mr=#1}{#2}
}
\providecommand{\href}[2]{#2}
\begin{thebibliography}{DPRT21}

\bibitem[AKP15]{aramayona_koberda_parlier_15}
Javier Aramayona, Thomas Koberda, and Hugo Parlier, \emph{Injective maps
  between flip graphs}, Annales de l'Institut Fourier \textbf{65} (2015),
  no.~5, 2037--2055.

\bibitem[BDT19]{bell_disarlo_tang_18}
Mark~C. Bell, Valentina Disarlo, and Robert Tang, \emph{Cubical geometry in the
  polygonalisation complex}, Mathematical Proceedings of the Cambridge
  Philosophical Society \textbf{167} (2019), no.~1, 1--22.

\bibitem[BM19]{brendle_margalit}
Tara Brendle and Dan Margalit, \emph{{Normal subgroups of mapping class groups
  and the metaconjecture of Ivanov}}, Journal of the American Mathematical
  Society \textbf{32} (2019), no.~4, 1009--1070.

\bibitem[DELS21]{duchin_erlandsson_leininger_sadanand_21}
Moon {Duchin}, Viveka {Erlandsson}, Christopher~J. {Leininger}, and Chandrika
  {Sadanand}, \emph{{You can hear the shape of a billiard table: symbolic
  dynamics and rigidity for flat surfaces}}, {Commentarii Mathematici
  Helvetici} \textbf{96} (2021), no.~3, 421--463.

\bibitem[Dis15]{disarlo_15}
Valentina Disarlo, \emph{{Combinatorial rigidity of arc complexes}},
  \href{https://arxiv.org/abs/1505.08080}{arXiv:1505.08080}, 2015.

\bibitem[DLR10]{duchin_leininger_rafi_10}
Moon Duchin, Christopher~J.\ Leininger, and Kasra Rafi, \emph{Length spectra
  and degeneration of flat metrics}, Inventiones Mathematicae \textbf{182}
  (2010), no.~2, 231--277.

\bibitem[DP19]{disarlo_parlier_14}
Valentina Disarlo and Hugo Parlier, \emph{{The geometry of flip graphs and
  mapping class groups}}, {Transactions of the American Mathematical Society}
  \textbf{372} (2019), no.~6, 3809--3844.

\bibitem[DPRT21]{disarlo_pan_randecker_tang_21}
Valentina Disarlo, Huiping Pan, Anja Randecker, and Robert Tang,
  \emph{Large-scale geometry of the saddle connection graph}, Trans. Amer.
  Math. Soc. \textbf{374} (2021), no.~11, 8101--8129.

\bibitem[FLP79]{fathi_laudenbach_poenaru_79}
Albert Fathi, Fran{\c{c}}ois Laudenbach, and Valentin Po{\'e}naru,
  \emph{Travaux de {T}hurston sur les surfaces}, Ast\'{e}risque, vol.~66,
  Soci\'{e}t\'{e} Math\'{e}matique de France, Paris, 1979, S\'{e}minaire Orsay,
  With an English summary.

\bibitem[Har81]{harvey_1978}
William~J. Harvey, \emph{Boundary structure of the modular group}, Riemann
  surfaces and related topics: {P}roceedings of the 1978 {S}tony {B}rook
  {C}onference, Annals of Mathematics Studies, vol.~97, Princeton University
  Press, 1981, pp.~245--251.

\bibitem[Har85]{harer_annals}
John~L. Harer, \emph{Stability of the homology of the mapping class groups of
  orientable surfaces}, Annals of Mathematics \textbf{121} (1985), no.~2,
  215--249.

\bibitem[Har86]{harer_inventiones}
\bysame, \emph{The virtual cohomological dimension of the mapping class group
  of an orientable surface}, Inventiones mathematicae \textbf{84} (1986),
  157--176.

\bibitem[HPW15]{hensel_webb_15}
Sebastian Hensel, Piotr Przytycki, and Richard C.~H. Webb, \emph{1-slim
  triangles and uniform hyperbolicity for arc graphs and curve graphs}, Journal
  of the European Mathematical Society \textbf{17} (2015), no.~4, 755--762.

\bibitem[HT80]{hatcher_thurston1980}
Allen Hatcher and William~P. Thurston, \emph{A presentation for the mapping
  class group of a closed orientable surface}, Topology \textbf{19} (1980),
  no.~3, 221--237.

\bibitem[IK07]{irmak_korkmaz}
Elmas Irmak and Mustafa Korkmaz, \emph{Automorphisms of the
  {H}atcher-{T}hurston complex}, Israel Journal of Mathematics \textbf{162}
  (2007), 183--196.

\bibitem[IM10]{irmak_mccarthy10}
Elmas Irmak and John~D. McCarthy, \emph{Injective simplicial maps of the arc
  complex}, Turkish Journal of Mathematics \textbf{34} (2010), no.~3, 339--354.

\bibitem[Iva97]{ivanov_97}
Nikolai~V. Ivanov, \emph{Automorphism of complexes of curves and of
  {T}eichm\"{u}ller spaces}, International Mathematics Research Notices (1997),
  no.~14, 651--666.

\bibitem[Iva02]{ivanov_02}
\bysame, \emph{Mapping class groups}, Handbook of geometric topology,
  North-Holland, Amsterdam, 2002, pp.~523--633.

\bibitem[Iva06]{ivanov06}
\bysame, \emph{Fifteen problems about the mapping class groups}, Problems on
  mapping class groups and related topics, Proceedings of Symposia in Pure
  Mathematics, vol.~74, American Mathematical Society, 2006, pp.~71--80.

\bibitem[Kob88]{kobayashi_88}
Tsuyoshi Kobayashi, \emph{Heights of simple loops and pseudo-{A}nosov
  homeomorphisms}, Braids ({S}anta {C}ruz, {CA}, 1986), Contemporary
  Mathematics, vol.~78, American Mathematical Society, 1988, pp.~327--338.

\bibitem[Kor99]{korkmaz}
Mustafa Korkmaz, \emph{Automorphisms of complexes of curves on punctured
  spheres and on punctured tori}, Topology and its Applications \textbf{95}
  (1999), no.~2, 85--111.

\bibitem[KP12]{korkmaz_papadopoulos}
Mustafa Korkmaz and Athanase Papadopoulos, \emph{On the ideal triangulation
  graph of a punctured surface}, Annales de l'Institut Fourier \textbf{62}
  (2012), no.~4, 1367--1382.

\bibitem[Luo00]{Luo}
Feng Luo, \emph{Automorphisms of the complex of curves}, Topology \textbf{39}
  (2000), no.~2, 283--298.

\bibitem[Mar04]{margalit}
Dan Margalit, \emph{Automorphisms of the pants complex}, Duke Mathematical
  Journal \textbf{121} (2004), no.~3, 457--479.

\bibitem[Mas86]{masur_86}
Howard Masur, \emph{Closed trajectories for quadratic differentials with an
  application to billiards}, Duke Mathematical Journal \textbf{53} (1986),
  no.~2, 307--314.

\bibitem[MM99]{masur_minsky_99}
Howard Masur and Yair~N. Minsky, \emph{Geometry of the complex of curves. {I}.
  {H}yperbolicity}, Inventiones Mathematicae \textbf{138} (1999), no.~1,
  103--149.

\bibitem[MP12]{pap_mccarthy}
John~D. McCarthy and Athanase Papadopoulos, \emph{Simplicial actions of mapping
  class groups}, Handbook of {T}eichm\"{u}ller theory. {V}olume {III}, European
  Mathematical Society, 2012, pp.~297--423.

\bibitem[MS13]{masur_schleimer_13}
Howard Masur and Saul Schleimer, \emph{The geometry of the disk complex},
  Journal of the American Mathematical Society \textbf{26} (2013), no.~1,
  1--62.

\bibitem[MT17]{minsky_taylor_17}
Yair~N. Minsky and Samuel~J. Taylor, \emph{Fibered faces, veering
  triangulations, and the arc complex}, Geometric and Functional Analysis
  \textbf{27} (2017), no.~6, 1450--1496.

\bibitem[Ngu17]{nguyen_15}
Duc-Manh Nguyen, \emph{Translation surfaces and the curve graph in genus two},
  Algebraic \& Geometric Topology \textbf{17} (2017), no.~4, 2177--2237.

\bibitem[Ngu18]{nguyen_18}
\bysame, \emph{{Topological Veech dichotomy and tessellations of the hyperbolic
  plane}}, \href{https://arxiv.org/abs/1808.09329}{arXiv:1808.09329}, 2018.

\bibitem[{Pan}22]{pan_22}
Huiping {Pan}, \emph{{Affine equivalence and saddle connection graphs of
  half-translation surfaces}}, {International Mathematics Research Notices}
  \textbf{2022} (2022), no.~4, 2861--2905.

\bibitem[Roy71]{royden_71}
Halsey~L. Royden, \emph{Automorphisms and isometries of {T}eichm\"{u}ller
  space}, Advances in the {T}heory of {R}iemann {S}urfaces ({P}roc. {C}onf.,
  {S}tony {B}rook, {N}.{Y}., 1969), Ann. of Math. Studies, No. 66. Princeton
  Univ. Press, Princeton, N.J., 1971, pp.~369--383.

\bibitem[Str84]{strebel_84}
Kurt Strebel, \emph{Quadratic differentials}, Ergebnisse der Mathematik und
  ihrer Grenzgebiete; Folge 3, Bd.\ 5, Springer, Berlin, 1984.

\bibitem[Tah19]{tahar_17}
Guillaume Tahar, \emph{Geometric triangulations and flips}, Comptes Rendus
  Math\'{e}matique. Acad\'{e}mie des Sciences. Paris \textbf{357} (2019),
  no.~7, 620--623.

\bibitem[TW18]{tang_webb_18}
Robert Tang and Richard C.~H. Webb, \emph{Shadows of {T}eichm\"{u}ller discs in
  the curve graph}, International Mathematics Research Notices \textbf{2018}
  (2018), no.~11, 3301--3341.

\bibitem[Vor03]{vorobets_03}
Yaroslav Vorobets, \emph{Periodic geodesics on translation surfaces},
  \href{https://arxiv.org/abs/math/0307249}{arXiv:math/0307249}, 2003.

\bibitem[{Wri}15]{wright_15}
Alex {Wright}, \emph{{Translation surfaces and their orbit closures: an
  introduction for a broad audience}}, {EMS Surveys in Mathematical Sciences}
  \textbf{2} (2015), no.~1, 63--108.

\bibitem[Zor06]{zorich_06}
Anton Zorich, \emph{Flat surfaces}, {Frontiers in number theory, physics, and
  geometry I. On random matrices, zeta functions, and dynamical systems. Papers
  from the meeting, Les Houches, France, March 9--21, 2003}, Springer, Berlin,
  2nd printing ed., 2006, pp.~437--583.

\end{thebibliography}

\end{document}